\documentclass{amsart}
\usepackage{bbm}
\usepackage{color}
\usepackage{amscd}
\usepackage{nicefrac}
\usepackage{graphicx}
\usepackage{hyperref}
\usepackage{amssymb}
\usepackage{euscript}
\usepackage{enumitem}
\usepackage[cmtip,all]{xy}
\usepackage[export]{adjustbox}

\numberwithin{equation}{section}
\renewcommand{\subsection}[1]{\hspace{-\parindent}\refstepcounter{subsection}{\bf (\thesubsection) #1.}}
\renewcommand{\thesubsection}{\arabic{section}\alph{subsection}}

\allowdisplaybreaks

\theoremstyle{plain}
\newtheorem{thm}{Theorem}[section]

\newtheorem{theorem}[thm]{Theorem}

\newtheorem{corollary}[thm]{Corollary}

\newtheorem{lemma}[thm]{Lemma}

\newtheorem{definition}[thm]{Definition}

\newtheorem{remark}[thm]{Remark}

\newtheorem{proposition}[thm]{Proposition}

\newtheorem{non-example}[thm]{Non-example}

\newtheorem{condition}[thm]{Condition}

\newtheorem{properties}[thm]{Properties}

\newtheorem{conjecture}[thm]{Conjecture}

\newtheorem{conventions}[thm]{Conventions}
\newtheorem*{claim*}{Claim} 
\newtheorem*{lemma*}{Lemma}
\newtheorem*{theorem*}{Theorem}
\newtheorem*{conjecture*}{Conjecture}

\newcommand{\bC}{{\mathbb C}}

\newcommand{\bF}{{\mathbb F}}

\newcommand{\bR}{{\mathbb R}}

\newcommand{\bZ}{{\mathbb Z}}

\newcommand{\scrA}{\EuScript A}
\newcommand{\scrB}{\EuScript B}
\newcommand{\scrC}{\EuScript C}

\newcommand{\scrF}{\EuScript F}

\newcommand{\scrO}{\EuScript O}

\newcommand{\frakm}{\mathfrak{m}}

\newcommand{\frakA}{\mathfrak{A}}
\newcommand{\frakB}{\mathfrak{B}}

\newcommand{\frakD}{\mathfrak{D}}

\newcommand{\frakM}{\mathfrak{M}}

\newcommand{\frakR}{\mathfrak{R}}

\newcommand{\half}{{\textstyle\frac{1}{2}}}

\newcommand{\iso}{\cong}
\newcommand{\htp}{\simeq}
\newcommand{\smooth}{C^\infty}

\begin{document}
\title[Quantum connection mod $p$]{The quantum connection\\ and its mod $p$ reduction}
\author{Dan Pomerleano, Paul Seidel}

\begin{abstract}
Recent progress on the structure of the quantum connection for monotone symplectic manifolds has used two approaches, which share the common feature of reducing to mod $p$ coefficients. We refine and compare those approaches. In particular, we establish a relation with quantum Steenrod operations which is stronger than that in \cite{chen24c}, leading to more precise information about the singularity at $\infty$ of the quantum connection; and for the version relative to a smooth anticanonical divisor, we draw attention to the implications of the mod $p$ Fontaine-Laffaille structure from \cite{petrov-vaintrob-vologodsky17}. 
\end{abstract}

\maketitle

\section{Introduction}
The results of this paper concern two related but distinct contexts. The first is the quantum connection of a closed monotone symplectic manifold. This is reviewed in Section \ref{subsec:quantum-setup}, and our results stated in Section \ref{subsec:quantum-new}. We then switch to a ``relative'' context (meaning, relative to a smooth anticanonical divisor). Section \ref{subsec:relative} introduces the relative version of the quantum connection, from a Floer-theoretic viewpoint. This in itself is not new, but as we'll see in Section \ref{subsec:fl}, it comes with additional structures of categorical origin, which appear to have no known explanation in terms of the standard TQFT picture of Floer theory.

\subsection{The setup\label{subsec:quantum-setup}}
Let $M^{2n}$ be a closed symplectic manifold which is monotone, 
\begin{equation} \label{eq:monotone}
[\omega_M] = c_1(M). 
\end{equation}
Among the most basic enumerative invariants of $M$ is the $q$-linear degree $0$ endomorphism
\begin{equation} \label{eq:quantum-product-c1}
H^*(M;\bC)[q^{\pm 1}] \longrightarrow H^*(M;\bC)[q^{\pm 1}], \quad
x \longmapsto q^{-1}c_1(M) \ast_q x.
\end{equation}
Here, $q$ is a formal variable of degree $2$, and $\ast_q$ is the small quantum product, which counts rational curves in class $A \in H_2(M;\bZ)$ with weight $q^{\int_A c_1(M)}$. Essentially equivalently, one can set $q = 1$ in the quantum product, with the outcome written as $\ast$, and look at the $\bZ/2$-graded map
\begin{equation} \label{eq:quantum-product-c1-2}
H^*(M;\bC) \longrightarrow H^*(M;\bC), \quad
x \longmapsto c_1(M) \ast x.
\end{equation}
This invariant has a deep influence on the symplectic topology of $M$. For instance, if $L$ is a monotone Lagrangian submanifold which is {\em Spin}, the Poincar\'e dual of $[L]$ is (zero or) an eigenvector of \eqref{eq:quantum-product-c1-2}, for the eigenvalue given by counting Maslov $2$ holomorphic discs through a generic point of $L$; this follows e.g.\ from \cite[Lemma 9.2]{ritter-smith12}. 

A closely related structure is the (small) quantum connection, the degree $0$ endomorphism
\begin{equation} \label{eq:quantum-connection}
\nabla_{u\partial_q}: H^*(M;\bC)[q^{\pm 1},u] \longrightarrow H^*(M;\bC)[q^{\pm 1},u], \quad
\nabla_{u\partial_q}x = u\partial_q x + q^{-1} c_1(M) \ast_q x,
\end{equation}
where $u$ is another formal variable of degree $2$. Again, this is closely connected to the structure of the Fukaya category of $M$, see e.g.\ \cite{hugtenburg24}. One can invert $u$, or basically equivalently, set $u = 1$ and reduce the grading to $\bZ/2$. This yields a $\bZ/2$-graded connection
\begin{equation} \label{eq:u-equals-1-connection}
\nabla_{\partial_q}: H^*(M;\bC)[q^{\pm 1}] \longrightarrow H^*(M;\bC)[q^{\pm 1}], \quad
\nabla_{\partial_q}x = \partial_q x + q^{-1} c_1(M) \ast_q x,
\end{equation}
which relates more easily to the classical theory of linear differential equations. From the definition, $\nabla_{\partial_q}$ has a regular singularity at $q = 0$, with unipotent monodromy. The singularity at $q = \infty$ is irregular in all known cases; its structure has been the subject of much attention, e.g.\ in Dubrovin's conjecture \cite{dubrovin98}. One recent result is the following:

\begin{theorem} \cite{chen24c} \label{th:exponential-type}
(i) In a Laurent expansion around $q = \infty$, which means working over $\bC((q^{-1}))$, the connection $\nabla_{\partial_q}$ is isomorphic to a direct sum
\begin{equation} \label{eq:exponential-type}
\bigoplus_\lambda\; \big(\partial_q + \lambda I + q^{-1} A_{\lambda}\big).
\end{equation}
Here $\lambda$ are the eigenvalues of \eqref{eq:quantum-product-c1-2}; the rank of each piece is the dimension of the corresponding generalized eigenspace; $I$ is the identity; and the $A_\lambda$ are constant ($q$-independent) matrices.

(ii) The $A_\lambda$ have rational eigenvalues; equivalently, $\exp(2\pi i A_\lambda)$ is quasi-unipotent (its eigenvalues are roots of unity).
\end{theorem}

From the Hukuhara-Turrittin-Levelt theorem which classifies formal linear differential equations, one knows that the normal form \eqref{eq:exponential-type} is unique in a suitable sense; one way to express that is to say that the $\exp(2\pi i A_\lambda)$ are unique up to conjugation. In particular, (ii) is independent of any ambiguities in the choice of $A_\lambda$. The same holds for further results to be stated below.

Prior to \cite{chen24c}, Theorem \ref{th:exponential-type} had been proved in \cite{pomerleano-seidel23} under the added assumption that $M$ contains a smooth anticanonical divisor, in the following sense:

\begin{definition} \label{th:divisor}
A smooth anticanonical divisor is a symplectic hypersurface $D \subset M$ Poincar{\'e} dual to $c_1(M)$, such that the natural Liouville manifold structure on $M \setminus D$ is actually Weinstein.
\end{definition}

In that situation, one has more information:

\begin{proposition} \cite[Theorem 1.2.4]{pomerleano-seidel23} \label{th:jordan-bound}
Suppose that $M$ contains a smooth symplectic anticanonical divisor. Then, the size of the Jordan blocks of each $A_\lambda$ in \eqref{eq:exponential-type} is bounded as follows: $\leq n+1$ for integer eigenvalues of $A_\lambda$, and $\leq n$ for the non-integer ones.
\end{proposition}

It is currently unknown whether Proposition \ref{th:jordan-bound} holds without the extra assumption. 

\subsection{New results\label{subsec:quantum-new}}
We refine the method of \cite{chen24c} to obtain a different bound on Jordan blocks.

\begin{theorem} \label{th:jordan-bound-2} (Proved in Section \ref{subsec:splittings})
Take the action of \eqref{eq:quantum-product-c1-2} on the generalized $\lambda$-eigenspace: its nilpotent part determines a conjugacy class of nilpotent matrices. Then, the nilpotent part of $A_\lambda$ lies in the closure of that conjugacy class. 
\end{theorem}


\begin{corollary} \label{th:semisimple}
If \eqref{eq:quantum-product-c1-2} is diagonalizable, so are the $A_\lambda$ in \eqref{eq:exponential-type}; equivalently, the $\exp(2\pi i A_\lambda)$ have finite order.
\end{corollary}

Note that the monodromy of the quantum connection in its standard sense (either in a formal expansion around $q = 0$; or more analytically, letting $q$ go around any loop in $\bC^*$) is conjugate to $x \mapsto \exp(-2\pi i c_1(M)) \smile x$, hence always has a Jordan block of size $n+1$. The two statements do not contradict each other, because the formal power series in $q^{-1}$ which appear in Theorem \ref{th:exponential-type} are not locally convergent (Stokes phenomenon).

\begin{remark} \label{th:singularities}
It is instructive to compare the statements above with ones in singularity theory, which concern the Gauss-Manin connection of a polynomial $W \in \bC[x_1,\dots,x_n]$ with an isolated critical point at $x = 0$. This is related to the quantum connection through mirror symmetry, taking concrete shape via results such as e.g.\ \cite{smith25}. However, the analogy is imperfect in two ways: first, mirror superpotentials may have non-isolated critical points; secondly, what's actually relevant is the Fourier-Laplace transform of the Gauss-Manin connection of the mirror (see e.g.\ \cite[Section 2.2]{pomerleano-seidel23} for discussion), a step which we have suppressed for simplicity.

(i) The Gauss-Manin connection is regular singular; its monodromy is quasi-unipotent; and the Jordan blocks of the monodromy are of size $\leq n$. This is an instance of the general monodromy theorem in algebraic geometry. Compare Theorem \ref{th:exponential-type} and Proposition \ref{th:jordan-bound}.

(ii) Let $\mathit{Jac}(W) = \bC[[x_1,\dots,x_n]]/(\partial_1W,\dots,\partial_nW)$ be the Jacobian ring. Take the nilpotent endomorphism
\begin{equation} \label{eq:w}
W: \mathit{Jac}(W) \longrightarrow \mathit{Jac}(W).
\end{equation}
A Hodge-theoretic argument shows that the nilpotent part of the monodromy of the Gauss-Manin connection lies in the closure of the conjugacy class of \eqref{eq:w}
(this is due to Varchenko \cite{varchenko81}; expository accounts are in \cite[Theorem 14.19]{agv2}, \cite[Section 7.4]{kulikov}; a precursor result is \cite{scherk}). The comparison is with Theorem \ref{th:jordan-bound-2}; Corollary \ref{th:semisimple} parallels the special case of a quasi-homogeneous singularity, which is when \eqref{eq:w} vanishes \cite{saito71}, and where the monodromy has finite order. In fact, one can derive Varchenko's theorem in a way which parallels our proof of Theorem \ref{th:jordan-bound-2}, avoiding Hodge theory; 
see Section \ref{sec:appendix}.
%

(iii) There is a singularity theory result which lacks a parallel for quantum cohomology. One has $W^n \in \mathit{Jac}(W)$ by \cite{briancon-skoda74}
; in contrast, it seems that we do not have a dimension-dependent upper bound on the Jordan block size of \eqref{eq:quantum-product-c1-2} (or equivalently, to use terminology which will become more important below, on the multiplicity of zeros of its minimal polynomial).
\end{remark}

The proofs of all theorems stated so far (both known and new) have one aspect in common, namely the use of reduction mod $p$, which is then translated back into properties of the quantum connection via Katz' method \cite{katz70}. The mechanism underlying Theorem \ref{th:jordan-bound-2} is:

\begin{proposition} \label{th:jae-conjecture} (Proved in Section \ref{subsec:jae})
Let's work in cohomology with $\bF_p$-coefficients, for some prime $p$. Then, the $p$-curvature of $\nabla_{u\partial_q}$ (which is just the $p$-th power of that operation) agrees with a specific quantum Steenrod endomorphism:
\begin{equation} \label{eq:p-curvature-is-qsigma}
\nabla_{u\partial_q}^p = Q\Sigma_{q^{-1}c_1(M)}.
\end{equation}
\end{proposition}

Quantum Steenrod operations were introduced by Fukaya \cite{fukaya93} and further studied in \cite{wilkins18, seidel-wilkins21}. The proof of Theorem \ref{th:exponential-type} used a weaker statement \cite[Proposition 3.3]{chen24c}, namely that the difference between the two sides of \eqref{eq:p-curvature-is-qsigma} is nilpotent. Proposition \ref{th:jae-conjecture} is an instance of a conjecture due to Etingof and J. H. Lee. That conjecture was first proved in \cite{lee23b} in a different situation, namely, for the equivariant quantum connection on the Springer resolution; recently, it has also been proved for Calabi-Yau threefolds \cite{bai-lee-pomerleano26}. 

Before continuing, we introduce some more notation. Write $f(z)$ for the minimal polynomial of \eqref{eq:quantum-product-c1-2}, or equivalently of \eqref{eq:quantum-product-c1-3}:
\begin{equation} \label{eq:minimal-polynomial}
\parbox{37em}{$f(z)$ is the lowest degree monic polynomial such that $f(c_1(M)) = 0$ in the ring $(H^*(M;\bC), \ast_1)$, or equivalently $f(q^{-1}c_1(M)) = 0$ in $(H^*(M;\bC)[q^{\pm 1}], \ast_q)$.}
\end{equation}
Because of the contribution from the classical cup product, the degree of $f(z)$ is $\geq n+1$, where as usual $n = \mathrm{dim}_{\bC}(M)$. Also, since \eqref{eq:quantum-product-c1-2} preserves the integer lattice in $H^*(M;\bC)$, $f(z)$ has integer coefficients (by Gauss' Lemma). Occasionally, we will encounter the homogeneous version $f_q(z) \in \bZ[q,z]$:
\begin{equation} \label{eq:f-q}
\text{if $f(z) = \sum_k b_k z^k$,$\;$ then $\;f_q(z) = \sum_k q^{\mathrm{deg}(f)-k} b_k z^k$,}
\end{equation}
which satisfies
\begin{equation} \label{eq:f-q-c1}
f_q(c_1(M)) = 0\; \text{ in } (H^*(M;\bC)[q],\ast_q).
\end{equation}

Here's the application to the quantum connection. It's a consequence of Proposition \ref{th:jae-conjecture}, but its statement does not refer to quantum Steenrod operations:

\begin{corollary} \label{th:congruence} (Proved in Section \ref{subsec:jae})
Let $f$ be the minimal polynomial \eqref{eq:minimal-polynomial}, and $p$ a prime such that
\begin{equation} \label{eq:no-torsion}
\parbox{37em}{
$H^*(M;\bZ)$ has no $p$-torsion.
}
\end{equation}
Take \eqref{eq:u-equals-1-connection} acting on integer cohomology, and insert its $p$-th power into $f$. The resulting map
\begin{equation}
f(\nabla_{\partial_q}^p): H^*(M;\bZ)[q^{\pm 1}] \longrightarrow H^*(M;\bZ)[q^{\pm 1}]
\end{equation}
is zero mod $p$, by which we mean that it lands in $pH^*(M;\bZ)[q^{\pm 1}] \subset H^*(M;\bZ)[q^{\pm 1}]$.
\end{corollary}

Corollary \ref{th:congruence} is really about cohomology with $\bF_p$-coefficients; we have preferred to state it as a congruence in integer cohomology, to emphasize that it provides infinitely many constraints on the same operation $\nabla_{\partial_q}$. The next result is more purely a characteristic $p$ statement.

\begin{corollary} \label{th:trivial} (Proved in Section \ref{subsec:splittings})
Take a finite field $F$, and suppose that the corresponding version of \eqref{eq:quantum-product-c1-2},
\begin{equation} \label{eq:quantum-product-c1-3}
H^*(M;F) \longrightarrow H^*(M;F), \;\; x \longmapsto c_1(M) \ast x, 
\end{equation}
is diagonalizable over $F$. Then the connection $\nabla_{\partial_q}$ on $H^*(M;F)[q^{\pm 1}]$ is isomorphic to 
\begin{equation} \label{eq:trivial-splitting}
\bigoplus_\lambda\; (\partial_q + \lambda I).
\end{equation} 
Here, $\lambda$ are the eigenvalues of \eqref{eq:quantum-product-c1-3}; and each piece of \eqref{eq:trivial-splitting} lives on a free $F[q^{\pm 1}]$-module whose rank is equal to the dimension of the corresponding eigenspace.
\end{corollary}

\begin{remark}
In both Theorem \ref{th:exponential-type} and \ref{th:jordan-bound-2}, the statement applies separately to each part of the $\bZ/2$-graded cohomology. For instance, if \eqref{eq:quantum-product-c1-2} is diagonalizable on $H^{\mathrm{odd}}(M;\bC)$, then the quantum connection on $H^{\mathrm{odd}}(M;\bC)((q^{-1}))$ is isomorphic to \eqref{eq:exponential-type} with diagonalizable $A_\lambda$. (We have not included that in the main statements, to keep things simple.) In contrast, the proof of Corollary \ref{th:trivial} does not respect the $\bZ/2$-grading, so that the isomorphism which transforms $\nabla_{\partial_q}$ into \eqref{eq:trivial-splitting} could in principle mix odd and even degrees. One can impose mild additional assumptions to avoid that, see Remark \ref{th:why-theta}, and then the resulting statement again applies to each part separately.
\end{remark}

%
%

\subsection{The relative situation\label{subsec:relative}}
We define the subject of this section, the quantum connection relative to a smooth anticanonical divisor, through symplectic cohomology, following \cite{pomerleano-seidel23, pomerleano-seidel24}. 

\begin{remark}
Quite likely, there is an equivalent definition in terms of punctured Gromov-Witten invariants, and that would be useful for computations. However, the symplectic cohomology perspective forms the necessary starting point for our argument, because it connects to the corresponding deformation of the wrapped Fukaya category of the complement. There are related papers \cite{borman-sheridan-varolgunes22, borman-el-alami-sheridan24, borman-el-alami-sheridan24b} which study deformed symplectic cohomology more generally for divisors with normal crossings.
\end{remark}

So, let $D \subset M$ be a smooth anticanonical divisor (Definition \ref{th:divisor}). The deformed symplectic cohomology $\mathit{SH}^*_q(M,D;\bZ)$ is a graded $\bZ[q]$-module obtained from a deformed version of the chain complex underlying the symplectic cohomology $\mathit{SH}^*(M \setminus D;\bZ)$. For clarity, we will always include the coefficient ring in the notation; for this initial discussion, we'll choose integer coefficients, but will later switch to $\bF_p$. It is easy to describe deformed symplectic cohomology additively (as a graded abelian group) \cite[Theorem 1.2.1]{pomerleano-seidel23}:
\begin{equation} \label{eq:deformed-sh}
\mathit{SH}^*_q(M,D;\bZ) \iso H^*(M;\bZ)[q] \oplus H^*(D;\bZ)z \oplus H^*(D;\bZ)z^2 \oplus \cdots
\end{equation}
The powers of $z$ are just notation which serves to distinguish the different copies of $H^*(D;\bZ)$; there is no $z$-linearity. 

The $\bZ[q]$-module structure on $\mathit{SH}^*_q(M,D;\bZ)$ has the following partial description in terms of \eqref{eq:deformed-sh}. $H^*(M;\bZ)[q]$ is a submodule (with the standard $q$-action); moreover, any element of $\mathit{SH}^*_q(M,D;\bZ)$ is mapped to that submodule by a sufficiently high power of $q$. Therefore, we get an isomorphism  (\cite[Corollary 1.2.2]{pomerleano-seidel24}, or \cite{borman-el-alami-sheridan24} in a slightly different setup) 
\begin{equation} \label{eq:invert-q}
H^*(M;\bZ)[q^{\pm 1}] \stackrel{\iso}{\longrightarrow} \bZ[q^{\pm 1}] \otimes_{\bZ[q]} \mathit{SH}^*_q(M,D;\bZ).
\end{equation}

We'll need two pieces of additional data. The first is a module structure over quantum cohomology (Sections \ref{sec:modulestructure} and \ref{sec:compatibility}):
\begin{equation} \label{eq:circ-module}
\begin{aligned}
& \bullet_q: H^*(M;\bZ)[q] \otimes_{\bZ[q]} \mathit{SH}^*_q(M,D;\bZ) \longrightarrow \mathit{SH}^*_q(M,D;\bZ), \\
& x_1 \bullet_q (x_2 \bullet_q y) = (x_1 \ast_q x_2) \bullet_q y.
\end{aligned}
\end{equation}
%
%
The second, related, structure is map (see \cite[Section 8]{pomerleano-seidel24} or the exposition in Section \ref{sec:kappa}) 
\begin{equation} \label{eq:kappa-map}
a_{q}: \mathit{SH}^*_q(M,D;\bZ) \longrightarrow \mathit{SH}^*_q(M,D;\bZ).
\end{equation}
This endomorphism plays a role parallel to \eqref{eq:quantum-product-c1}, without requiring us to invert $q$. That is formalized by the next result:

\begin{lemma} (Proved in Section \ref{sec:modulestructure}) \label{th:kappa-and-quantum-product}
For $x \in \mathit{SH}^*_q(M,D;\bZ)$,
\begin{equation}
q\,a_{q}(x) = c_1(M) \bullet_q x.
\end{equation}
\end{lemma}

\begin{remark} \label{th:pair-of-pants}
A possible alternative approach might be to first construct a pair-of-pants product on deformed symplectic cohomology, 
\begin{equation}
\label{eq:relative-quantum-product}
\mathit{SH}^*_q(M,D;\bZ) \otimes \mathit{SH}^*_q(M,D;\bZ) \stackrel{\ast_q}{\longrightarrow} \mathit{SH}^*_q(M,D;\bZ),
\end{equation}
and show that on $H^*(M;\bZ)[q] \subset \mathit{SH}^*_q(M,D;\bZ)$ it is the quantum product. Then, the map \eqref{eq:circ-module} would be defined by restricting the pair-of-pants product. Next, \eqref{eq:kappa-map} would be defined as the pair-of-pants-product with a specific class in deformed symplectic cohomology, namely 
\begin{equation} \label{eq:1z}
1 z \in H^0(D;\bZ)z \subset \mathit{SH}^0_q(M,D;\bZ). 
\end{equation}
Finally, Lemma \ref{th:kappa-and-quantum-product} would be a consequence of the fact that 
\begin{equation} \label{eq:q-times-z}
q (1 z) = c_1(M), 
\end{equation}
which is a special case of \cite[Lemma 8.2.1]{pomerleano-seidel23}. Even though this provides a more unified picture, we have preferred not to follow that approach, since the pair-of-pants product is not important for our main topic. 
\end{remark}

Here's an easy consequence (we follow our exposition of quantum cohomology by stating it with complex coefficients).

\begin{corollary} (Proved in Section \ref{sec:applications}) \label{th:torsion-1}
Let $f$ be the minimal polynomial \eqref{eq:minimal-polynomial}. Then 
\begin{equation} \label{eq:f-kappa}
q^n f(a_{q}) = 0
\end{equation}
as an endomorphism of $\mathit{SH}^*_q(M,D;\bC)$ (meaning, $f(a_{q})$ is a linear combination of iterates of $a_{q}$).
\end{corollary}

There is also an $S^1$-equivariant (with respect to loop rotation) version of the theory, denoted by $\mathit{SH}^*_{u,q}(M,D;\bZ)$. As a graded $\bZ[u]$-module, it can be computed as follows \cite[Corollary 1.2.4]{pomerleano-seidel24}:
\begin{equation} \label{eq:deformed-sh-2}
\mathit{SH}^*_{u,q}(M,D;\bZ) \iso H^*(M;\bZ)[u,q] \oplus H^*(D;\bZ)[u]z \oplus H^*(D;\bZ)[u]z^2 \oplus \cdots
\end{equation}
Forgetting the equivariant structure relates this to \eqref{eq:deformed-sh} in the obvious way. It also carries a $\bZ[q]$-module structure; as before, the module structure is standard on the first summand in \eqref{eq:deformed-sh-2}, and the inclusion of that summand induces an isomorphism \cite[Corollary 1.2.6]{pomerleano-seidel24}
\begin{equation} \label{eq:invert-q-2}
H^*(M;\bZ)[u,q^{\pm 1}] \stackrel{\iso}{\longrightarrow} \bZ[q^{\pm 1}] \otimes_{\bZ[q]} \mathit{SH}^*_{u,q}(M,D;\bZ).
\end{equation}
Finally, $\mathit{SH}^*_{u,q}(M,D;\bZ)$ carries a connection $\nabla_{u\partial_q}$, which fits into a commutative diagram with the quantum connection, after multiplication by $q$ \cite[Proposition 1.2.8]{pomerleano-seidel24}:
\begin{equation} \label{eq:connections}
\xymatrix{
H^*(M;\bZ)[u,q] \ar[d]_-{q\nabla_{u\partial_q}} \ar@{^{(}->}[rr] && \mathit{SH}_{u,q}^*(M,D;\bZ) \ar[d]^-{q\nabla_{u\partial_q}}
\\
H^{*+2}(M;\bZ)[u,q] \ar@{^{(}->}[rr] && \mathit{SH}_{u,q}^{*+2}(M,D;\bZ)
}
\end{equation}
By construction \cite[Section 9]{pomerleano-seidel24}, $a_{q}$ is the non-equivariant part of the connection; meaning that there is a commutative diagram
\begin{equation} \label{eq:kappa-and-connection}
\xymatrix{
\ar[d]_-{\nabla_{u\partial_q}}
\mathit{SH}_{u,q}^*(M,D;\bZ) \ar@{->>}[rr]^-{u=0} &&
\ar[d]^-{a_{q}}
\mathit{SH}_{q}^*(M,D;\bZ)
\\
\mathit{SH}_{u,q}^*(M,D;\bZ) \ar@{->>}[rr]^-{u=0} &&
\mathit{SH}_{q}^*(M,D;\bZ)
}
\end{equation}

\begin{remark} \label{th:relative-etingof-lee}
There is an analogue of the Etingof-Lee $p$-curvature conjecture in the relative context. This would say that the $p$-curvature of the connection, extended to $\mathit{SH}^*_{u,q}(M,D;\bF_p)[\theta]$, can be written as the symplectic cohomology analogue of a quantum Steenrod operation (a natural equivariant lift of the $p$-th power of $a_q$; there might be more such operations here than in the quantum cohomology context, because rotations of the cylinder are nontrivial, which is why we wrote ``a'' and not ``the'' natural equivariant lift). As far as the authors can see, the strategy of proof of Proposition \ref{th:jae-conjecture} does not readily carry over to the relative version.%
%
%
%
%
\end{remark}

Because inverting $q$ brings us back to the classical quantum connection, the results from Section \ref{subsec:quantum-new} have implications in the relative context. For instance:

\begin{corollary} \label{th:torsion-2}
Let $f$ be the minimal polynomial \eqref{eq:minimal-polynomial}. Take $p$ such that \eqref{eq:no-torsion} holds. Then for each $x \in \mathit{SH}^*_{u,q}(M,D;\bZ)$ there is an $A \geq 0$ such that 
\begin{equation} \label{eq:a-constant}
q^A f(\nabla_{u\partial_q}^p) x \in p\,\mathit{SH}^*_{u,q}(M,D;\bZ).
\end{equation}
\end{corollary}


\begin{proof}
Corollary \ref{th:congruence} can equivalently be stated as saying that $f(\nabla_{u\partial_q}^p)$, as an endomorphism of $H^*(M;\bF_p)[q^{\pm 1},u]$, is zero (indeed, when we prove it, this is the form in which it will appear). From that and the $\bF_p$-coefficient version of \eqref{eq:invert-q-2}, it follows that for any $x \in \mathit{SH}^*_{u,q}(M,D;\bF_p)$, $f(\nabla_{u\partial_q}^p)x$ is $q$-torsion. To get back to integer coefficients, one uses the long exact coefficient sequence
\begin{equation} \label{eq:sh-coefficient}
\cdots
\rightarrow \mathit{SH}^*_{u,q}(M,D;\bZ) \stackrel{p}{\longrightarrow} \mathit{SH}^*_{u,q}(M,D;\bZ) 
\longrightarrow \mathit{SH}^*_{u,q}(M,D;\bF_p) \rightarrow \cdots
\end{equation}
\end{proof}

Actually, the constant $A$ in \eqref{eq:a-constant} can be taken to be independent of $x$, because of a suitable finite generation property of $\mathit{SH}^*_{u,q}(M,D;\bF_p)$; and by a more precise version of that argument, one should be able to get an explicit bound on $A$.
We will explain the idea in Remark \ref{th:bound}, but will not carry out the argument completely; in any case, the resulting bound \eqref{eq:weak-bound} is likely far from optimal, compare Conjecture \ref{th:stronger}.
%
%
%

\subsection{The mod $p$ Fontaine-Laffaille structure\label{subsec:fl}}
The following discussion applies only to symplectic cohomology with mod $p$ coefficients. We need to consider certain algebraic modifications of that theory. As before, let $f$ be the minimal polynomial \eqref{eq:minimal-polynomial}. Define
\begin{align}
\label{eq:invert-f}
\mathit{SH}^*_{q,1/f}(M,D;\bF_p) & \stackrel{\mathrm{def}}{=} \bF_p[a_{q},1/f(a_{q})] \otimes_{\bF_p[a_{q}]} \mathit{SH}_q^*(M,D;\bF_p), \\
\label{eq:invert-f-2}
\mathit{SH}^*_{u,q,1/f}(M,D;\bF_p) & \stackrel{\mathrm{def}}{=} \bF_p[\nabla_{u\partial_q},1/f(\nabla_{u\partial_q})] \otimes_{\bF_p[\nabla_{u\partial_q}]} \mathit{SH}^*_{u,q}(M,D;\bF_p).
\end{align}
Here, \eqref{eq:invert-f-2} is equipped with the following action of $q$:
\begin{equation} \label{eq:leibniz-rule}
q \big( f(\nabla_{u\partial_q}\big)^{-k} x) = f(\nabla_{u\partial_q})^{-k} (qx) +
uk f(\nabla_{u\partial_q})^{-k-1} f'(\nabla_{u\partial_q}) x.
\end{equation}
Finally, one can invert $u$ to form
\begin{align}
\label{eq:tate-sh}
\mathit{SH}^*_{u^{\pm 1},q,1/f}(M,D;\bF_p) & \stackrel{\mathrm{def}}{=} \bF_p[u^{\pm 1}] \otimes_{\bF_p[u]}
\mathit{SH}^*_{u,q,1/f}(M,D;\bF_p).
\end{align}

\begin{remark}
In the algebra generated by $q$ and $\nabla_{u\partial_q}$, one has $[\nabla_{u\partial_q},q] = u$, and therefore 
\begin{equation} \label{eq:f-leibniz}
[f(\nabla_{u\partial_q}),q] = u \,f'(\nabla_{u\partial_q}). 
\end{equation}
The need for that to continue to hold, after inverting $f(\nabla_{u\partial_q})$, prescribes the formula \eqref{eq:leibniz-rule}. For readers unfamiliar with this kind of Weyl algebra, it may be helpful to pass to the Fourier-Laplace dual generators, setting $t = \nabla_{u\partial_q}$, $\nabla_{u\partial_t} = -q$. In that case, \eqref{eq:f-leibniz} is written as the more familiar differentiation rule $[\nabla_{u\partial_t},f(t)] = uf'(t)$, and correspondingly for \eqref{eq:leibniz-rule}. 
\end{remark}

In this context, there is a highly nontrivial relation between the equivariant and non-equivariant theories. It is obtained from the mod $p$ Fontaine-Laffaille structure for smooth and proper families of $A_\infty$-categories constructed in \cite{petrov-vaintrob-vologodsky17}, by applying the strategy of \cite{pomerleano-seidel23}. The outcome is:

\begin{theorem} (Proved in Section \ref{sec:the-end}) \label{th:q-fontaine-laffaille}
Suppose that:
\begin{equation} \label{eq:large-p}
\parbox{37em}{$p$ is an odd prime with $p \geq n+3$, where $n = \mathrm{dim}_{\bC}(M)$; and moreover, \eqref{eq:no-torsion} holds.}
\end{equation}
Then there is a canonical degree $0$ map between \eqref{eq:invert-f} and \eqref{eq:tate-sh},
\begin{equation} \label{eq:q-fontaine-laffaille}
\Phi: \mathit{SH}^*_{q,1/f}(M,D;\bF_p) \longrightarrow \mathit{SH}^*_{u^{\pm 1},q,1/f}(M,D;\bF_p),
\end{equation}
such that
\begin{align}
\label{eq:q-fontaine-laffaille-1} & \nabla_{u\partial_q}^p \Phi(x) = \Phi(a_{q}(x)), \\
\label{eq:q-fontaine-laffaille-2} & q\Phi(x) = \nabla_{u\partial_q}^{p-1} \Phi(q x).
\end{align}
This map is injective, and its image satisfies
\begin{equation} \label{eq:q-fontaine-laffaille-3}
\bF_p[u,u^{-1}] \otimes_{\bF_p} \big(\mathit{im}(\Phi) \oplus \nabla_{u\partial_q} \mathit{im}(\Phi) \oplus \cdots \oplus \nabla_{u\partial_q}^{p-1} \mathit{im}(\Phi)\big)
= \mathit{SH}^*_{u^{\pm 1},q,1/f}(M,D;\bF_p).
\end{equation}
\end{theorem}

One can use Lemma \ref{th:kappa-and-quantum-product} to write a consequence of \eqref{eq:q-fontaine-laffaille-1} and \eqref{eq:q-fontaine-laffaille-2}, which involves the quantum module structure rather than $a_q$:
\begin{equation}
\nabla_{u\partial_q} ( q \,\Phi(x) ) = \Phi(c_1(M) \bullet_q x).
\end{equation}

\begin{remark} \label{th:p-m}
In algebraic geometry, the mod $p$ Fontaine-Laffaille structure is only the first truncation of a richer structure (integral $p$-adic Hodge theory, e.g.\ \cite{faltings89}). One can speculate on what such a structure would do in symplectic topology. Let's work with coefficients in $\bZ/p^m$, for some $m \geq 1$. Take the chain complex underlying $\mathit{SH}^*_{u,q}(M,D;\bZ/p^m)$, and change both the differential and connection by multiplying the $u^k$ term with $p^k$. Denote the outcome by $\overline{\mathit{SH}}_{u,q}^{*}(M,D;\bZ/p^m)$, and the connection by $\overline{\nabla}_{pu\partial_q}$ (the last-mentioned notation is motivated by 
$[\overline{\nabla}_{pu\partial_q}, q] = pu$).
A plausible generalization of \eqref{eq:q-fontaine-laffaille} might be to have a $u$-linear map
\begin{equation} \label{eq:q-fontaine-laffaille-m}
\Phi: \overline{\mathit{SH}}^{*}_{u^{\pm 1}, q,1/f}(M,D;\bZ/p^m) \longrightarrow \mathit{SH}^*_{u^{\pm 1},q,1/f}(M,D;\bZ/p^m),
\end{equation}
where on the domain we have inverted $f(\overline{\nabla}_{pu\partial_q})$,
satisfying
\label{eq:q-fontaine-laffaille-1m} 
$\nabla_{u\partial_q}^p \Phi(x) = \Phi(\overline{\nabla}_{pu\partial_q} x)$
and \eqref{eq:q-fontaine-laffaille-2}. 
It remains to be seen whether such a map exists, and what implications that might have (we will look at the simplest example in Remark \ref{th:speculation}).
\end{remark}

Here is a concrete consequence of Theorem \ref{th:q-fontaine-laffaille} for the relative quantum connection:

\begin{corollary} (Proved in Section \ref{sec:the-end}) \label{th:torsion-3}
Suppose \eqref{eq:large-p} holds. Then for each $x \in \mathit{SH}^*_{u,q}(M,D;\bZ)$ there is a $B \geq 0$ such that 
\begin{equation} \label{eq:b-constant}
q^{np} f(\nabla_{u\partial_q}^p)^Bx \in p\,\mathit{SH}^*_{u,q}(M,D;\bZ)
\end{equation}
\end{corollary}

It is possible to take $B$ independent of $x$, because of the finite generation property mentioned before (see Remark \ref{th:bound}); but that observation is not particularly useful, as it does not provide any bound on $B$. 

It is instructive to compare Corollary \ref{th:torsion-3}, which relies on the Fontaine-Laffaille map and therefore ultimately on the categorical methods from \cite{pomerleano-seidel23}, with Corollary \ref{th:torsion-2}, which used quantum Steenrod operations as in \cite{chen24c}. Qualitatively these statements are of the same kind, but quantitatively they are different, and neither can be obtained from the other. There is a possible common strengthening:

\begin{conjecture} \label{th:stronger}
Assume $p$ is chosen so that \eqref{eq:no-torsion} holds. Then
\begin{equation} 
q^{np} f(\nabla_{u\partial_q}^p): \mathit{SH}^*_{u,q}(M,D;\bZ) \longrightarrow \mathit{SH}^*_{u,q}(M,D;\bZ) \;\;\text{ is zero mod $p$.}
\end{equation}
\end{conjecture}

One potential path towards this starts with the relative Etingof-Lee conjecture from Remark \ref{th:relative-etingof-lee} (see Remark \ref{th:more} for further discussion).

{\em Structure of the paper.} Sections \ref{sec:steenrod}--\ref{sec:apply} discuss the effect of quantum Steenrod operations on the quantum connection. Sections \ref{sec:algebra}--\ref{sec:proof} concern deformed symplectic cohomology, as well as the categorical background required for Theorem \ref{th:q-fontaine-laffaille}. Section \ref{sec:p1} looks at the example $M = \bC P^1$, and compares our results with the outcome of independent computations. The final Section \ref{sec:appendix} shows how the material on $p$-curvature from Section \ref{sec:p-curvature} can be used in singularity theory (see the discussion in Remark \ref{th:singularities}); this is independent of any symplectic geometry considerations.

{\em Acknowledgments.} The authors would like to thank Sasha Petrov for patiently answering questions about \cite{petrov-vaintrob-vologodsky17} and Fontaine-Laffaille structures. D.P. was partially supported by NSF grant DMS-2306204.

\section{Quantum Steenrod\label{sec:steenrod}}
This section starts by reviewing the formalism of quantum Steenrod operations, and then specializes to the operation associated with the first Chern class. The main result (Corollary \ref{th:conjugacy}) compares that operation with the quantum product \eqref{eq:quantum-product-c1-3}.

\subsection{Basic properties}
Let $F$ be a finite field, with $p = \mathrm{char}(F)$, and $\mathit{Frob}: F \rightarrow F$, $\mathit{Frob}(x) = x^p$ its Frobenius automorphism. Following \cite{seidel-wilkins21}, the quantum Steenrod operation associated to $b \in H^*(M;F)[q]$ is a $(u,\theta,q)$-linear map
\begin{equation} \label{eq:quantum-steenrod-endomorphism}
Q\Sigma_b: H^*(M;F)[u,\theta,q] \longrightarrow H^*(M;F)[u,\theta,q].
\end{equation}
The degree of \eqref{eq:quantum-steenrod-endomorphism} is $p$ times that of $b$. The extra formal variable $\theta$ has degree $1$, and satisfies $\theta^2 = u$ if $p = 2$, respectively $\theta^2 = 0$ for $p>2$. It appears here in the context of the cohomology ring of the cyclic group $C_p$,
\begin{equation}
H^*_{C_p}(\mathit{point};F) = H^*(BC_p;F) \iso F[u,\theta].
\end{equation}
We have allowed coefficients in any $F$, rather than just $\bF_p$ as in the literature. This causes occasional appearances of the Frobenius in formulae; apart from that, there is nothing significantly different, and indeed the operations are completely determined by their restriction to $\bF_p \subset F$. 
Let's recall some basic facts.

\begin{properties} \label{th:quantum-properties}
(i) $Q\Sigma_{b+c} = Q\Sigma_b + Q\Sigma_c$, and $Q\Sigma_{g(q)b} = g(q)^p\,Q\Sigma_b$ for $g \in F[q]$.

(ii) For $b \in H^*(M;F)$, $Q\Sigma_b(x) = \mathit{St}(b)x + O(q)$. Here, $\mathit{St}(b) \in H^*(M;F)[u,\theta]$ is the classical Steenrod $p$-th power of $b$ (see \cite[Equation (1.6)]{seidel23} for conventions); and we take the cup product of that with $x$. This expresses the fact that the $q^0$ term of $Q\Sigma$ is a version of the definition of Steenrod operations. 

(iii) $Q\Sigma_b(x) = (b^{\ast_q p}) \ast_q x + O(u,\theta)$, where $b^{\ast_q p}$ is the $p$-th power with respect to the quantum product. This is due to the definition of $Q\Sigma$ as an equivariant version of the $(p+1)$-fold quantum product.

(iv) $Q\Sigma_b$ commutes with $\nabla_{uq\partial_q} = q\nabla_{u\partial_q}$. Here, we have multiplied by $q$ merely to avoid the appearance of negative powers of $q$; and the connection has been tacitly extended to be $\theta$-linear. This is the main result of \cite{seidel-wilkins21}.

(v) $Q\Sigma_b \circ Q\Sigma_c = (-1)^{|b|\, |c|\,\frac{p(p-1)}{2}}\, Q\Sigma_{b \ast_q c}$. This is \cite[Proposition 4.8]{seidel-wilkins21}. Moreover, $Q\Sigma_1 = \mathit{id}$.
\end{properties}

To make full use of (ii), it is convenient to add some basic facts about classical Steenrod operations. For any $b \in H^*(M;F)$, 
\begin{equation} \label{eq:classical-st}
\mathit{St}(b) = \pm u^{(p-1)|b|/2} \mathit{Frob}(b) + \text{elements of }
H^k(M;F)[u,\theta], \text{ for $k > |b|$.}
\end{equation}
Here $\mathit{Frob}$ is Frobenius acting on $F$-coefficients, compare the second part of (i); and if $p = 2$ and $|b|$ is odd, we set $u^{1/2} = \theta$. Finally, consider the special case where
$b \in H^2(M;\bF_p)$ is the mod $p$ reduction of a class in $H^2(M;\bZ)$. Then, the only nonzero Steenrod operations are (in the classical notation) $P^0(b) = b$ and $P^1(b) = b^p$, which with our conventions means 
\begin{equation} \label{eq:degree-2-st}
\mathit{St}(b) = b^p - u^{p-1} b. 
\end{equation}

\begin{lemma} \label{th:no-theta}
Suppose that \eqref{eq:no-torsion} holds, as well as the following additional assumption:
\begin{equation}
\label{eq:no-theta-2} \parbox{37em}{$p>2$, or $b$ has even degree.}
\end{equation}
Then $Q\Sigma_b$ has trivial $\theta$-component, which means that it maps $H^*(M;F)[u,q]$ to itself. 
\end{lemma}

\begin{proof}
It is enough to show this for $F = \bF_p$. One can define quantum Steenrod operations with $(\bZ/p^2)$-coefficients (see the discussion in \cite[Section 1.4]{seidel25}). Under assumption \eqref{eq:no-theta-2}, they take the following form:
\begin{equation} \label{eq:p-square-diagram}
\xymatrix{
\ar[d] H_{C_p}^*(M;\bZ/p^2)[q] \ar[rr]^-{Q\Sigma_{\tilde{b}}} && H^*_{C_p}(M;\bZ/p^2)[q] \ar[d] \\
H^*_{C_p}(M;\bF_p)[q] \ar[rr]^-{Q\Sigma_b} && H_{C_p}^*(M;\bF_p)[q] 
}
\end{equation}
Here, $H^*_{C_p}$ is equivariant cohomology with respect to the trivial $C_p$-action, so $H^*_{C_p}(M;\bF_p) = H^*(M;\bF_p)[u,\theta]$. In the top row, we have $\tilde{b} \in H^*(M;\bZ/p^2)[q]$ and the mod $p^2$ operations; in the bottom row, its mod $p$ reduction $b$ and the usual quantum Steenrod. 

Let's spell out what happens on the group cohomology level. We have
\begin{equation}
H^*_{C_p}(\mathit{point};\bZ/p^2) = \begin{cases} \bZ/p^2 & * = 0, \\ \bF_p & * > 0. \end{cases}
\end{equation}
The reduction map $H^*_{C_p}(\mathit{point};\bZ/p^2) \rightarrow H^*_{C_p}(\mathit{point};\bF_p)$ is an isomorphism in positive even degrees, and zero in odd degrees. Together with \eqref{eq:no-torsion}, this means that any $x \in H^*(M;\bF_p)[u,q] \subset H^*_{C_p}(M;\bF_p)[q]$ can be lifted to some $\tilde{x} \in H^*_{C_p}(M;\bZ/p^2)[q]$. Again appealing to \eqref{eq:no-torsion}, any $b \in H^*(M;\bF_p)[q]$ can be lifted to $\tilde{b} \in H^*(M;\bZ/p^2)[q]$. From \eqref{eq:p-square-diagram} we then see that $Q\Sigma_b(x)$, as the reduction of $Q\Sigma_{\tilde{b}}(\tilde{x})$, has trivial $\theta$-component.
\end{proof}

\begin{remark} \label{th:twisted-coefficients}
In the classical construction of Steenrod operations, one encounters equivariant cohomology with twisted coefficients, involving the sign of permutations. This appears because of the need to make $b^{\otimes p}$ into an equivariant class. The same happens for the mod $p^2$ quantum Steenrod operations: the general version of \eqref{eq:p-square-diagram} has as its top row
\begin{equation} \label{eq:twisted-steenrod}
H_{C_p}^*(M;\mathit{sgn}^a)[q] \xrightarrow{Q\Sigma_{\tilde{b}}} H^*_{C_p}(M;\mathit{sgn}^{(a+|\tilde{b}|)})[q],
\end{equation}
Here, $\mathit{sgn}: C_p \rightarrow \{\pm 1\} \subset (\bZ/p^2)^\times$ is the sign of a cyclic permutation, and $\mathit{sgn}^a$ its $a$-th power, hence trivial for even $a$. For $p>2$, we have that $\mathit{sgn}$ itself is trivial. For $p = 2$ and $|\tilde{b}|$ even, one can take $a = 0$ and then again, only trivial coefficients appear in \eqref{eq:twisted-steenrod}. This was the situation in \eqref{eq:p-square-diagram}, and explains the condition \eqref{eq:no-theta-2}. Let's quickly look at the remaining case $p = 2$ and $|b|$ odd. The forgetful map
\begin{equation}
H^*_{C_2}(\mathit{point};\mathit{sgn}) \longrightarrow H^*_{C_2}(\mathit{point};\bF_2)
\end{equation}
is zero in even degrees, and an isomorphism in odd degrees. In terms of quantum Steenrod operations, still assuming \eqref{eq:no-torsion}, this leads to the opposite conclusion: only the terms of $Q\Sigma_b$ which involve multiplying by $\theta$ are nonzero.
\end{remark}

\subsection{Comparison with the quantum product}
Define
\begin{equation} \label{eq:qst}
\begin{aligned}
& \mathit{QSt}: H^{*/p}(M;F)[q^{1/p}] \longrightarrow H^*(M;F)[u,\theta,q], \\
& \mathit{QSt}(q^{k/p} b) = q^k Q\Sigma_b(1).
\end{aligned}
\end{equation}
Here, $H^{*/p}$ means that we multiply the natural grading by $p$. The same applies to $q$, which is defined to have degree $2p$ on the domain of the map (so that the root $q^{1/p}$ has degree $2$). The outcome is that \eqref{eq:qst} is a graded map. 

At this point, we invert $u$ (for $p = 2$, read this as $F[u^{\pm 1},\theta] = F[\theta^{\pm 1}]$), and then take the graded completion with respect to $q$. When applied to $H^*(M;F)[u,\theta,q]$, the outcome is the graded space $H^*(M;F)[u^{\pm 1},\theta][[q]]$ of formal power series in $q$ with coefficients in $H^*(M;F)[u^{\pm 1},\theta]$. To spell this out: elements of $H^*(M;F)[u^{\pm 1},\theta][[q]]$ of degree $d$ are series
\begin{equation}
c = \sum_{m \geq 0,\; i \in \bZ} q^m (u^i c_{m,i} + u^i\theta\, c_{m,i}'), \quad
\left\{
\begin{aligned}
& c_{m,i} \in H^{d-2i-2m}(M;F), \\ & c_{m,i}' \in H^{d-2i-2m-1}(M;F).
\end{aligned}
\right.
\end{equation}
Note that for grading reasons, only finitely many $i$ are relevant for each value of $m$. The same inverting-and-completing can be applied to $H^{*/p}(M;F)[u,\theta,q^{1/p}]$, with the grading as in \eqref{eq:qst}, leading to a space $H^{*/p}(M;F)[u^{\pm 1},\theta][[q^{1/p}]]$. This time, elements of degree $d$ are of the form
\begin{equation} \label{eq:b-series}
b = \sum_{m \geq 0,\; i \in \bZ} q^{m/p} (u^i b_{m,i} + u^i\theta\, b_{m,i}'), \quad
\left\{
\begin{aligned}
& b_{m,i} \in H^{(d-2i-2m)/p}(M;F), \\ & b_{m,i}' \in H^{(d-2i-2m-1)/p}(M;F).
\end{aligned}
\right.
\end{equation}
We can extend \eqref{eq:qst} to a $(u,\theta)$-linear map 
\begin{equation} \label{eq:extended-qst}
\mathit{QSt}: H^{*/p}(M;F)[u^{\pm 1},\theta][[q^{1/p}]] \longrightarrow H^*(M;F)[u^{\pm 1},\theta][[q]].
\end{equation}
Concretely, this is defined by applying $b \mapsto Q\Sigma_b(1)$ to each coefficient in \eqref{eq:b-series}, and treating powers of $q^{1/p}$ as in \eqref{eq:qst}.


\begin{lemma} \label{th:extended-qst-1}
The map \eqref{eq:extended-qst} is a (Frobenius-twisted) isomorphism. 
\end{lemma}

\begin{proof}
As a consequence of \eqref{eq:classical-st}, the $q = 0$ reduction is an isomorphism. Using the complete decreasing filtration by powers of $q^{1/p}$ (on the domain) and $q$ (on the target), it follows that the entire map is an isomorphism.
\end{proof}

\begin{lemma} \label{th:extended-qst-2}
The map \eqref{eq:extended-qst} fits into a commutative diagram
\begin{equation} \label{eq:qst-diagram}
\xymatrix{
H^{*/p}(M;F)[u^{\pm 1},\theta][[q^{1/p}]] 
\ar[d]_-{c_1(M) \ast_q \cdot} 
\ar[rr]^-{\mathit{QSt}}_-{\iso} 
&&
H^*(M;F)[u^{\pm 1},\theta][[q]] 
\ar[d]^-{Q\Sigma_{c_1(M)}} \\
H^{*/p}(M;F)[u^{\pm 1},\theta][[q^{1/p}]] 
\ar[rr]^-{\mathit{QSt}}_-{\iso} 
&& 
H^*(M;F)[u^{\pm 1},\theta][[q]]
}
\end{equation}
where in the left $\downarrow$ the quantum product has been extended linearly with respect to $(u^{\pm 1},\theta)$.
\end{lemma}

\begin{proof}
This is an immediate consequence of Properties \ref{th:quantum-properties}(i) and (v): for $b \in H^*(M;F)$ we have $Q\Sigma_{c_1(M)} Q\mathit{St}(b) = Q\Sigma_{c_1(M)} Q\Sigma_b(1) = Q\Sigma_{c_1(M) \ast_q b}(1) = Q\mathit{St}(c_1(M) \ast_q b)$; the general case follows from that by $(u,\theta,q^{1/p})$-linearity.
\end{proof}

Assume that:
\begin{equation} \label{eq:all-eigenvalues}
\parbox{37em}{All eigenvalues $\lambda$ of \eqref{eq:quantum-product-c1-3} lie in $F$.} 
\end{equation}
Then, there are degree zero elements $e_{\lambda} \in H^*(M;F)[q^{\pm 1}]$, which are idempotents for the quantum product, and project to the corresponding generalized eigenspaces. Consider the operations
\begin{equation} \label{eq:quantum-ei}
Q\Sigma_{e_\lambda}: H^*(M;F)[u,\theta,q^{\pm 1}] \longrightarrow H^*(M;F)[u,\theta,q^{\pm 1}],
\end{equation}
defined by the obvious extension of \eqref{eq:quantum-steenrod-endomorphism} to classes $b$ which have negative powers of $q$. The $Q\Sigma_{e_\lambda}$ are mutually orthogonal projections, and add up to the identity, by Property \ref{th:quantum-properties}(v). The same argument as in Lemma \ref{th:extended-qst-2} (applied to spaces which allow finitely many inverse powers of $q$ respectively $q^{1/p}$) shows:

\begin{lemma} \label{th:extended-qst-3}
The map \eqref{eq:quantum-ei} sits in a commutative diagram
\begin{equation}
\xymatrix{
H^{*/p}(M;F)[u^{\pm 1},\theta]((q^{1/p}))
\ar[d]_-{e_{\lambda} \ast_q\; \cdot } 
\ar[rr]^-{\mathit{QSt}}_-{\iso} 
&&
H^*(M;F)[u^{\pm 1},\theta]((q)) 
\ar[d]^-{Q\Sigma_{e_\lambda}} \\
H^{*/p}(M;F)[u^{\pm 1},\theta]((q^{1/p}))
\ar[rr]^-{\mathit{QSt}}_-{\iso} 
&& 
H^*(M;F)[u^{\pm 1},\theta]((q)). 
}
\end{equation}
\end{lemma}

From now on, let's assume \eqref{eq:no-torsion} and also a stronger version of \eqref{eq:no-theta-2}:
\begin{equation} \label{eq:no-theta-3}
\parbox{37em}{$p>2$, or $H^*(M;F)$ is zero in odd degrees.}
\end{equation}
Lemma \ref{th:no-theta} tells us that we can consider quantum Steenrod operations without using the variable $\theta$. As for the variable $u$, rather than inverting it, we set it equal to $1$. The resulting versions of $Q\Sigma_{c_1(M)}$, $Q\Sigma_{e_\lambda}$ and \eqref{eq:extended-qst} are $\bZ/2$-graded maps
\begin{align}
\label{eq:qsigmac1-u-1}
& (Q\Sigma_{c_1(M)})_{u=1}: H^*(M;F)[q] \longrightarrow H^*(M;F)[q], 
\\ 
\label{eq:qsigmaelambda-u-1}
& (Q\Sigma_{e_\lambda})_{u=1}: H^*(M;F)[q^{\pm 1}] \longrightarrow H^*(M;F)[q^{\pm 1}], 
\\
\label{eq:qst-u-1}
& (\mathit{QSt})_{u=1}: H^*(M;F)[[q^{1/p}]] \stackrel{\iso}{\longrightarrow} H^*(M;F)[[q]].
\end{align}
In \eqref{eq:qst-u-1}, it was not necessary to write the domain as $H^{*/p}$: if $p$ is odd, this is because the grading is mod $2$; and for $p = 2$, everything is concentrated in even degrees by \eqref{eq:no-theta-3}. The following result, while somewhat technical, is a key ingredient in the proof of Theorem \ref{th:jordan-bound-2}. Because that theorem concerns the behaviour near $q = \infty$, it naturally involves the corresponding formal expansion of \eqref{eq:qsigmac1-u-1} and \eqref{eq:qsigmaelambda-u-1}, which means working with Laurent series in $q^{-1}$.

\begin{corollary} \label{th:conjugacy}
Suppose that \eqref{eq:no-torsion}, \eqref{eq:all-eigenvalues}, and \eqref{eq:no-theta-3} hold.

(i) The action of $(Q\Sigma_{q^{-1}c_1(M)})_{u = 1} = q^{-p}(Q\Sigma_{c_1(M)})_{u=1}$ on $H^*(M;F)((q^{-1}))$ is conjugate to the $F((q^{-1}))$-linear extension of \eqref{eq:quantum-product-c1-3}.

(ii) For any eigenvalue $\lambda$ of \eqref{eq:quantum-product-c1-3}, consider the action of $(Q\Sigma_{q^{-1}c_1(M)})_{u = 1} - \lambda^p I$ on the direct summand of $H^*(M;F)((q^{-1}))$ which is the image of the idempotent endomorphism $(Q\Sigma_{e_{\lambda}})_{u = 1}$. Then, that (nilpotent) action is conjugate to the $F((q^{-1}))$-linear extension of the action of $(c_1(M) \ast \cdot) - \lambda I$ on the generalized $\lambda$-eigenspace of \eqref{eq:quantum-product-c1-3}.
\end{corollary}

\begin{proof}
(i) We start by working with Laurent series in $q$ rather than $q^{-1}$. Take the $\bZ/2$-graded isomorphism
\begin{equation}
\begin{aligned}
& Q: H^*(M;F)((q^{1/p})) \stackrel{\iso}{\longrightarrow} H^*(M;F)((q)), \\
& Q(q^r c) = q^{p(r + \sigma(|c|))} \mathit{Frob}(c) \;\; \text{ for $c \in H^*(M;F)$,} \\
& \qquad \text{where }\sigma(k) = k/2 \text{ for even $k$}, \;\; \sigma(k) = (k-1)/2 \text{ for odd $k$.}
\end{aligned}
\end{equation}
Because of the way in which the quantum product is graded,
\begin{equation} \label{eq:big-q-map-2}
Q(b \ast_q c) = Q(b) \ast Q(c) \text{ if $b$ or $c$ are of even degree.}
\end{equation}
Note that $Q(q^{-1}c_1(M)) = c_1(M)$. As a consequence of this, the $u = 1$ version of \eqref{eq:qst-diagram}, and \eqref{eq:big-q-map-2}, we have a commutative diagram
\begin{equation} \label{eq:forget-u-2}
\xymatrix{
H^*(M;F)((q)) \ar[d]_-{c_1(M) \ast\; \cdot\,} 
&&
H^*(M;F)((q^{1/p}))
\ar[ll]^-{\iso}_-{Q}
\ar[d]_-{q^{-1}c_1(M) \ast_q\; \cdot}
\ar[rr]_-{\iso}^-{(\mathit{QSt})_{u=1}}
&&
H^*(M;F)((q))
\ar[d]^-{(Q\Sigma_{q^{-1}c_1(M)})_{u=1}}
\\
H^*(M;F)((q)) && 
H^*(M;F)((q^{1/p})) 
\ar[ll]^-{\iso}_-{Q}
\ar[rr]_-{\iso}^-{(\mathit{QSt})_{u=1}}
&&
H^*(M;F)((q))
}
\end{equation}
The leftmost $\downarrow$ is \eqref{eq:quantum-product-c1-3} extended trivially to $F((q))$-coefficients. We have $Q(q^{1/p}b) = qQ(b)$ and $(\mathit{QSt})_{u=1}(q^{1/p}b) = q(\mathit{QSt})_{u=1}(b)$. As a consequence, the composition of the two horizontal isomorphisms in \eqref{eq:forget-u-2} is $F((q))$-linear. Hence, the rightmost $\downarrow$ is conjugate over $F((q))$ to the leftmost one. To get from here to the desired result, one starts with rational functions $F(q)$ and the map
\begin{equation}
(Q\Sigma_{q^{-1}c_1(M)})_{u=1}: H^*(M;F)(q) \longrightarrow H^*(M;F)(q).
\end{equation}
From the previous argument, we know that this becomes conjugate to \eqref{eq:quantum-product-c1-3} after extending constants to $F((q))$. Note that because of \eqref{eq:all-eigenvalues}, its eigenvalues lie in $F$, hence a fortiori in $F(q)$. By comparing Jordan normal forms, one sees that conjugacy already holds in $F(q)$; from there, one extends constants to $F((q^{-1}))$.
%

%
(ii) Because $e_\lambda$ lies in degree $0$, We have $Q(e_\lambda) = \mathit{Frob}(e_\lambda)_{q = 1}= (e_{\lambda^p})_{q=1}$. As in \eqref{eq:forget-u-2}, there is a diagram
\begin{equation} \label{eq:bigq-elambda}
\xymatrix{
H^*(M;F)((q)) \ar[d]_-{(e_{\lambda^p})_{q=1} \ast \cdot} 
&&
H^*(M;F)((q^{1/p}))
\ar[ll]^-{\iso}_-{Q}
\ar[d]_-{e_{\lambda} \ast_q\; \cdot}
\ar[rr]_-{\iso}^-{(\mathit{QSt})_{u=1}}
&&
H^*(M;F)((q))
\ar[d]^-{(Q\Sigma_{e_{\lambda}})_{u=1}}
\\
H^*(M;F)((q)) && 
H^*(M;F)((q^{1/p}))
\ar[ll]^-{\iso}_-{Q}
\ar[rr]_-{\iso}^-{(\mathit{QSt})_{u=1}}
&&
H^*(M;F)((q)).
}
\end{equation}
The $\downarrow$ are projections (idempotent endomorphisms), and commute with the corresponding maps in \eqref{eq:forget-u-2}. We can therefore restrict the diagram \eqref{eq:forget-u-2} to the images of those projections. Let's write down those restrictions, for the left and right columns only, and subtract $\lambda^p$ times the identity:
\begin{equation}
\label{eq:forget-u-3}
\xymatrix{
\mathit{im}\big( (e_{\lambda^p})_{q=1} \ast \cdot \big) \ar[d]_-{c_1(M) \ast\; \cdot\, - \lambda^p I} \ar@{<->}[rr]^-{\iso}
&&
\mathit{im}\big( (Q\Sigma_{e_\lambda})_{u=1} \big)
\ar[d]^-{(Q\Sigma_{q^{-1}c_1(M)})_{u=1} - \lambda^p I}
\\
\mathit{im}\big( (e_{\lambda^p})_{q=1} \ast \cdot \big)
\ar@{<->}[rr]^-{\iso}
&&
\mathit{im}\big( (Q\Sigma_{e_\lambda})_{u=1} \big)
}
\end{equation}
On the left hand side, what one has is: take the restriction of \eqref{eq:quantum-product-c1-3} to the generalized $\lambda^p$-eigenspace, consider its nilpotent part, and then extend that trivially to the variable $q$. This nilpotent part is conjugate to its counterpart for the $\lambda$-eigenspace, because one can use the action of Frobenius to compare the two. This proves the $F((q))$-version of the desired statement. The last step is to switch fields, as before: the two vertical maps in \eqref{eq:forget-u-3} are actually defined over $F(q)$; we know that they become conjugate over $F((q))$, and being nilpotent, they therefore must already be conjugate over $F(q)$; from that, one gets conjugacy over $F((q^{-1}))$.
\end{proof}

\section{$P$-curvature\label{sec:p-curvature}}
This section is entirely elementary: it concerns general linear differential equations (connections) and their reduction to characteristic $p$. The main result (Proposition \ref{th:katz}) refines the classical relation between $p$-curvature and monodromy from \cite{katz70}, giving more precise control over the conjugacy class of the monodromy.

\subsection{Nilpotent $p$-curvature and its consequences\label{subsec:nilpotent}}
Take a connection $\nabla_{\partial_q}$ of rank $r$ over a field $F$ of characteristic $p$. Here, we mean that the connection is defined on the free rank $r$ module over one of the rings
\begin{equation} \label{eq:rings}
F[q], \;\; F[q^{\pm 1}],\;\; F(q), \;\; F((q)),
\end{equation}
and satisfies the Leibniz identity $\nabla_{\partial_q}(g(q)x) = g(q) \nabla_{\partial_q}x + (\partial_q g) x$. We will be specific about the choice of ring \eqref{eq:rings} only when it matters for the proofs or subsequent applications. The $p$-curvature is the $p$-th iterate $\nabla_{\partial_q}^p$, which  (nontrivially) is a $q$-linear endomorphism. Equivalently, this can be written in terms of $\nabla_{q\partial_q} = q\nabla_{\partial_q}$ as
\begin{equation} \label{eq:p-curvature-2}
q^p (\nabla_{\partial_q})^p = (\nabla_{q\partial_q})^p - \nabla_{q\partial_q}.
\end{equation}

%
%

\begin{lemma} \label{th:zero-p-curvature}
If the $p$-curvature is zero, the connection is equivalent (related by a change of basis) to the trivial one.
\end{lemma}

This is a consequence of \cite[Theorem 5.1]{katz70}. Let's explain the point concretely for connections over $F[q]$ (the other situations are parallel and even a bit simpler). The argument in \cite{katz70} shows that any $v \in F[q]^r$ can be uniquely written as a linear combination 
\begin{equation} \label{eq:katz}
v = v_0 + qv_1 + \cdots + q^{p-1}v_{p-1}, \quad \text{ where } \nabla_{\partial_q}v_k = 0.
\end{equation}
The space $\mathit{ker}(\nabla_{\partial_q}) \subset F[q]^r$ of covariantly constant sections is an $F[q^p]$-submodule, hence necessarily free of finite rank. If $(v_k)$ is a basis of $\mathit{ker}(\nabla_{\partial_q})$ over $F[q^p]$, then \eqref{eq:katz} shows that it is also a basis of $F[q]^r$ over $F[q]$. In that basis, the connection is obviously trivial.

\begin{remark}
Over $F(q)$ or $F((q))$, a classification of connections based on their $p$-curvature has been obtained in \cite{vanderput95} (see \cite[Section 2]{vanderput96} for a concise exposition).
\end{remark}
%
%

Our next topic is the relation between characteristic $0$ and $p>0$. On the characteristic $0$ side, we work in this context:
\begin{equation} \label{eq:number-ring}
\text{$R \subset \bC$ is a finitely generated ring.}
\end{equation}

\begin{properties} \label{th:r-properties}
(i) If $\frakm \subset R$ is a maximal ideal, then $F = R/\frakm$ is a field which is finitely generated as a ring, and hence finite.

(ii) For sufficiently large $p$, there exists $\frakm$ such that $F = R/\frakm$ has characteristic $p$ \cite[Exercise 2]{serre}.

(iii) $R$ is a Jacobson ring \cite[Tag OOGC]{stacks}, hence the intersection of all maximal ideals is zero.
\end{properties}

The fundamental use of reduction mod $p$ is the following:

\begin{proposition} \label{th:katz}
(i) Let $\nabla_{\partial_q}$ be a connection defined over $R((q))$. Suppose that for every $\frakm$, the induced connection with $F = (R/\frakm)$-coefficients has nilpotent $p$-curvature, which means that $\nabla_{\partial_q}^{pk} \equiv 0$ mod $\frakm$ for some $k$ (note that $p$ depends on $\frakm$, and so does $k$). Then $\nabla_{\partial_q}$ has a regular singularity, with quasi-unipotent monodromy.

(ii) In the same situation, let's make the stronger assumption that there is some $k$ independent of $\frakm$, so that $\nabla_{\partial_q}^{pk} = 0$ mod $\frakm$ for all $\frakm$. Then the Jordan blocks of the monodromy of $\nabla_{\partial_q}$ are of size $\leq k$.
\end{proposition}

This is classical for rings of functions on affine algebraic curves, in which case it applies to the monodromy around each puncture \cite[Theorem 13.0]{katz70}. The adaptation to the Laurent series case is carried out in \cite[Appendix A]{chen24c}. 

\subsection{The unipotent part of the monodromy\label{subsec:nilpotent-2}}
We will need a technical result, which allows one to transform connections with a regular singularity into ones with a non-resonant simple pole. This would be routine if we were working over a field, but it requires a little extra care in our context.

\begin{lemma} \label{th:nilpotent-form}
Let $\nabla_{\partial_q}$ be a connection defined over $R((q))$, which is regular singular with quasi-unipotent monodromy. Then there is a larger ring $\tilde{R} \supset R$ of the same kind \eqref{eq:number-ring}, and an $\tilde{R}((q))$-base change $\tilde{\Phi}$, such that the transformed connection is of the form
\begin{equation} \label{eq:nilpotent-form}
\tilde{\Phi}^{-1} \nabla_{\partial_q} \tilde{\Phi} = \partial_q + (D+N) q^{-1} + O(1), 
\end{equation}
where $D$ is a diagonal matrix with rational entries, which is non-resonant (no two diagonal coefficients differ by a nonzero integer); and $N$ a nilpotent matrix which commutes with $D$. The notation $O(1)$ stands for powers $q^0$ and higher.
\end{lemma}

\begin{proof}
The assumption on $\nabla_{\partial_q}$ means (as part of the elementary theory of regular singular points) that there is $\bC((q))$-base change $\Phi$, which satisfies all the properties we have stated. The problem is that, even though at any order in $q$ only finitely elements of $\bC$ appear in $\Phi$, the entire series might not be defined over a finitely generated ring. We circumvent this by a suitable truncation. Let's expand the connection, the transformation, and its inverse, as
\begin{equation}
\begin{aligned}
& \nabla_{\partial_q} = \partial_q + A_{-N} q^{-N} + A_{-N+1} q^{-N+1} + \cdots \\
& \Phi = \Phi_{-N} q^{-N} + \Phi_{-N+1} q^{-N+1} + \cdots \\
& \Psi = \Phi^{-1} = \Psi_{-N} q^{-N} + \Psi_{-N+1} q^{-N+1} + \cdots
\end{aligned}
\end{equation}
for some $N>0$. Let $\tilde{\Phi}$, $\tilde{\Psi}$ be the result of truncating $\Phi$, $\Psi$ modulo $q^M$, for some $M>0$. Then
\begin{equation}
\tilde{\Phi} \tilde{\Psi} = (\Phi + O(q^M))(\Psi + O(q^M)) = \mathit{Id} + O(q^{M-N}).
\end{equation}
Suppose $M>N$. Then $\tilde{\Phi}$ is an isomorphism, and its inverse is approximately equal to $\Psi$. Namely, if we write $\tilde{\Phi} \tilde{\Psi} = \mathit{Id} + \Delta$, with $\Delta = O(q^{M-N})$, then
\begin{equation} \label{eq:tilde-inverse}
\tilde{\Phi}^{-1} = \tilde{\Psi} (\mathit{Id} - \Delta + \Delta^2 - \cdots) =
(\Psi + O(q^M)) (\mathit{Id} + O(q^{M-N})) = \Psi + O(q^{M-2N}).
\end{equation}
Therefore
\begin{equation} \label{eq:tilde-transform}
\tilde{\Phi}^{-1} \circ \nabla_{\partial_q} \circ \tilde{\Phi} = (\Psi + O(q^{M-2N})) \circ \nabla_{\partial_q} \circ (\Phi + O(q^M)) = \Psi \circ \nabla_{\partial_q} \circ \Phi + O(q^{M-4N}).
\end{equation}
This shows that for $M \geq 4N$, the connection \eqref{eq:tilde-transform} still satisfies our conditions. Now, $\tilde{\Phi}$ and $\tilde{\Psi}$ are polynomial, hence involve only finitely many elements of $\bC$. One can enlarge $R$ to an $\tilde{R}$ containing those elements. Then, $\Delta$ has coefficients in $\tilde{R}$, and hence so does \eqref{eq:tilde-inverse}.
\end{proof}

Next, let's recall some elementary facts about nilpotent matrices. Over any field $K$, denote by $\scrO_{\pi}(K)$ the conjugacy class of nilpotent $(r \times r)$-matrices which corresponds to a given partition $\pi$ of $r$, and $\bar\scrO_{\pi}(K)$ its Zariski closure. 
Concretely, if our partition is given by $\sum_j \pi_j = r$, then $\bar{\scrO}_\pi(K)$ consists of those matrices $N$ such that
\begin{equation} \label{eq:kernel-dimension}
\mathrm{dim}(\mathit{ker}\, N^i) \geq d_i(\pi) \stackrel{\mathrm{def}}{=} \sum_j \mathrm{min}(\pi_j,i).
\end{equation}
Equivalently,
\begin{equation} \label{eq:minors}
\text{all $(r-d_i(\pi)+1)$-minors of $N^i$ are zero.}
\end{equation}
Note that \eqref{eq:minors} is a system of polynomial equations in the coefficients of $N$, with integer coefficients.

\begin{lemma} \label{th:constant-term}
Consider a nilpotent $(r \times r)$-matrix over $K[[q]]$, $N = N_0 + qN_1 + q^2N^2 + \cdots$. If $N$ belongs to $\bar{\scrO}_\pi(K((q)))$ for some $\pi$, then $N_0$ belongs to $\bar{\scrO}_{\pi}(K)$.
\end{lemma}

This is obvious by taking the equations \eqref{eq:minors} that define $\bar\scrO_\pi(K((q)))$, and then setting $q = 0$. 

\begin{lemma} \label{th:orbit-characteristic}
Let $N$ be a nilpotent matrix over $R$, a ring in the class \eqref{eq:number-ring}. Suppose that for all $\frakm$, the mod $\frakm$ reduction of $N$ lies in $\bar{\scrO}_\pi(F)$, for some fixed $\pi$ independent of $\frakm$. Then $N$ lies in $\bar{\scrO}_\pi(\bC)$.
\end{lemma}

This is true because $N$ satisfies \eqref{eq:minors} modulo any $\frakm$, hence must satisfy those equations outright, by Property \ref{th:r-properties}(iii).

With those preliminaries at hand, we turn to our aim, which refines Proposition \ref{th:katz}(ii).

\begin{proposition} \label{th:orbit-closure}
Let $\nabla_{\partial_q}$ be a connection over $R((q))$. Suppose that its reductions modulo $\frakm$ have the property that their $p$-curvature lies in $\bar\scrO_{\pi}(F)$, for some fixed $\pi$ independent of $\frakm$. Then, over $\bC((q))$, the connection is equivalent to one of the form
$\partial_q + q^{-1}A$, where: $A$ has rational eigenvalues, and its nilpotent part lies in $\bar\scrO_{\pi}(\bC)$.
\end{proposition}

\begin{proof}
Suppose we pass to a larger ring $\tilde{R} \supset R$ of the same kind. For any maximal ideal $\tilde\frakm \subset \tilde{R}$, the composition 
\begin{equation}
R \hookrightarrow \tilde{R} \twoheadrightarrow \tilde{F} = \tilde{R}/\tilde{\frakm}
\end{equation}
is a map to a field, so factors through $F = R/\frakm$ for some $\frakm$, followed by $F \hookrightarrow \tilde{F}$. As a consequence, the $p$-curvature of the reduction mod $\tilde{\frakm}$ still lies in $\bar\scrO_{\pi}(\tilde{F})$. The purpose of passing to $\tilde{R}$ is to apply Proposition \ref{th:katz}(i) and Lemma \ref{th:nilpotent-form}, which yield a transformation $\tilde{\Phi}$ that simplifies the connection.

For simplicity, we will suppress both the enlargement of rings and the transformation from the notation. So let's assume that the original connection already had the form
\begin{equation} \label{eq:rational-pole}
\nabla_{\partial_q} = \partial_q + q^{-1}(D + N) + O(1)
\end{equation}
with $(D,N)$ defined over $R$ and having the same properties as in \eqref{eq:nilpotent-form}. Denote the reduction modulo some $\frakm$ by
\begin{equation}
\bar\nabla_{\partial_q} = \partial_q + q^{-1}(\bar{D} + \bar{N}) + O(1).
\end{equation}
Since the coefficients of the diagonal matrix $D$ are rational numbers, they will be mapped to elements of $\bF_p$ under any reduction to $F$, which means that $\bar{D}^p = \bar{D}$. Following \eqref{eq:p-curvature-2}, we compute 
\begin{equation}
\begin{aligned}
q^p\bar\nabla_{\partial_q}^p & = (q\partial_q + \bar{D} + \bar{N} + O(q) )^p - (q\partial_q + \bar{D} + \bar{N} + O(q) )
\\ \qquad \qquad & = \bar{N}^p - \bar{N} + O(q).
\end{aligned}
\end{equation}
It is an elementary fact that multiplication with an invertible element of the ground field maps each nilpotent orbit to itself. By assumption, the $p$-curvature $\bar\nabla_{\partial_q}^p$ lies in $\bar\scrO_{\pi}(F((q)))$; hence, so does $q^p\bar\nabla_{\partial_q}^p$. By Lemma \ref{th:constant-term}, this implies that that $\bar{N}^p-\bar{N}$ lies in $\bar\scrO_{\pi}(F)$. Another elementary fact about nilpotent matrices, clear from looking at single Jordan blocks, is that $\bar{N}^p-\bar{N}$ lies in the same conjugacy class as $\bar{N}$. We therefore get $\bar{N} \in \bar\scrO_{\pi}(F)$. By Lemma \ref{th:orbit-characteristic} it follows that 
\begin{equation} \label{eq:n-conjugacy-class}
N \in \bar\scrO_{\pi}(\bC). 
\end{equation}
Lemma \ref{th:nilpotent-form} also tells us that $D$ is non-resonant. Therefore, as part of the elementary theory of regular singular points, one can find a base change over $\bC$ of the form $I + O(q)$, which preserves the $q^{-1}$ term of the connection, and gets rid of the remaining $O(1)$ term, meaning that it tranforms our connection into $\partial_q + q^{-1}(D+N)$.
\end{proof}

\section{\label{sec:apply}Applications to the quantum connection} 

We now return to the mod $p$ reduction of the quantum connection. By a fairly direct use of the properties of quantum Steenrod operations, this leads to a proof of Proposition \ref{th:jae-conjecture} (and its most immediate consequence, Corollary \ref{th:congruence}). The next part of our argument follows \cite{chen24c} to show how certain splittings of the quantum connection near $q = \infty$, when reduced mod $p$, are related to quantum Steenrod operations. Then, we bring Section \ref{sec:p-curvature} to bear, and use that to derive the remaining results from Section \ref{subsec:quantum-new}.

\subsection{$P$-curvature of the quantum connection\label{subsec:jae}}
We need to slightly adjust the notions from Section \ref{sec:p-curvature} to the presence of the additional variable $u$. For now, it will be sufficient to take $F = \bF_p$. The $p$-curvature of the quantum connection is defined to be $\nabla_{u\partial_q}^p$. To avoid inverse powers of $q$, we prefer to write it as in \eqref{eq:p-curvature-2}, which means to consider the (degree $2p$) map
\begin{equation} \label{eq:p-curvature-alt}
q^p \nabla_{u\partial_q}^p = (\nabla_{uq\partial_q})^p - u^{p-1} \nabla_{uq\partial_q} \,:\,
H^*(M;\bF_p)[u,q] \longrightarrow H^*(M;\bF_p)[u,q].
\end{equation}

\begin{properties} \label{th:p-curvature-properties}
(i) $q^p \nabla_{u\partial_q}^p$ is $(u,q)$-linear. This is obvious for $u$, and true for $q$ because it's a general property of $p$-curvature.

(ii) $q^p \nabla_{u\partial_q}^p x = (c_1(M)^p - u^{p-1} c_1(M)) x + O(q)$. This follows from \eqref{eq:p-curvature-alt} and the fact that $\nabla_{uq\partial_q}$ preserves the filtration by powers of $q$, with the associated graded map being $x \mapsto uq\partial_q x + c_1(M)x$.

(iii) $q^p\nabla_{u\partial_q}^p(x) = c_1(M)^{\ast_q p} \ast_q x + O(u)$. This follows from $\nabla_{uq\partial_q}(x) = c_1(M) \ast_q x + O(u)$.

(iv) $q^p\nabla_{u\partial_q}^p$ commutes with $\nabla_{uq\partial_q}$. This follows directly from \eqref{eq:p-curvature-alt}.

(v) The $\theta$-linear extension of $q^p\nabla_{u\partial_q}^p$ commutes with $Q\Sigma_b$ for all $b$. This follows from \eqref{eq:p-curvature-alt} and Property \ref{th:quantum-properties}(iv).
\end{properties}

\begin{proof}[Proof of Proposition \ref{th:jae-conjecture}]
Having extended the $p$-curvature $\theta$-linearly, we consider the difference
\begin{equation} \label{eq:difference}
q^p\nabla_{u\partial_q}^p - Q\Sigma_{c_1(M)}.
\end{equation}
The first step is to show that \eqref{eq:difference} is strictly decreasing with respect to the grading of cohomology classes. This is \cite[Proposition 3.3]{chen24c}, whose proof we recall here. Both $q^p\nabla_{u\partial_q}^p$ and $Q\Sigma_{c_1(M)}$ commute with $\nabla_{uq\partial_q}$, by Properties \ref{th:quantum-properties}(iv) and \ref{th:p-curvature-properties}(iv). They also have the same $q = 0$ term, by Property \ref{th:quantum-properties}(ii) combined with \eqref{eq:degree-2-st} on one hand, and Property \ref{th:p-curvature-properties}(ii) on the other hand. From those two facts, it follows (by a simple order-by-order argument) that \eqref{eq:difference} is of order $O(q^p)$; see \cite[Lemma 1.2]{seidel-wilkins21}. On the other hand, by comparing Properties \ref{th:quantum-properties}(iii) and \ref{th:p-curvature-properties}(iii), we know that \eqref{eq:difference} becomes zero if we set $u$ and $\theta$ to zero. Altogether, these facts mean that any nontrivial term in \eqref{eq:difference} must be a multiple of $q^{p+1}$ or $q^p u$ or $q^p \theta$. Because the operation has degree $2p$, any such term must decrease the degree of cohomology classes.

Since \eqref{eq:difference} is strictly decreasing with respect to cohomological degrees, it must vanish if we apply it to the class $1$. From Properties \ref{th:quantum-properties}(v) and \ref{th:p-curvature-properties}(v), we see that for $b \in H^*(M;\bF_p)[q]$,
\begin{equation}
\big(q^p\nabla_{u\partial_q}^p - Q\Sigma_{c_1(M)}\big) Q\mathit{St}(b) = 
\big(q^p\nabla_{u\partial_q}^p - Q\Sigma_{c_1(M)}\big) Q\Sigma_b(1) = 
Q\Sigma_b \big(q^p\nabla_{u\partial_q}^p - Q\Sigma_{c_1(M)}\big)(1).
\end{equation}
Hence \eqref{eq:difference} vanishes on all classes $Q\mathit{St}(b)$. By $\bF_p[u^{\pm 1},\theta][[q]]$-linearity, it also vanishes on the image of \eqref{eq:extended-qst}, hence by Lemma \ref{th:extended-qst-1} on all of $H^*(M;\bF_p)[u^{\pm 1},\theta][[q]]$. One can then restrict back to $H^*(M;\bF_p)[u,\theta,q]$.
\end{proof}

\begin{proof}[Proof of Corollary \ref{th:congruence}]
On $H^*(M;\bF_p)[u,\theta,q^{\pm 1}]$, Proposition \ref{th:jae-conjecture} and Property \ref{th:quantum-properties}(v) imply that
\begin{equation}
f(\nabla_{u\partial_q}^p) = f(q^{-p}Q\Sigma_{c_1(M)}) = f(Q\Sigma_{q^{-1}c_1(M)}) = Q\Sigma_{f(q^{-1}c_1(M))}.
\end{equation}
The relation $f(q^{-1}c_1(M)) = 0$ holds in quantum cohomology with $\bF_p$-coefficients (for a simple reason: by definition, it holds integrally after multiplication by some natural number; by assumption \eqref{eq:no-torsion} that number can be chosen coprime to $p$, hence becomes invertible in $\bF_p$). Setting $u = 1$ shows that $f(\nabla_{\partial_q}^p) = 0$ as an endomorphism of $H^*(M;\bF_p)[q^{\pm 1}]$, which is equivalent to the result as stated.
\end{proof}

\subsection{Splittings\label{subsec:splittings}}
Take $R$ as in \eqref{eq:number-ring}. We assume that:
\begin{equation} \label{eq:differences-are-invertible}
\parbox{37em}{all eigenvalues $\lambda$ of \eqref{eq:quantum-product-c1} lie in $R$; and the difference of any two eigenvalues lies in $R^\times$.}
\end{equation}
We then have a decomposition of $H^*(M;R)[q^{\pm 1}]$ into generalized eigenspaces, induced by quantum idempotents $e_\lambda \in H^*(M;R)[q^{\pm 1}]$. We also assume:
\begin{equation} \label{eq:p-large}
\parbox{37em}{for any $R \twoheadrightarrow F = R/\frakm$, we have $p = \mathrm{char}(F) > 2$; and $H^*(M;\bZ)$ has no $p$-torsion.}
\end{equation}
This is unproblematic, since one can enlarge any given $R$ by inverting $2$ as well as those odd primes for which there is torsion in cohomology. 

We will work with the graded $u$-adic completion of $H^*(M;R)[q^{\pm 1},u]$, which is $H^*(M;R)[q^{\pm 1}][[u]]$; in other words, the graded space of power series in $u$ whose coefficients lie in $H^*(M;R)[q^{\pm 1}]$. 

\begin{lemma} \label{th:splitting-lemma} 
(i) There is a unique splitting of free graded $R[q^{\pm 1}][[u]]$-modules,
\begin{equation} \label{eq:splitting-lemma}
H^*(M;R)[q^{\pm 1}][[u]] \iso \bigoplus_{\lambda} H_u^{\lambda},
\end{equation}
whose reduction to $u = 0$ recovers the generalized eigenspace decomposition of $H^*(M;R)[q^{\pm 1}]$ with respect to \eqref{eq:quantum-product-c1}, and which is invariant under $\nabla_{u\partial_q}$.

(ii) Take the reduction of \eqref{eq:splitting-lemma} to coefficients in $F = R/\frakm$ for some $\frakm$. The outcome agrees with the splitting given by the idempotent endomorphisms $Q\Sigma_{e_\lambda}$, where we have reduced $e_\lambda$ mod $\frakm$. Note that, by \eqref{eq:p-large} and Lemma \ref{th:no-theta}, these endomorphisms have no $\theta$-component, hence act on $H^*(M;F)[q^{\pm 1},u]$ (and its completion).
\end{lemma}

\begin{proof}
(i) This a version of \cite[Lemma 2.1]{chen24c}. The existence part can also be expressed as follows: by a change of basis over $R[q^{\pm 1}][[u]]$, the connection can be brought into a split form 
\begin{equation} \label{eq:block}
u\partial_q + A_u, \quad A_u = \bigoplus_{\lambda} A_u^{\lambda},
\end{equation}
with the following property: in the $u$-expansion $A_u^{\lambda} = A^{\lambda}_0 + uA^{\lambda}_1 + \cdots$, the leading term is $A^\lambda_0 = \lambda I + \text{(\em nilpotent endomorphism)}$. We will not further explain the existence part here, referring instead to \cite{chen24c}; but we will spell out the proof of uniqueness, since that is more important for the rest of our argument.

Let $\Phi_u$ be a graded $R[q^{\pm 1}][[u]]$-linear endomorphism which commutes with \eqref{eq:block}, or equivalently 
\begin{equation} \label{eq:covariant-phi}
u\partial_q \Phi_u = \Phi_u A_u - A_u \Phi_u.
\end{equation}
Expand it as $\Phi_u = \Phi_0 + u\Phi_1 + \cdots$, and its blocks with respect to \eqref{eq:block} as $\Phi^{\mu\lambda}_u = \Phi_0^{\mu\lambda} + u\Phi_1^{\mu\lambda} + \cdots$. At $u = 0$ we have $\Phi_0 A_0 - A_0 \Phi_0 = 0$, and therefore
\begin{equation} \label{eq:0-order}
\Phi_0^{\mu\lambda} A_0^{\lambda} - A_0^{\mu} \Phi_0^{\mu\lambda} = 0 \;\; \Leftrightarrow \;\; \Phi_0^{\mu\lambda} (A_0^{\lambda}-\lambda I) - (A_0^{\mu}-\mu I) \Phi_0^{\mu\lambda} = (\mu-\lambda) \Phi_0^{\mu\lambda}.
\end{equation}
For $\lambda \neq \mu$, we have that $\mu-\lambda \in R^\times$ by \eqref{eq:differences-are-invertible}, while $(A^\lambda_0-\lambda I)$ and $(A_0^\mu-\mu I)$ are nilpotent; from that and \eqref{eq:0-order}, one gets $\Phi_0^{\mu\lambda} = 0$. Repeat the argument inductively: suppose that $\Phi_0,\dots,\Phi_{m-1}$ are block diagonal. Then, the equation \eqref{eq:covariant-phi} yields $\Phi_m^{\mu\lambda} A_0^{\lambda} - A_0^{\mu} \Phi_m^{\mu\lambda} = 0$ for $\mu \neq \lambda$, which as before implies $\Phi_m^{\mu\lambda} = 0$. It follows that the entirety of $\Phi$ is block diagonal.

The application to uniqueness is as follows. Given another splitting of $H^*(M;R)[q^{\pm 1}][[u]]$ with the properties from (i), but which is potentially different from \eqref{eq:block}, take $\Phi$ to be the projection to the $\mu$-summand. By the argument above, that projection must be the direct sum of idempotent endomorphisms, one for each $H_u^\lambda$. Moreover, by assumption on the $u = 0$ behaviour of the splitting under consideration, we have: $\Phi_0^{\lambda\lambda} = 0$ for $\lambda \neq \mu$, and then idempotence ensures that $\Phi^{\lambda\lambda} = 0$; and $\Phi^{\mu\mu}_0 = I$, and then it follows that $\Phi^{\mu\mu} = I$. Hence, $\Phi$ is in fact equal to the projection to the $\mu$-summand in \eqref{eq:block}, which means our splitting wasn't different after all.

(ii) This is a version of \cite[Proposition 2.5 and Lemma 4.3]{chen24c}. Note that, as a consequence of \eqref{eq:differences-are-invertible} and \eqref{eq:p-large}, the eigenvalues of the quantum product with $c_1(M)$ remain distinct after reduction to $F$, and the splitting into generalized eigenspaces is also compatible with that reduction. By the uniqueness argument from (i) applied to $F$-coefficients, there is a unique splitting of $H^*(M;F)[q^{\pm 1}][[u]]$ which is compatible with $\nabla_{u\partial_q}$, and which at $u = 0$ recovers the generalized eigenspace decomposition. The splitting given by $Q\Sigma_{e_\lambda}$ satisfies the required conditions: first of all, because of \eqref{eq:p-large} we can apply Lemma \ref{th:no-theta}, so $Q\Sigma_{e_\lambda}$ has no $\theta$-term and acts on our space; the resulting splitting is compatible with $\nabla_{u\partial_q}$, by Property \ref{th:quantum-properties}(iv); and the $u = 0$ reduction of $Q\Sigma_{e_\lambda}$ is quantum product with $e_\lambda^{\ast_q p} = e_\lambda$, see Property \ref{th:quantum-properties}(iii), which by definition is projection to the corresponding generalized eigenspace. Hence, that splitting must agree with the reduction of \eqref{eq:splitting-lemma}.
\end{proof}

\begin{remark}
To see the relation with the way results are formulated in \cite{chen24c}, it is useful to apply a change of variables. Namely, write $Q = u/q$, so that $H^*(M)[q^{\pm 1},u] \iso H^*(M)[q^{\pm 1},Q]$. In the new variables,
$u\partial_q = Qq\partial_q - Q^2\partial_Q$. We adjust our connection so that it differentiates in $Q$-direction:
\begin{equation}
\nabla_{\partial_Q} = -Q^{-2} \nabla_{u\partial_q} = \partial_Q - Q^{-2}(c_1(M) \ast_q \cdot) - Q^{-1}q\partial_q.
\end{equation}
Since $q$ has degree $2$, the endomorphism $\mathit{Gr} = -q\partial_q$ is a ``grading operator'': on the degree $d$ part of our space, it acts by $k$ on $H^{d+2k}(M)q^{-k}$. The same change of variables, applied to the completion, shows that we are working formally around $Q = 0$:
\begin{equation}
H^*(M)[q^{\pm 1}][[u]] = H^*(M)[q^{\pm 1}][[Q]].
\end{equation}
\end{remark}

Let's invert $u$, which means passing to $H^*(M;R)[q^{\pm 1}]((u))$; these are graded Laurent series in $u$ with coefficients in $H^*(M;R)[q^{\pm 1}]$. For grading reasons, increasingly positive powers of $u$ are necessarily accompanied by increasingly negative powers of $q$; so the space can equivalently be written as $H^*(M;R)[u^{\pm 1}]((q^{-1}))$. At the cost of reducing the grading mod $2$, one can therefore set $u = 1$, which yields $H^*(M;R)((q^{-1}))$; more concretely, the part of $H^*(M;R)[q^{\pm 1}]((u))$ in each degree $d$ is canonically isomorphic to $H^*(M;R)((q^{-1}))$ in degree ($d$ mod $2$). From \eqref{eq:splitting-lemma} one inherits a decomposition into $\bZ/2$-graded $R((q^{-1}))$-modules,
\begin{equation} \label{eq:g-decomposition}
H^*(M;R)((q^{-1})) \iso \bigoplus_{\lambda} H^\lambda,
\end{equation}
which is invariant under $\nabla_{\partial_q}$. As a consequence of Lemma \ref{th:splitting-lemma}(ii), the reduction mod $\frakm$ of that splitting agrees with the one obtained from $(Q\Sigma_{e_\lambda})_{u = 1}$.

\begin{corollary} \label{th:p-curvature-constraint}
On each summand in \eqref{eq:g-decomposition}, consider the modified connection 
\begin{equation} \label{eq:modified-connection}
\nabla_{\partial_q}|H^\lambda - \lambda I. 
\end{equation}
(i) After reduction modulo $\frakm$, the $p$-curvature of this connection is nilpotent. 

(ii) More precisely, if the action of \eqref{eq:quantum-product-c1-2} on the generalized $\lambda$-eigenspace has nilpotent part lying in $\bar\scrO_{\pi_\lambda}(\bC)$ for some partition $\pi_\lambda$, then the $p$-curvature of \eqref{eq:modified-connection} lies in $\bar\scrO_{\pi_\lambda}(F((q^{-1})))$.
\end{corollary}

\begin{proof}
(i) From Proposition \ref{th:jae-conjecture} (extended to $F$-coefficients, which is unproblematic), we see that the $p$-curvature of the mod $\frakm$ reduction of \eqref{eq:modified-connection} is 
\begin{equation} \label{eq:modified-p-curvature}
(\nabla_{\partial_q} - \lambda I)^p = \nabla_{\partial_q}^p - \lambda^p I = (Q\Sigma_{q^{-1}c_1(M)})_{u=1} - \lambda^p I = (Q\Sigma_{q^{-1}c_1(M)-\lambda})_{u=1},
\end{equation}
as an endomorphism of the image of the projection $(Q\Sigma_{e_\lambda})_{u= 1}$. By Corollary \ref{th:conjugacy}(ii), which applies here thanks to the assumption \eqref{eq:p-large}, this is conjugate to the ($q$-linear extension of) the action of $c_1(M) \ast \cdot - \lambda I$ on the generalized $\lambda$-eigenspace in $H^*(M;F)$. Clearly, that is nilpotent.

(ii) Consider the nilpotent endomorphism $c_1(M) \ast \cdot - \lambda I$ of the generalized $\lambda$-eigenspace inside $H^*(M;R)$. After passing to $\bC$-coefficients, that endomorphism lies in $\bar{\scrO}_{\pi_\lambda}(\bC)$, hence satisfies the equations \eqref{eq:minors} associated to $\pi_\lambda$. Of course, it must then satisfy those equations already over $R$, which means that it lies in $\bar\scrO_{\pi_\lambda}(R)$. After reduction mod $\pi$, it still satisfies the same equations, hence lies in $\bar\scrO_{\pi_\lambda}(F)$. This and the argument from (i) yield the desired result.
\end{proof}

\begin{proof}[Proof of Theorem \ref{th:jordan-bound-2}]
The first step is to restrict coefficients to a ring $R$ as in \eqref{eq:number-ring}. We choose this so that \eqref{eq:differences-are-invertible} and \eqref{eq:p-large} are satisfied. Take the quantum connection $\nabla_{\partial_q}$ over $R((q^{-1}))$, and split it into direct summands \eqref{eq:g-decomposition}. On each summand, we consider the correspondingly modified connection \eqref{eq:modified-connection}.

Next comes reduction to positive characteristic. For any $\frakm \subset R$, the connection with coefficients in $F = R/\frakm$ induced by \eqref{eq:modified-connection} has $p$-curvature in $\bar\scrO_{\pi_\lambda}(F((q)))$. This is Corollary \ref{th:p-curvature-constraint}, from which we borrow the notation $\pi_{\lambda}$.

Finally, return to characteristic $0$. We apply Proposition \ref{th:orbit-closure}, but with a change of variables to $q^{-1}$ instead of $q$. Given the previous results about $p$-curvature, this tells us that \eqref{eq:modified-connection} is isomorphic over $\bC((q^{-1}))$ to some $\partial_q + q^{-1}A_{\lambda}$, where the nilpotent part of $A_{\lambda}$ lies in $\bar{\scrO}_{\pi_\lambda}(\bC$; and that of course means that $\nabla_{\partial_q}|H^\lambda$ itself is isomorphic to $\partial_q + \lambda I + q^{-1}A_{\lambda}$. After passing to complex coefficients, this gives the constraint on \eqref{eq:exponential-type} which we have claimed.
\end{proof}

\begin{proof}[Proof of Corollary \ref{th:trivial}]
We consider the $u = 1$ version of the quantum connection (with $F$-coefficients), but include the variable $\theta$ (with $\theta^2 = 0$ for $p>2$, respectively $\theta^2 = 1$ for $p = 2$). This means we are considering the $\bZ/2$-graded $(q,\theta)$-linear map
\begin{equation} \label{eq:theta-connection}
\nabla_{\partial_q}: H^*(M;F)[\theta, q^{\pm 1}] \longrightarrow H^*(M;F)[\theta,q^{\pm 1}].
\end{equation}
By assumption, there are quantum idempotents $e_{\lambda} \in H^*(M;F)[q^{\pm 1}]$ which project to the eigenspaces of \eqref{eq:quantum-product-c1-3}, and these satisfy
\begin{equation} \label{eq:semisimple-c1}
(q^{-1}c_1(M) - \lambda) \ast_q e_{\lambda} = 0.
\end{equation}

Take the $\bZ/2$-graded decomposition of $H^*(M;F)[\theta,q^{\pm 1}]$ given by $(Q\Sigma_{e_\lambda})_{u=1}$. On each piece, the $p$-curvature of the modified connection $\nabla_{\partial_q} - \lambda I$ is again given by \eqref{eq:modified-p-curvature}. From \eqref{eq:semisimple-c1} and Property \ref{th:quantum-properties}(v) one gets
\begin{equation}
(Q\Sigma_{q^{-1}c_1(M)-\lambda})_{u=1} \circ (Q\Sigma_{e_\lambda})_{u=1} = 0,
\end{equation}
so the $p$-curvature is zero. One applies Lemma \ref{th:zero-p-curvature} to conclude that the modified connection is isomorphic over $F[q^{\pm 1}]$ to the trivial one. Hence, on each summand, $\nabla_{\partial_q}$ is isomorphic to $\partial_q + \lambda I$.

We still have to do some minor cleanup. The argument above applies separately to each part of the $\bZ/2$-graded connection \eqref{eq:theta-connection}. Let's say that we look only at the even degree part, which is 
\begin{equation} \label{eq:odd-even}
H^{\mathrm{even}}(M;F)[q^{\pm 1}] \oplus \theta H^{\mathrm{odd}}(M;F)[q^{\pm 1}] \iso H^*(M;F)[q^{\pm 1}],
\end{equation}
compatibly with $\nabla_{\partial_q}$. Hence, the outcome is a statement about the (ungraded) space on the right hand side of \eqref{eq:odd-even}. Finally, from \eqref{eq:bigq-elambda} one sees that the $F[q^{\pm 1}]$-rank of the piece of \eqref{eq:odd-even} given by $(Q\Sigma_{e_\lambda})_{u=1}$ equals that of the $\lambda^p$-eigenspace of \eqref{eq:quantum-product-c1-3}, which is the same as that of the $\lambda$-eigenspace (since this is just a computation of the rank of a free module, it can be carried out in the $q$-completion).
\end{proof}

\begin{remark} \label{th:why-theta}
Corollary \ref{th:trivial} allows those $p = \mathrm{char}(F)$ for which $H^*(M;\bZ)$ has $p$-torsion. We have therefore designed the argument to avoid Lemma \ref{th:no-theta}, which is why $\theta$ appeared. If \eqref{eq:no-torsion} holds then using $\theta$ is unnecessary, and one gets a sharper result, which applies to the odd and even degree parts of the quantum connection separately.
\end{remark}

\section{Homological algebra\label{sec:algebra}}
Here, we review and extend the noncommutative geometry part of \cite{pomerleano-seidel23}. On a technical level, the main difference is that part of our argument involves working over a coefficient ring, rather than field. This is required in order to construct the mod $p$ Fontaine-Laffaille structure from \cite{petrov-vaintrob-vologodsky17}, since that depends on a lifting from $\bF_p$ to $\bZ/p^2$; the outcome is packaged as Corollary \ref{th:fontaine-laffaille-2} (if the reader is willing to accept the input from algebra as a black box, they can skip everything except that statement).

\begin{conventions}
(i) When defining $A_\infty$-structures or dga structures over a commutative ring, we always assume that the underlying chain complexes are homotopically projective, and hence also homotopically flat ($K$-projective and $K$-flat, in the language of \cite{spaltenstein88}). In particular, quasi-isomorphisms are chain homotopy equivalences. 

(ii) $A_\infty$-structures are understood to be cohomologically unital (unless strict unitality is specified). Dga structures are always strictly unital.
\end{conventions}

\subsection{Central cocycles\label{subsec:central}}
We start in the classical context of dg algebras, following \cite[Section 3.1]{pomerleano-seidel23}. Let $\scrA$ be a dga over $\bF_p$, together with an element $W \in \scrA^0$ which is a cocycle ($d_\scrA W = 0$) and central ($Wa = aW$ for all $a$). Take the curved deformation $\scrA_q$ (where as usual $q$ is a formal variable of degree $2$) obtained by thinking of $qW$ as curvature term, and leaving the remaining structure (differential, product) unchanged. The chain complex underlying the Hochschild homology $\mathit{HH}_*(\scrA_q)$ will be denoted by $A_q = \mathit{C}_*(\scrA_q)$; like any cohomology theory associated to a formal deformation, this complex is defined so as to be $q$-adically complete. We will need a number of variations. First, there is a version with coefficients in negative powers of $q$:
\begin{equation} \label{eq:negative-q}
q^{-1}A_{q^{-1}} \stackrel{\mathrm{def}}{=} q^{-1}\bF_p[q^{-1}] \otimes_{\bF_p[q]} A_q,
\end{equation}
where $q^{-1}\bF_p[q^{-1}] = \bF_p((q))/\bF_p[[q]]$ (one could write the tensor product as over $\bF_p[[q]]$; there is no difference between the two). Next, the central element $W$ defines a class in the Hochschild cohomology of $\scrA_q$,
\begin{equation} \label{eq:w-class}
[W] \in \mathit{HH}^0(\scrA_q).
\end{equation}
As part of the general action of Hochschild cohomology on Hochschild homology, $W$ acts on $A_q$, by a $q$-linear endomorphism which we write as $\iota_W$; and that clearly carries over to $q^{-1}A_{q^{-1}}$. Given a nonzero polynomial $f(z) \in \bF_p[z]$, we introduce the localised version
\begin{equation} \label{eq:localise-t}
q^{-1}A_{q^{-1},1/f} \stackrel{\mathrm{def}}{=} \bF_p[\iota_W,1/f(\iota_W)] \otimes_{\bF_p[\iota_W]} q^{-1}A_{q^{-1}}.
\end{equation}

A similar, but slightly more complicated, discussion applies to cyclic homology. Let $A_{u,q}$ be the chain complex underlying the negative cyclic homology of $\scrA_q$; this is both $q$-adically and $u$-adically complete. The analogue of \eqref{eq:negative-q} is
\begin{equation} \label{eq:negative-q-2}
q^{-1}A_{u,q^{-1}} \stackrel{\mathrm{def}}{=} q^{-1}\bF_p[q^{-1}] \hat\otimes_{\bF_p[q]} A_{u,q},
\end{equation}
where $\hat{\otimes}$ means that the tensor product should be $u$-adically completed. These complexes carry Getzler-Gauss-Manin connections \cite{getzler95}, which are degree $0$ endomorphisms $\nabla_{u\partial_q}$. We use them to form the counterpart of \eqref{eq:localise-t}, namely
\begin{equation} \label{eq:localise-t-2}
q^{-1}A_{u,q^{-1},1/f} \stackrel{\mathrm{def}}{=} \bF_p[\nabla_{u\partial_q},1/f(\nabla_{u\partial_q})] \hat\otimes_{\bF_p[\nabla_{u\partial_q}]} A_{u,q^{-1}},
\end{equation}
with $t$ acting by $\nabla_{u\partial_q}$, and the same meaning for $\hat\otimes$ as before. Note that the ordering of the operations \eqref{eq:negative-q-2} and \eqref{eq:localise-t-2} matters, since $q$ and $\nabla_{u\partial_q}$ do not commute. The final step is to invert the variable $u$, as in the definition of periodic cyclic homology:
\begin{equation} \label{eq:localise-t-3}
q^{-1}A_{u^{\pm 1},q^{-1},1/f} \stackrel{\mathrm{def}}{=} \bF_p[u^{\pm 1}] \otimes_{\bF_p[u]} q^{-1} A_{u,q^{-1},1/f}.
\end{equation}

The following theorem introduces the mod $p$ Fontaine-Laffaille structure from \cite{petrov-vaintrob-vologodsky17}. 

\begin{theorem} \label{th:fontaine-laffaille}
Let $p$ be an odd prime; $\scrA$ a dga over $\bF_p$, with a central cocycle $W$; and $f(t) \in \bF_p[t]$ a nonzero polynomial. Suppose that:

(i) $\scrA$ admits a lift to $\bZ/p^2$, and $W$ correspondingly lifts to a central cocycle. 

(ii) $H^*(\scrA)$ is of finite rank over $\bF_p[W]$.

(iii) $\scrA$ is homologically smooth over $\bF_p$. 

(iv) In the Hochschild cohomology ring $\mathit{HH}^*(\scrA)$, the class \eqref{eq:w-class} satisfies $q^B f([W]) = 0$ for some $B \geq 0$.

(v) The Hochschild homology $\mathit{HH}_*(\scrA)$ is concentrated in degrees $[3-p,\dots,p-3]$.

Then there is a graded map (generally, depending on the choice of lift in (i), as well as a lift of $f$ to a polynomial with $(\bZ/p^2\bZ)$-coefficients)
\begin{equation} \label{eq:dg-fontaine-laffaille-1}
\Phi: H^*(q^{-1}A_{q^{-1},1/f}) \longrightarrow H^*(q^{-1}A_{u^{\pm 1},q^{-1},1/f}), 
\end{equation}
with the properties
\begin{equation} \label{eq:dg-fontaine-laffaille-1b}
\nabla_{u\partial_q}^p \Phi(x) = \Phi(\iota_W x), \quad 
q\Phi(x) = \nabla_{u\partial_q}^{p-1} \Phi(qx);
\end{equation}
this map is injective, and its image satisfies
\begin{equation} \label{eq:dg-fontaine-laffaille-1c}
\bF_p[u^{\pm 1}] \otimes_{\bF_p} \big( \mathit{im}(\Phi) \oplus \nabla_{u\partial_q} \mathit{im}(\Phi) \oplus \cdots \oplus \nabla_{u\partial_q}^{p-1} \mathit{im}(\Phi) \big) =
H^*(q^{-1}A_{u^{\pm 1},q^{-1},1/f}).
\end{equation}
\end{theorem}

\begin{proof}
A basic construction \cite[Equation (3.1.2)]{pomerleano-seidel23} yields a dga $\scrA_t$ over $\bF_p[t]$, which is quasi-isomorphic to $\scrA$ in such a way that $t$ times the identity corresponds to $W$ in cohomology:
\begin{equation} \label{eq:a-t}
\begin{aligned}
& \scrA_t = \scrA[t, \epsilon], \\
& d_{\scrA_t}a = d_{\scrA}a, \;\; d_{\scrA_t}(a\epsilon) = (d_{\scrA}a)\epsilon + (t-W)a
\;\; \text{for } a \in \scrA,
\end{aligned}
\end{equation}
where the formal variable $\epsilon$ has degree $-1$ (and therefore $\epsilon^2 = 0$). We look at
\begin{equation} \label{eq:invert-f-t}
\scrA_{t,1/f} = \bF_p[t,1/f(t)] \otimes_{\bF_p[t]} \scrA_t,
\end{equation}
which has the following properties.

{\em (i') Let $\tilde{f} \in (\bZ/p^2)[t]$ be a lift of $f$. Then $\scrA_{t,1/f}$ admits a lift to $(\bZ/p^2\bZ)[t,1/\tilde{f}(t)]$.} Indeed, the lift of $(\scrA,W)$ to $\bZ/p^2$ provided by assumption (i) also yields a corresponding lift of \eqref{eq:a-t}, and one can then invert $\tilde{f}$ as in \eqref{eq:invert-f-t}.

{\em (ii') $H^*(\scrA_{t,1/f})$ is of finite rank of $\bF_p[t,1/f(t)]$.} From (ii), we know that $H^*(\scrA_t)$ is of finite rank over $\bF_p[t]$. Inverting $f(t)$ is unproblematic.

{\em (iii') $\scrA_{t,1/f}$ is smooth over $\bF_p[t,1/f(t)]$.} This follows from (iii) and (iv) by \cite[Proposition 3.1.12]{pomerleano-seidel23}.

{\em (iv') The Hochschild homology of $\scrA_{t,1/f}$, defined over $\bF_p[t,1/f(t)]$, is concentrated in degrees $[3-p,\dots,p-3]$.} The upper bound follows from (iii) and (v) as in \cite[Corollary 3.1.13]{pomerleano-seidel23}, but the other part of the argument has to be elaborated on a little. The smoothness property from (iii') implies the nondegeneracy of the Shklyarov pairing on the Hochschild complex, which therefore provides a quasi-isomorphism between that complex and its dual over $\bF_p[t,1/f(t)]$. Using that duality and the universal coefficient theorem (for complexes of free modules, over the principal ideal domain $\bF_p[t,1/f]$), one converts the previous upper bound into a lower bound.
%

At this point, we have all the properties required in order to apply \cite[Theorem 1]{petrov-vaintrob-vologodsky17} to \eqref{eq:invert-f-t}. As formulated in \cite[Equation (1.8)]{petrov-vaintrob-vologodsky17}, the outcome is an isomorphism from the Frobenius-pullback of Hochschild homology to periodic cyclic homology. Spelling out the Frobenius in our case, we get a $u$-linear graded isomorphism
\begin{equation} \label{eq:fontaine-laffaille}
\begin{aligned}
& \Phi: \bF_p[t^{1/p}] \otimes_{\bF_p[t]} \mathit{HH}_*(\scrA_{t,1/f})[u,u^{-1}] \stackrel{\iso}{\longrightarrow} \mathit{HP}_*(\scrA_{t,1/f}), \\
& \Phi(t^{1/p}g \otimes x) = t \Phi(g \otimes x).
\end{aligned}
\end{equation}
The target of \eqref{eq:fontaine-laffaille} carries the Getzler-Gauss-Manin connection $\nabla_{u\partial_t}$ with respect to the variable $t$. Maybe more surprisingly, there is a corresponding connection on the source side, defined as follows. Differentiating the dg structure of $\scrA_{t,1/f}$ with respect to $t$ gives rise to a class in $\mathit{HH}^2(\scrA_{t,1/f})$ (the Kodaira-Spencer class). More concretely, one sees from \eqref{eq:a-t} that this class is represented by the degree $1$ derivation 
\begin{equation} \label{eq:partial-epsilon}
\begin{aligned}
& \partial_\epsilon: \scrA_t \longrightarrow \scrA_t, \\
& \partial_\epsilon(a) = 0, \;\; \partial_\epsilon(a\epsilon) = a.
\end{aligned}
\end{equation}
One then defines the connection $\nabla_{u\partial_{t^{1/p}}}$ on the domain of \eqref{eq:fontaine-laffaille} by 
\begin{equation} \label{eq:frobenius-connection}
\begin{aligned}
& \nabla_{u\partial_{t^{1/p}}}(g(t^{1/p}) \otimes x) = u\frac{\partial g}{\partial (t^{1/p})} \otimes x + t^{(p-1)/p} g \otimes \iota_{\partial_\epsilon} x \\ & 
\quad \text{for } g \in \bF_p[t^{1/p}], \; x \in \mathit{HH}_*(\scrA_{t,1/f})[u,u^{-1}];
\end{aligned}
\end{equation}
where $\iota_{\partial_\epsilon}$ denotes the action of \eqref{eq:partial-epsilon} on Hochschild homology. The first term in \eqref{eq:frobenius-connection} is the canonical connection on the Frobenius pullback; this works because we are in characteristic $p$, where $\partial_{t^{1/p}}(t) = pt^{(p-1)/p} = 0$. The second term in \eqref{eq:frobenius-connection} is an instance of the inverse Cartier transform from \cite[Section 1.1]{petrov-vaintrob-vologodsky17}; the power of $t$ which appears there comes from the elementary computation
\begin{equation}
{\textstyle\frac{1}{p}} dt = t^{(p-1)/p} d(t^{1/p}).
\end{equation}
Compatibility between the operations $\nabla_{u\partial_{t^{1/p}}}$ and $\nabla_{u\partial_t}$ on the two sides of \eqref{eq:fontaine-laffaille} is part of the statement of \cite[Theorem 1.1]{petrov-vaintrob-vologodsky17}. One can simplify the formulation by looking at the restriction of \eqref{eq:fontaine-laffaille} to $g = 1$ and $u^0$. That restriction is a graded map
\begin{equation} \label{eq:fontaine-laffaille-2}
\Phi: \mathit{HH}_*(\scrA_{t,1/f}) \longrightarrow \mathit{HP}_*(\scrA_{t,1/f}), 
\end{equation}
which is injective, and whose image satisfies 
\begin{equation} \label{eq:fl-property-1}
\bF_p[u,u^{-1}] \otimes_{\bF_p} \big(\mathit{im}(\Phi) \oplus t\,\mathit{im}(\Phi) \oplus \cdots
\oplus t^{p-1}\,\mathit{im}(\Phi)\big) = \mathit{HP}_*(\scrA_{t,1/f}).
\end{equation}
Compatibility with the $t$-action and connections can be expressed as
\begin{equation} \label{eq:fl-property-2}
\begin{aligned}
& t^p \Phi(x) = \Phi(tx), \quad
& \nabla_{u\partial_t} \Phi(x) = t^{p-1} \Phi(\iota_{\partial_\epsilon} x).
\end{aligned}
\end{equation}
The second part of \eqref{eq:fl-property-2} is obtained by starting with \eqref{eq:frobenius-connection} and writing 
\begin{equation}
\nabla_{u\partial_t} \Phi(1 \otimes x) = \Phi(\nabla_{u\partial_t^{1/p}}(1 \otimes x)) =
\Phi(t^{(p-1)/p} \otimes \iota_{\partial_\epsilon} x) = t^{p-1} \Phi(1 \otimes \iota_{\partial_\epsilon} x).
\end{equation}

We need one more ingredient, the categorical Fourier transform. The first version, for Hochschild homology \cite[Proposition 3.1.6]{pomerleano-seidel23}, says that the Hochschild complex of $\scrA_t$ is quasi-isomorphic to \eqref{eq:negative-q}. Passing to cohomology and inverting $f(t)$, one finds that 
\begin{equation} \label{eq:categorical-transform-1}
\mathit{HH}_*(\scrA_{t,1/f}) \iso H^{*-2}(q^{-1}A_{q^{-1},1/f}).
\end{equation}
Under that isomorphism, 
\begin{equation} \label{eq:categorical-transform-1b}
\parbox{37em}{
$-\iota_{\partial_\epsilon}$ (on the left) corresponds to the action of $q$ (on the right). Similarly, the action of $t$ corresponds to $\iota_W$.
}
\end{equation}
There is also a categorical Fourier transform for cyclic homology \cite[Theorem 3.1.14]{pomerleano-seidel23}. It says that the negative cyclic complex of $\scrA_t$ is quasi-isomorphic to \eqref{eq:negative-q-2}. A slightly more tricky algebraic argument \cite[Proof of Corollary 3.1.15]{pomerleano-seidel23} allows us to pass to a version with $f(t)$ inverted. Finally, inverting $u$ as well gives us a description of the periodic cyclic homology,
\begin{equation} \label{eq:categorical-transform-2}
\mathit{HP}_*(\scrA_{t,1/f}) \iso H^{*-2}(q^{-1}A_{u^{\pm 1},q^{-1},1/f}),
\end{equation}
such that:
\begin{equation} \label{eq:categorical-transform-2b}
\parbox{37em}{
$-\nabla_{u\partial_t}$ corresponds to the action of $q$; and the action of $t$ corresponds to $\nabla_{u\partial_q}$.}
\end{equation}
One applies \eqref{eq:categorical-transform-1} and \eqref{eq:categorical-transform-2} to the source and target of \eqref{eq:fontaine-laffaille-2}, which yields \eqref{eq:dg-fontaine-laffaille-1}. In view of \eqref{eq:categorical-transform-1b} and \eqref{eq:categorical-transform-2b}, the property \eqref{eq:fl-property-2} translates into \eqref{eq:dg-fontaine-laffaille-1b}; and \eqref{eq:fl-property-1} into \eqref{eq:dg-fontaine-laffaille-1c}.
\end{proof}

\subsection{$A_\infty$-deformations\label{subsec:ainfty-deformations}}
We need to apply the construction above to a wider class of deformations. This is done via a suitable quasi-isomorphism. The argument follows \cite[Sections 3.2--3.3]{pomerleano-seidel23}, but several points need to revisited so that they work in our context.

\begin{lemma} \label{th:general-1}
(Replaces the first part of \cite[Lemma 3.2.5]{pomerleano-seidel23}, as well as the use of the Yoneda embedding in the proof of \cite[Proposition 3.37]{pomerleano-seidel23}) Let $\scrA$ be an $A_\infty$-category (over some ring). Then there is a dg category $\scrB$ and a quasi-isomorphism $\scrA \rightarrow \scrB$ (over the same ring).
\end{lemma}

\begin{proof}
The Yoneda embedding yields a quasi-isomorphic $A_\infty$-functor $\scrA$ into the $A_\infty$-module category $\scrC = \mathit{mod}(\scrA)$ \cite[Corollary A.8]{lyubashenko-manzyuk08b}. The underlying chain maps are homotopy equivalences (true in general, but particularly easy to see in our setup because of the $K$-projectivity assumption on $\scrA$). Hence, one can use transfer (see e.g.\ \cite{markl04}) to find an inverse quasi-isomorphism $\scrC \rightarrow \scrA$.

Strictly speaking, $\scrC$ is not a dg category in our sense: because the definition of morphism spaces involves infinite direct products, the resulting chain complexes may not be $K$-projective (this was not a problem for the transfer argument cited above, since that involves only explicit formulae). To fix that, one uses a quasi-isomorphism $\scrB \rightarrow \scrC$ purely in the dg world, where $\scrB$ is cofibrant in the sense of a suitable model category structure. Then, the morphism spaces in $\scrB$ are themselves cofibrant complexes, meaning $K$-projective. (For an exposition see e.g.\ \cite[Lemmas 3.6--3.7]{belmans13} or \cite[Remark 4.7]{canonaco-ornaghi-stellari24}. As pointed out in the second reference, readers wishing to avoid the abstract theory of cofibrant replacements may instead use the more classical quasi-free resolutions of \cite[Lemma 13.5]{drinfeld02}.) Finally, one combines the two steps into a quasi-isomorphism $\scrB \rightarrow \scrA$ and then, to get the result as stated, again uses transfer to reverse direction.
\end{proof}

An $A_\infty$-deformation $\scrA_q$ of an $A_\infty$-category $\scrA$ is given by deformations of the $A_\infty$-operations, $\mu_{\scrA_q}^d = \mu_{\scrA}^d + q\mu_{\scrA_q}^{d,1} + \cdots$; this includes a curvature term $d = 0$, whose $q$-constant part is trivial. Differentiating in $q$-direction gives a Hochschild cocycle, whose cohomology class is called the Kaledin class,
 \begin{equation} \label{eq:kaledin-class}
[\partial_q \mu_{\scrA_q}] \in \mathit{HH}^0(\scrA_q).
\end{equation}

\begin{lemma} \label{th:general-2}
(Replaces the second part of \cite[Lemma 3.2.5]{pomerleano-seidel23})
We again work over an arbitrary coefficient ring. Take a quasi-isomorphism $\scrF: \scrA \rightarrow \scrB$ of $A_\infty$-categories, where $\scrB$ is strictly unital. Given a deformation $\scrA_q$ of $\scrA$, there is a strictly unital deformation $\scrB_q$ of $\scrB$ and a filtered quasi-isomorphism $\scrF_q: \scrA_q \rightarrow \scrB_q$, which extends the previous one.
\end{lemma}

\begin{proof}[Sketch of proof]
Let $C^*(\scrB)$ be the chain complex underlying the Hochschild cohomology $\mathit{HH}^*(\scrB)$, and $\bar{C}^*(\scrB) \subset C^*(\scrB)$ its reduced (unital) version. Through $\scrF$, one can consider $\scrB$ as an $A_\infty$-bimodule over $\scrA$. There is a corresponding Hochschild complex $C^*(\scrA,\scrB)$, underlying $\mathit{HH}^*(\scrA,\scrB)$, the Hochschild cohomology with coefficients in the $\scrA$-bimodule $\scrB$. This comes with a natural pullback map $\scrF^*: C^*(\scrB) \rightarrow C^*(\scrA,\scrB)$. The construction of both $\scrB_q$ and $\scrF_q$ proceeds order by order in $q$. At each step, the obstruction is a cohomology class in the cone complex
\begin{equation}
\mathit{Cone}\big( \bar{C}^*(\scrB) \xrightarrow{\scrF^*} C^*(\scrA,\scrB) \big);
\end{equation}
In our situation $\scrF^*$ is a quasi-isomorphism, so the cone is acyclic.
\end{proof}

From now on, consider an $A_\infty$-algebra $\scrA$ over $\bF_p$, and a deformation $\scrA_q$. We take the chain complex underlying the Hochschild homology of $\scrA_q$; then introduce a $q$-negative version as in \eqref{eq:negative-q}; and then invert $f$ as in \eqref{eq:localise-t}, but using the action of \eqref{eq:kaledin-class} instead of the previous $W$ (which is indeed the correct generalization). One similarly introduces modified versions of the chain complex underlying negative cyclic homology, using the $A_\infty$-version of the Getzler-Gauss-Manin connection. For all of these complexes, we use the same notation as before: starting with $A_q$ and $A_{u,q}$ for the Hochschild and cyclic complexes, and then as in \eqref{eq:invert-q}, \eqref{eq:invert-f}--\eqref{eq:localise-t-3} for the modified versions.

\begin{corollary} \label{th:fontaine-laffaille-2}
Let $p$ be an odd prime; $\scrA$ an $A_\infty$-algebra over $\bF_p$, with a deformation $\scrA_q$; and $f(t) \in \bF_p[t]$ a nonzero polynomial. Suppose that:

(i) $\scrA$ and its deformation $\scrA_q$ admit lifts $\tilde{\scrA}$, $\tilde{\scrA}_q$ to $\bZ/p^2$. 

(ii) $H^*(\scrA)$ is of finite rank over $\bF_p[t]$, where $t$ acts by multiplication with $[\mu_{\scrA_q}^{0,1}] \in H^0(\scrA^0)$.

(iii) $\scrA$ is homologically smooth over $\bF_p$. 

(iv) The class \eqref{eq:kaledin-class} satisfies $q^B f([\partial_q\mu_{\scrA_q}]) = 0 \in \mathit{HH}^{2B}(\scrA_q)$, for some $B \geq 0$.

(v) As in the corresponding assumption for Theorem \ref{th:fontaine-laffaille}.

(vi) $\mathit{HH}^*(\scrA) = 0$ for $\ast < 0$.

Then, the same conclusion as in Theorem \ref{th:fontaine-laffaille} holds, replacing $\iota_W$ by the action of \eqref{eq:kaledin-class}.
\end{corollary}

\begin{proof} This follows the same strategy as \cite[Corollary 3.3.9]{pomerleano-seidel23}. All the statements are quasi-isomorphism invariant. Therefore, in view of Lemma \ref{th:general-1} applied with the coefficient ring $\bZ/p^2$, we may replace the given $\tilde{\scrA}$, and its mod $p$ reduction $\scrA$, by strictly unital $A_\infty$-structures (in fact dg structures, but that's not relevant at this point). Next, we apply Lemma \ref{th:general-2}, again with $\bZ/p^2$ as a ring, to transfer the given deformation to the strictly unital version. To save on notation, let's therefore assume from now on that all structures involved were strictly unital in the first place.

The next step is to apply a $\bZ/p^2$-coefficient version of \cite[Proposition 3.3.8]{pomerleano-seidel23} to $\tilde{\scrA}$ (this needs strict unitality, which is why we had to arrange that first). The assumption here is that the Hochschild cohomology $\mathit{HH}^*(\tilde{\scrA})$ is zero in negative degrees. We know this holds for $\scrA$, by (vi); the desired property then follows from the long exact sequence
\begin{equation} \label{eq:coefficient-les}
\cdots \rightarrow \mathit{HH}^*(\scrA) \stackrel{p}{\longrightarrow} \mathit{HH}^*(\tilde{\scrA}) \xrightarrow{\text{mod $p$ reduction}} \mathit{HH}^*(\scrA) \rightarrow \cdots
\end{equation}
The proof of \cite[Proposition 3.3.8]{pomerleano-seidel23} includes an appeal to the Yoneda embedding, which we replace with a second use of Lemma \ref{th:general-1}, this time over the ring $(\bZ/p^2)[t]$. Otherwise, the argument goes through as written. The outcome is a quasi-isomorphism which reduces us to the situation in Theorem \ref{th:fontaine-laffaille}. 
\end{proof}

\begin{remark}
In our application, all algebraic structures come from Floer theory, and are in fact defined over $\bZ$. One could use that to simplify some of the arguments above. The rings $\bZ$, $\bZ[t]$, and $\bZ[t,1/f(t)]$ all have finite global dimension, so one could use unbounded complexes of projective modules instead of the more general homotopically projective notion. Also, there is no real need for the step involving \eqref{eq:coefficient-les}, since the Hochschild cohomology with $\bZ/p^2$-coefficients can be determined geometrically through the closed-open string map. 
We have preferred to give a purely algebraic discussion in this section, because that is more self-contained.
\end{remark}

\section{Symplectic cohomology and its deformation\label{sec:sh}}
This section carries out the required constructions in Hamiltonian Floer theory. The first four parts are preparatory, summarizing material from \cite{pomerleano-seidel24}. 
Based on that, we introduce the map \eqref{eq:circ-module}, a deformed version of the Floer-theoretic ``cap product''; show that it turns deformed symplectic cohomology into a module over quantum cohomology; and finally, relate this structure to the endomorphism \eqref{eq:kappa-map}, by proving Lemma \ref{th:kappa-and-quantum-product}, which is then used to prove Corollary \ref{th:torsion-1b}. The convention throughout most of this section is that all constructions are carried out with $\bZ$-coefficients; only at the end do we switch to coefficients $\bC$ or $\bF_p$ (for the proof of Corollary \ref{th:torsion-1}), and a mod $p$ version of it).
%

\subsection{Symplectic cohomology}
We recall the Floer-theoretic setup from \cite[Section 2]{pomerleano-seidel24}. 
\begin{equation}
\parbox{37em}{$D \subset M$ is a smooth symplectic hypersurface Poincar{\'e} dual to $c_1(M)$.}
\end{equation}
(This is weaker than asking for a smooth symplectic divisor, but sufficient for the present discussion.) In a neighbourhood of $D$ there is a Hamiltonian circle action $(\rho_t)$ which fixes $D$ pointwise, and rotates the normal bundle anticlockwise. We denote its moment map by $h$ (the normalization conventions are that $\rho_t$ rotates the normal bundle with angle $2\pi t$; and $(\partial_t\rho_t)_{t=0}$ is the Hamiltonian vector field of $h$).
%
%

\begin{definition} \label{defn:complexclass} 
(i) Take $\sigma \in (\mathbb{Q} \setminus \mathbb{Z})_{>0} = (\mathbb{Q} \setminus \mathbb{Z}) \cap \mathbb{R}_{>0}$.  We say that a function $H: M \to \mathbb{R}$ is of \emph{slope} $\sigma > 0$ if, in some neighborhood of $D$, $H + \sigma h$ is constant.  

(ii) A compatible almost complex structure $J$ is \emph{locally $S^1$-invariant} if $(\rho_t)$ preserves $J$ in some neighborhood of $D$. This implies that $D$ is a $J$-complex submanifold.
\end{definition}

Let $\bar{H} = (\bar{H}_t): S^1 \times M \to \mathbb{R}$ be a time-dependent Hamiltonian which is of the same slope $\sigma$ for all times $t \in S^1$. We assume that: 
\begin{align} \label{eq:genericHams} 
\parbox{37em}{the one-periodic orbits $x: S^1 \rightarrow M$ of $\overline{H}$ which lie in $M \setminus D$, are nondegenerate (additionally, there are constant periodic orbits at every point of $D$, which form a Morse-Bott family).}
\end{align} 
Choose a $t$-dependent almost complex structure $\bar{J} = (\bar{J}_t)$ which is locally $S^1$-invariant for each $t$. Fixing two periodic orbits $x^\pm$, the Floer equation is
\begin{equation}
\label{eq:floer}
\begin{cases}
u = u(s,t): \mathbb{R} \times S^1 \longrightarrow M, \\
\lim_{s \to \pm \infty} u = x^\pm, \\
\partial_s u + \bar{J}_t (\partial_t u - X_{\bar{H}_t}) = 0.
\end{cases}
\end{equation} 
The Floer complex $\mathit{CF}(\bar{H},\bar{J})$ is freely generated by one-periodic orbits which lie in $M \setminus D$. For a sufficiently generic choice of $\bar{J}$, the Floer differential $\delta$ counts Floer trajectories in $M \setminus D$ (solutions of \eqref{eq:floer} contribute to the coefficient of $\delta$ which sends $x^+$ to $x^-$). Choose, once and for all, a $\smooth$ section of the anticanonical bundle of $M$, for some compatible almost complex structure, which cuts out $D$ positively and transversally. Then, each generator $x$ carries a Conley-Zehnder index $\mathrm{deg}(x)$, which gives our complex a $\bZ$-grading. Note that since the space $M \setminus D$ is non-compact, and the Hamiltonian has additional periodic orbits in $D$, some analysis is required to show that the standard compactness results hold. This relies on positivity of intersections; we refer to \cite[Section 2.1--2.3]{pomerleano-seidel24} for details. 

Suppose we are given time-dependent Hamiltonians $\bar{H}^\pm = (\bar{H}^{\pm}_t)$ of slopes $\sigma^{\pm}$, where $\sigma^- \geq \sigma^+$, and almost complex structures $\bar{J}^{\pm} = (\bar{J}^\pm_t)$, as in Definition \ref{defn:complexclass}, which are  sufficiently generic for defining the Floer complexes. We then choose an interpolating family of Hamiltonians $H = (H_{s,t})$ and compatible almost complex structures $J = (J_{s,t})$, such that:
\begin{itemize} \itemsep .5 em 
\item Each $H_{s,t}$ is constant along $D$, and has vanishing derivative along $D$; this ensures that the corresponding Hamiltonian vector field $X_{H_{s,t}}$ is zero along $D$. Each $J_{s,t}$ makes $D$ into an almost complex submanifold. (Note that these conditions are weaker than those in Definition \ref{defn:complexclass}.) 

\item $H_{s,t} = \bar{H}^\pm_t$ and $J_{s,t} = \bar{J}^\pm_t$ for $\pm s \gg 0$. 
\end{itemize} 
The continuation map equation replaces the last line in \eqref{eq:floer} with
\begin{equation}
\label{eq:continuation}
\partial_s u + J_{s,t} (\partial_t u - X_{H_{s,t}}) = 0.
\end{equation}
For sufficiently generic $(H, J)$, counting solutions to \eqref{eq:continuation} defines a continuation map
\begin{equation} 
\mathit{CF}^*(\bar{H}^{+},\bar{J}^+) \longrightarrow \mathit{CF}^*(\bar{H}^{-},\bar{J}^-). 
\end{equation}
We again refer to \cite[Section 2.3]{pomerleano-seidel24} for compactness of the moduli spaces. As a consequence of standard arguments about continuation maps and their compositions, Floer cohomology at any given slope is well-defined up to canonical isomorphism:
\begin{equation} 
\mathit{HF}^*(\sigma) \stackrel{\mathrm{def}}{=} H^*(\mathit{CF}^*(\bar{H},\bar{J})), 
\end{equation} 
where $\bar{H}$ is any Hamiltonian of slope $\sigma$ satisfying \eqref{eq:genericHams}, and $\bar{J}$ is generic. Moreover, these Floer groups come with canonical continuation maps increasing the slope. To define symplectic cohomology, one takes the slope to infinity: 
\begin{align} \label{eq:directlimit} 
\mathit{SH}^*(M\setminus D) \stackrel{\mathrm{def}}{=} \varinjlim_\sigma\, \mathit{HF}^*(\sigma). 
\end{align}

We need a chain level version of \eqref{eq:directlimit}. Fix an increasing sequence of slopes $\sigma_w$, $w = 0,1,\dots$, which go to $\infty$. For each $w$, we choose a pair $(\bar{H}_w,\bar{J}_w)$ with slope $\sigma_w$, which defines the Floer complex $\mathit{CF}^*(w) = \mathit{CF}^*(\bar{H}_w,\bar{J}_w)$. The telescope construction is the cochain complex
\begin{equation} \label{eq:telescope}
C = \Big( \bigoplus_{w \geq 0} CF(w) \Big) \oplus \eta \Big( \bigoplus_{w \geq 0} \mathit{CF}(w) \Big),
\end{equation}
where $\eta$ is a formal variable of degree $-1$. The differential $d_C$ is a sum of three components: 
\begin{itemize} \itemsep .5 em  
\item  On each piece \( \mathit{CF}(w) \), we have the Floer differential $\delta$. On each piece \( \eta \mathit{CF}(w) \), we have the shifted version $-\delta$. 
\item For each slope, one chooses a continuation map
\begin{equation} \label{eq:w-continuation}
  \xymatrix{
    \eta \mathit{CF}^*(w) \ar[r] 
    & \mathit{CF}^*(w+1).
  }
\end{equation}
\item Finally, there is a purely ``algebraic" term $-\!\mathit{id}: 
\eta \mathit{CF}^*(w) \rightarrow \mathit{CF}^*(w)$.
\end{itemize} 
On cohomology one recovers \eqref{eq:directlimit}, $H^*(C,d_C) = \mathit{SH}^*(M \setminus D)$.

\subsection{Deformed symplectic cohomology}
Assume that all data needed to define the telescope complex \eqref{eq:telescope} have been fixed, where the slopes satisfy the following more specific condition: 
\begin{equation}
\label{eq:integerinbetween}
\text{there is at least one integer between } \sigma_w \text{ and } \sigma_{w+1}.
\end{equation}
We wish to construct a formal deformation
\begin{equation}
\label{eq:Cq}
C_q = C[[q]], \quad d_{C_q} = d_C + O(q): C_q \longrightarrow C_q
\end{equation}
where $q$, as usual, is a variable of degree $2$. The construction involves certain parameter spaces $\overline{\frakD}_m$, $m \geq 1$, which we now recall. Consider divisors of degree $m$ on the cylinder, which means (unordered) collections of marked points with multiplicities, written as
\begin{equation} \label{eq:divisor}
\Sigma = (\Sigma_z)_{z \in \mathbb{R} \times S^1} \quad \text{with } \Sigma_z \geq 0,\quad \sum_z \Sigma_z = m.
\end{equation}
Such divisors, up to $\bR$-translation, are parametrized by
\begin{equation}
\frakD_m = \mathit{Sym}^m(\mathbb{R} \times S^1)/\mathbb{R}.
\end{equation}
Given a partition $\Pi$ of $m$, let $\frakD_\Pi \subset \frakD_m$ be the locus where the marked points in $\Sigma$ coincide exactly as dictated by $\Pi$; it is is a locally closed submanifold with $\mathrm{dim}\,\frakD_\Pi = 2|\Pi|-1$. These loci form a stratification of $\frakD_m$. The main stratum, associated to $\Pi = (1,\dots,1)$, is the open subset where $\Sigma$ consists of $m$ pairwise distinct points. 

One compactifies $\frakD_m$ to a manifold with corners $\overline{\frakD}_m$, by allowing the cylinder to break into several components, each carrying a nontrivial divisor:
\begin{equation} 
\overline{\frakD}_m = \!\!\!\! \coprod_{\substack{R \geq 1 \\ m^1 + \cdots + m^R = m}} \!\!\!\! \frakD_{m^1} \times \cdots \times \frakD_{m^R}. \label{eq:thecompactificationDm}
\end{equation} 

Over $\frakD_m$ there is a universal curve: a fibration whose fibres are cylinders, identified with $\bR \times S^1$ up to translation, and carrying a divisor \eqref{eq:divisor}. The universal family extends to $\overline{\frakD}_m$; over each product in \eqref{eq:thecompactificationDm}, the fibres of this extension are disjoint unions of the corresponding universal curves over the $\frakD_{m^k}$ factors.

Fix $w\geq 0$, $m>0$. On the universal curve over $\frakD_m$, choose continuation map data which agree with 
\begin{equation} \label{eq:fibrewise-asymptotics}
(\bar{H}_w, \bar{J}_w) \text{ for $s \gg 0$, respectively with } (\bar{H}_{w+m}, \bar{J}_{w+m}) \text{ for $s \ll 0$}.
\end{equation}
For the purposes of transversality, we require that the region over which \eqref{eq:fibrewise-asymptotics} holds, on each fibre $C$ of the universal curve, is disjoint from the divisor $\Sigma$. These data should extend smoothly to $\overline{\frakD}_m$, compatibly with the boundary decomposition \eqref{eq:thecompactificationDm} (see \cite[Section 2.3]{pomerleano-seidel24} for a precise formulation of this consistency condition). 
%
%
Take periodic orbits $x^{+}$ for $\bar{H}_w$ and $x^{-}$ for $\bar{H}_{w+m}$, both in $M \setminus D$. We consider the space $\frakD_m(x^-,x^+)$ of pairs $(r,u)$, where: \begin{itemize} 
\itemsep.5em 
\item $r$ is a point in $\frakD_m$. Write $C$ for the fibre of the universal curve at that point, and let $\Sigma$ be its natural divisor;
\item $u:C \to M$ is a solution to the continuation equation for the chosen datum on $C$, with limits $x^{\pm}$, and which satisfies
\begin{equation} \label{eq:intersectd}
u^{-1}(D) = \Sigma \text{ as an equality of divisors.} 
\end{equation}
This means that $u^{-1}(D)$ is the support of $\Sigma$, and that the intersection multiplicity at any point $z$ is equal to $\Sigma_z$.
\end{itemize}
The space $\frakD_m(x^-,x^+)$ has expected dimension 
\begin{align} 
\label{eq:vdim} 
\operatorname{deg}(x^{-})-\operatorname{deg}(x^{+})+2m-1.
\end{align}
When setting up the analysis for these spaces, one actually considers each stratum $\frakD_{\Pi}$, and the associated condition \eqref{eq:intersectd}, separately. If \eqref{eq:vdim} is $\leq 1$, which is the situation that matters to us, transversality will ensure that only the main stratum gives a nonempty moduli space. 

Next, we state the precise transversality requirements placed on our choice of data (essentially the same as in \cite[Section 2.3]{pomerleano-seidel24}, even though we have chosen a different order of presentation). We will not explain how these technical conditions enter into the construction; they are recorded for reference in Section \ref{sec:modulestructure} below.

\begin{condition}[Stratification] \label{th:stratification}
Take a partition $\Pi$ of $m$. Consider all $(r,u)$ where $r \in \frakD_\Pi$, and which satisfy a weakened version of \eqref{eq:intersectd}: 
\begin{equation} \label{eq:mult-leq}
u^{-1}(D) \leq \Sigma \text{ as divisors.}
\end{equation}
(In other words, the intersection multiplicity of $u$ and $D$ at any point $z$ is $\leq \Sigma_z$. This includes allowing multiplicity zero, which is when $u(z) \notin D$. Hence, part of \eqref{eq:mult-leq} is that $u^{-1}(D)$ is a subset of, but not necessarily equal to, the support of $\Sigma$.) For every $\Pi$ and every choice of intersection multiplicities, we require that the space is regular, hence has the expected dimension 
\begin{equation}
\mathrm{deg}(x^-) - \mathrm{deg}(x^+) + 2|\Pi| - 1.
\end{equation}
\end{condition}

\begin{condition}[Generic bubbling] \label{th:bubble}
Consider the main stratum (where all points of $\Sigma$ are distinct). Look at solutions $(r,u)$ in this stratum, satisfying \eqref{eq:mult-leq}, which means: at each point $z$ in $\Sigma$, we either have $u(z) \notin D$, or else $u$ intersects $D$ transversally (with multiplicity $1$). In addition, we require that there should be a point $z$ in $\Sigma$ with $u(z) \notin D$, and a sphere bubble $v: \bC P^1 \rightarrow M$, pseudo-holomorphic for the almost complex structure associated to $(r,z)$, such that $v \cdot D = 1$, and $u(z)$ lies on the image of $v$ (which is a codimension $2$ condition). These spaces should be empty whenever \eqref{eq:vdim} is $\leq 1$.
\end{condition}

Under these conditions, the zero-dimensional spaces $\frakD_m(x^-,x^+)$ are regular and finite \cite[Proposition 2.3.2(ii)]{pomerleano-seidel24}. Counting points in those moduli spaces gives rise to operations
\begin{equation} \label{eq:dm-map}
d_m: \mathit{CF}^*(\bar{H}_w,\bar{J}_w) \longrightarrow CF^{*+1-2m}(\bar{H}_{w+m},\bar{J}_{w+m}).
\end{equation} 
We extend this by taking $d_0 = \delta$ the Floer differential on $\mathit{CF}(w)$. Then \cite[Proposition 2.3.2(iii)]{pomerleano-seidel24} 
\begin{equation} 
\label{eq:dsquarestozero} 
\sum_{i+j = m} d_i d_{j}=0.
\end{equation} 
Geometrically, the terms with $i,j>0$ come from the codimension one faces of \eqref{eq:thecompactificationDm}, and the remaining ones from splitting off of Floer trajectories. 

We will also need a variant, which starts with the space parametrizing divisors \eqref{eq:divisor} without dividing by translation,
\begin{equation}
\frakD_m^{\dag} = \mathit{Sym}^m(\mathbb{R} \times S^1). 
\end{equation}
The corresponding compactification is
\begin{equation} \label{eq:c-strata}
\bar\frakD_m^\dag = \!\!\!\! \coprod_{\substack{R \geq 1,\; 1 \leq c \leq R \\ m^1 + \cdots + m^R = m}}
\frakD_{m^1} \times \cdots \times \frakD_{m^c}^\dag \times \cdots \times \frakD_{m^R}.
\end{equation}
We again choose continuation map data on the universal curve, with the analogue of \eqref{eq:fibrewise-asymptotics} being that they agree with
\begin{align} \label{eq:avoidrepetition} 
(\bar{H}_w,\bar{J}_w) \text{ for } s \gg 0, \text{ and with } (\bar{H}_{m+w+1},\bar{J}_{m+w+1}) \text{ for } s \ll 0. 
\end{align} 
As before, the data are chosen to extend smoothly over the compactification, compatibly with the boundary decompositions (hence with the previous choices over $\frakD_m$) in \eqref{eq:c-strata}. In the special case $m = 0$, where $\frakD_m^\dag = \mathit{point}$, we use the data previously used to define the maps \eqref{eq:w-continuation} in the telescope construction. 
Generically, counting points in the resulting zero-dimensional spaces $\frakD_m^\dag(x^-,x^+)$ defines operations
\begin{equation} 
d_m^{\dag}: \mathit{CF}^*(\bar{H}_w,\bar{J}_w) \longrightarrow CF^{*-2m}(\bar{H}_{w+m+1},\bar{J}_{w+m+1})
\end{equation}
which for $m = 0$ agree with \eqref{eq:w-continuation}. These satisfy
\begin{equation} \label{eq:cm-equation}
\sum_{i+j=m} d_i^\dag d_{j} - d_i d_{j}^\dag = 0;
\end{equation}
except for the Floer cylinder bubbling ($i = 0$ or $j = 0$), the two possible terms in \eqref{eq:cm-equation} reflect the fact that the $\frakD^\dag$ factor can appear either on the left or right in the product structure of the boundary faces \eqref{eq:c-strata} of codimension $1$ (with $R = 2$).

Finally, we assemble these operations into the differential \eqref{eq:Cq}:
\begin{itemize} \itemsep .5 em 
\item  on each piece \( \mathit{CF}(w) \) the maps $d_m$ contribute a term 
\begin{equation} 
q^m d_m: \mathit{CF}(w)[[q]] \longrightarrow \mathit{CF}(w+m)[[q]]
\end{equation} 
Correspondingly, we have $-q^m d_m: \eta\mathit{CF}(w)[[q]] \rightarrow \eta\mathit{CF}(w+m)[[q]]$.

\item The operations $d_m^{\dag}$ contribute terms 
\begin{equation} 
q^m d_m^\dag: \eta \mathit{CF}(w)[[q]] \longrightarrow \mathit{CF}(w+m+1)[[q]]. 
\end{equation}

\item We use the same algebraic component $-\!\mathit{id}: \eta \mathit{CF}^*(w)[[q]] \to \mathit{CF}^*(w)[[q]]$ as before.
\end{itemize}
The relations \eqref{eq:dsquarestozero} and \eqref{eq:cm-equation} imply $d_{C_q}^2 = 0$, and 
\begin{equation} 
\mathit{SH}^*_q(M,D) \stackrel{\mathrm{def}}{=} H^*(C_q,d_{C_{q}}). 
\end{equation}

\subsection{The equivariant version\label{subseq:equivariant}}
$S^1$-equivariant deformed symplectic cohomology $\mathit{SH}^*_{u,q}(M,D)$ is constructed in \cite[Section 7]{pomerleano-seidel24}. Here, we will only give a brief reminder, so as to fix the notation. As usual, $u$ is a formal variable of degree $2$. The (undeformed) symplectic cohomology has an $S^1$-equivariant version, where $S^1$ rotates the parametrization of the circle. The underlying graded $\bZ[[u]]$-module is
\begin{equation}
C_u = C[[u]], \quad d_{C_u} = d_C + O(u): C_u \longrightarrow C_u.
\end{equation}
The $u^1$ term of the differential is constructed from two pieces of data, which are maps
\begin{align}
\label{eq:bv}
& d_0^1: \mathit{CF}^*(w) \longrightarrow \mathit{CF}^{*-1}(w), \\
\label{eq:bv-2}
& d_0^{1,\dag}: \mathit{CF}^*(w) \longrightarrow \mathit{CF}^{*-2}(w+1).
\end{align}
The more familiar \eqref{eq:bv} is the BV or loop rotation operator for the time-dependent Hamiltonian $\bar{H}_w = \bar{H}_{w,t}$. It is defined using a family of continuation map equations parametrized by $\theta \in S^1$, whose asymptotic behaviour is
\begin{equation} \label{eq:theta-rotation}
(\bar{H}_{w,t},\bar{J}_{w,t}) \text{ for } s \gg 0, \text{ and } (\bar{H}_{w,t-\theta},\bar{J}_{w,t-\theta}) \text{ for } s \ll 0. 
\end{equation}
The condition on limits of solutions is modified accordingly, to $\lim_{s \rightarrow -\infty} u(s,t) = x^-(t-\theta)$. The BV operator commutes with continuation maps up to chain homotopy, and \eqref{eq:bv-2} is precisely that homotopy. In the next step, one has that
\begin{equation}
d_0^1 d_0^1: \mathit{CF}^*(w) \longrightarrow \mathit{CF}^{*-2}(w)
\end{equation}
is nullhomotopic. The construction of the homotopy involves a parameter space which (when compactified) is a solid torus, whose boundary $S^1 \times S^1$ corresponds to the product of two copies of the space underlying $d_0^1$. The outcome is part of the $u^2$ term of $d_{C_u}$, and so on.

Deformed $S^1$-equivariant symplectic cohomology combines that idea with the $q$-deformation. The underlying graded $\bZ[[u,q]]$-module is 
\begin{equation}
C_{u,q} = C_q[[u]] = C_u[[q]] = C[[u,q]], \quad d_{C_{u,q}} = d_{C_q} + O(u) = d_{C_u} + O(q):
C_{u,q} \longrightarrow C_{u,q}.
\end{equation}
We'll only look at the first piece of new geometric information. Consider the compositions
\begin{equation} \label{eq:u1q1}
d_1 d_0^1, \;\; d_0^1 d_1: \mathit{CF}^*(w) \longrightarrow \mathit{CF}^{*-2}(w+1),
\end{equation}
where $d_1$ is as in \eqref{eq:dm-map}. Each of them can be thought of as a single operation defined using cylinders broken into two pieces, and parameters $(t,\theta) \in S^1 \times S^1$. One piece carries a marked point (the unique point of $\Sigma$) which, breaking translation-invariance, is of the form $z = (0,t)$ for an arbitrary $t \in S^1$; this is $d_1$. The other piece, underlying $d_0^1$, is the one where the rotation of the asymptotic behaviour \eqref{eq:theta-rotation}, parametrized by $\theta$, takes place. Here's a schematic picture of the situation, as well as what happens when the two pieces are glued together:
\begin{equation}
\includegraphics{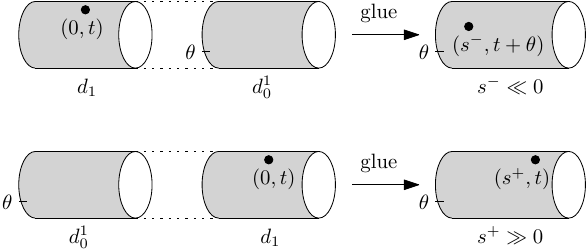}
\end{equation}
One can interpolate between the two compositions by an operation with parameter space $\bR \times S^1 \times S^1$ (its ends will be identified with the previous $S^1 \times S^1$ in two different ways, but that is unproblematic). The outcome is a chain homotopy between the two maps in \eqref{eq:u1q1}, denoted by
\begin{equation}
d_1^1: \mathit{CF}^*(w) \longrightarrow \mathit{CF}^{*-3}(w+1),
\end{equation}
which enters into the $u^1q^1$ terms of $d_{C_{u,q}}$.

\subsection{The endomorphism\label{sec:kappa}}
Our next task is to recall, from \cite[Section 8.1]{pomerleano-seidel24}, the construction of the endomorphism \eqref{eq:kappa-map}; geometrically, it is very similar to the definition of $d_{C_q}$. We use cylinders equipped with a divisor \eqref{eq:divisor}, which is allowed to be empty, together with one more point $z^* \in \bR \times \{0\}$. The pair $(\Sigma,z^*)$ is considered up to common $\bR$-translation, so the parameter spaces, for $m \geq 0$, are
\begin{equation}
\frakA_m = (\mathit{Sym}^m(\bR \times S^1) \times \bR)/\bR;
\end{equation}
they comes with compactifications $\bar\frakA_m$ as before. For each $w,m \geq 0$, we choose continuation map data on the universal curve over $\frakA_m$. The asymptotics are \eqref{eq:avoidrepetition}, and instead of \eqref{eq:intersectd} we ask that
\begin{equation} \label{eq:star-intersection}
u^{-1}(D) = \Sigma + z^*.
\end{equation}
In particular, $u(z^*) \in D$, and the intersection multiplicity there is $\Sigma_{z^*} + 1$. Let $\frakA_m(x^-,x^+)$ be the resulting moduli space of pairs $(r,u)$. The outcome of counting points in zero-dimensional moduli spaces are maps (for $w,m \geq 0$)
\begin{equation} \label{eq:iota-m-map}
\begin{aligned}
& a_m: \mathit{CF}^*(w) \longrightarrow \mathit{CF}^{*-2m}(w+m+1), \\
& \!\!\! \sum_{i+j = m} d_i a_j - a_i d_j = 0.
\end{aligned}
\end{equation}

\begin{remark} 
One can break translation-invariance by setting $z^* = (0,0)$. This gives an isomorphism $\frakA_m \iso \mathit{Sym}^m(\bR \times S^1) = \frakD_m^\dag$, and that extends to compactifications. This explains the similarity between \eqref{eq:cm-equation} and \eqref{eq:iota-m-map}. However, it would not be useful to make that identification of parameter spaces in our discussion, because each appears as part of a different larger structure.
\end{remark}

Again, there is a version where one does not divide by translation,
\begin{equation}
\frakA_m^\dag = \mathit{Sym}^m(\bR \times S^1) \times \bR.
\end{equation}
The compactification is
\begin{equation} \label{eq:a-dag-compactified}
\begin{aligned}
\bar\frakA_m^\dag & = \!\!\!\!
\coprod_{\substack{R \geq 1,\; 1 \leq c \leq R \\ m^1 + \cdots + m^R = m}} 
\frakD_{m^1} \times \cdots \times \frakA_{m^c}^\dag \times \cdots \times \frakD_{m^R}
\\ &
\sqcup \!\!\! \coprod_{\substack{R > 1,\; 1 \leq b < c \leq R \\ m^1 + \cdots + m^R = m}} 
\frakD_{m^1} \times \cdots \times \frakA_{m^b} \times \cdots \times \frakD_{m^c}^\dag \times \cdots \times \frakD_{m^R} 
\\ &
\sqcup \!\!\! \coprod_{\substack{R > 1,\; 1 \leq c < b\leq R \\ m^1 + \cdots + m^R = m}} \frakD_{m^1} \times \cdots \times \frakD_{m^c}^\dag \times \cdots \times \frakA_{m^b} \times \cdots \times \frakD_{m^R}. 
\end{aligned}
\end{equation}
The boundary faces in the second line of \eqref{eq:a-dag-compactified} appear as limits when $z^* \rightarrow -\infty$, and correspondingly for $z^* \rightarrow +\infty$ and the third line. For the choice of continuation data on the universal curve, one modifies \eqref{eq:avoidrepetition} to 
\begin{align} \label{eq:avoidrepetition2} 
(\bar{H}_w,\bar{J}_w) \text{ for } s \gg 0, \text{ and } (\bar{H}_{m+w+2},\bar{J}_{m+w+2}) \text{ for } s \ll 0. 
\end{align} 
Using these spaces one produces maps
\begin{equation} \label{eq:iota-m-map-2}
\begin{aligned}
& a_m^{\dag}: \mathit{CF}^*(w) \longrightarrow \mathit{CF}^{*-2m-1}(w+m+2), \\
& \!\!\! \sum_{i+j=m} d_i^\dag a_j - a_i d_j^\dag + a_i^\dag d_j + d_i a_j^\dag  = 0.
\end{aligned}
\end{equation}
This equation contains terms $d_0 = \delta$ corresponding to splitting off a Floer cylinder. The remaining terms come from the codimension $1$ faces in \eqref{eq:a-dag-compactified}.
\begin{itemize} \itemsep.5em
\item The faces of the form $\frakA_i \times \frakD_j^\dag$ respectively $\frakD_i^\dag \times \frakA_i$, which appear when a cylinder carrying the marked point $z^*$ splits off on the left or right, lead to terms $a_i d_j^\dag$ and $d_i^\dag a_j$.

\item If a cylinder splits off on the left or the right, carrying part of $\Sigma$ but not the marked point $z^{\ast}$, we get faces of the form $\frakD_i \times \frakA_j^\dag$ or $\frakA_i^{\dag} \times \frakD_j$, leading to terms $d_i a_j^\dag$, $a_i^\dag d_j$.
\end{itemize} 
Take
\begin{equation} \label{eq:define-a}
\begin{aligned}
& a_{C_q}: C_q \longrightarrow C_q, \\
& a_{C_q}(x) = \sum_m q^m a_m(x), \;\;
a_{C_q}(\eta x) = \sum_m q^m (\eta a_m(x) + a_m^\dag(x))
\quad \text{for } x \in \mathit{CF}^*(w)[[q]].
\end{aligned}
\end{equation}
It follows from \eqref{eq:iota-m-map}, \eqref{eq:iota-m-map-2} that this is a chain map. We define \eqref{eq:kappa-map} to be the induced map on cohomology.

\subsection{The module structure\label{sec:modulestructure}}
For the first time, our exposition moves beyond the literature (to a small extent: a construction analogous to the one we'll explain, but using the thimble instead of the cylinder, is in \cite[Section 3.1]{pomerleano-seidel24}). Therefore, we will give more details than before. The construction depends on a choice of pseudo-cycle:
\begin{equation} \label{eq:transverse-pseudo-cycle}
\parbox{37em}{
Let $c_P: P \rightarrow M$ be a pseudo-cycle which is transverse to $D$, in the following sense. First of all, the map $c_P$ is transverse to $D$. Let $\Omega_P \subset M$ be the limit set of $P$ \cite[Definition 6.5.1]{mcduff-salamon}. By definition, this is of dimension $\leq \mathrm{dim}(P) - 2$, which means that it is contained in the image of a smooth map from a manifold of that dimension to $D$. Our second condition is that the subset $\Omega_P \cap D \subset D$ should be of dimension $\leq \mathrm{dim}(P) - 4$, in the same sense; this is equivalent to the notion of ``weak transversality'' in \cite[Definition 6.5.10]{mcduff-salamon}.
}
\end{equation}
We need a bit of classical pseudo-holomorphic curve background (see e.g.\ \cite[Chapter 6]{mcduff-salamon} for this kind of argument in general, and \cite[Lemma 3.3.2]{pomerleano-seidel24} for the more specific application here).

\begin{lemma} \label{th:bubble-dim}
Fix an integer $e \geq 1$. For $J$ which makes $D$ into an almost complex submanifold, let $\bar\frakM_e(P)$ be the space of genus zero two-pointed $J$-holomorphic stable maps $v: (S,\zeta_0,\zeta_1) \rightarrow M$, with $v \cdot D = e$, such that $v(\zeta_1) \in \overline{c_P(P)}$. Consider the evaluation map at the other marked point,
\begin{equation} \label{eq:1-evaluation}
\bar\frakM_e(P) \longrightarrow M, \quad v \longmapsto v(\zeta_0).
\end{equation}
For generic $J$, the image of \eqref{eq:1-evaluation} is of dimension $\leq \mathrm{dim}(P)+2e-2$, in the same sense as in \eqref{eq:transverse-pseudo-cycle}. Moreover, its intersection with $D$ is of dimension $\leq \mathrm{dim}(P) +2e-4$.
\end{lemma}


We will use the same parameter spaces as in Section \ref{sec:kappa}, but with the adjacency condition at $z^*$ governed by \eqref{eq:transverse-pseudo-cycle}, leading to different moduli spaces, transversality conditions, and so on. To keep that distinction clear, we find it convenient to change the notation and to write
\begin{equation}
\frakB_m = \frakA_m, \quad \frakB_m^\dag = \frakA_m^\dag.
\end{equation}
Choose continuation data over $\frakB_m$, satisfying \eqref{eq:fibrewise-asymptotics}. We consider triples $(p,r,u)$, where $p \in P$, $r \in \frakB_m$, and $u: C \rightarrow M$ satisfies \eqref{eq:intersectd} as well as
\begin{equation} \label{eq:adjacency-condition}
u(z^*) = c_P(p).
\end{equation}
This gives rise to moduli spaces $\frakB_m(P,x^{-},x^{+})$ of expected dimension
\begin{align} 
\label{eq:mainstratumdim0} 
\operatorname{deg}(x^{-})-\operatorname{deg}(x^{+})+2m-\operatorname{codim}(P).
\end{align}


We now turn to the required transversality conditions (analogous to those in \cite[Section 3.1]{pomerleano-seidel24}). The first one is the counterpart of Condition \ref{th:stratification}.

\begin{condition}[Stratification] \label{th:a-stratification}
Take some $0 \leq m^* \leq m$, and a partition $\Pi$ of $m-m^*$. Consider the stratum of $\frakB_m$ where $\Sigma_{z^*} = m^*$, and where $\Pi$ describes how the remaining points of $\Sigma$ are grouped. We consider triples $(p,r,u)$ where $r$ lies in that stratum, and $u$ satisfies \eqref{eq:mult-leq} as well as \eqref{eq:adjacency-condition}. As in Condition \ref{th:stratification}, the requirement is that for each choice of intersection multiplicities, the resulting moduli space should be regular, hence of dimension $\mathrm{deg}(x^-) - \mathrm{deg}(x^+) + 2|\Pi| - \operatorname{codim}(P)$. 

There is a necessary addendum: by definition of pseudo-cycle, there are maps from manifolds of dimension $\leq \mathrm{dim}(P) - 2$, which cover the limit points of $c_P$. We impose the corresponding requirement on those maps (with the appropriate change in dimensions).
\end{condition}

To be entirely precise, there are two sub-cases in Condition \ref{th:a-stratification}: if $u(z^*) \notin D$, then \eqref{eq:adjacency-condition} is an intersection condition inside $M$; and if $u(z^*) \in D$ (which can only happen if $m^*>0$), we think of it as a condition inside $D$. Thanks to the transversality assumption on the pseudo-cycle \eqref{eq:transverse-pseudo-cycle}, the dimension is the same in both cases. 

\begin{condition}[Bubbling at $z^*$] \label{th:a-chain}
Take $0 < m^* \leq m$. We consider $(r,u,v)$, where: $r$ describes a divisor with $\Sigma_{z^*} = m^*$; and $v: (S,\zeta_0,\zeta_1) \rightarrow M$ is in $\bar\frakM_e(P)$, see Lemma \ref{th:bubble-dim}, where the relevant almost complex structure is that associated to the point $z^* \in C$. Instead of \eqref{eq:adjacency-condition}, these should satisfy the modified incidence 
\begin{equation} \label{eq:bubble-adjacency}
u(z^*) = v(\zeta_0).
\end{equation}
Moreover, we replace \eqref{eq:mult-leq} with 
\begin{equation} \label{eq:mult-leq-2}
u^{-1}(D) + (v \cdot D) z^* \leq \Sigma.
\end{equation}
(In words, at $z^*$, the multiplicity of intersection of $u$ and $D$ is at most $m^* - v \cdot D$.) We ask for the resulting space to be empty whenever \eqref{eq:mainstratumdim0} is $\leq 1$.
\end{condition}

As usual, the space of $(r,u)$ generically has dimension
\begin{equation}
\mathrm{deg}(x^-) - \mathrm{deg}(x^+) + 2|\Pi| \leq \mathrm{deg}(x^-) - \mathrm{deg}(x^+) + 2(m-m^*).
\end{equation}
By Lemma \ref{th:bubble-dim} and \eqref{eq:mult-leq-2}, the image of the evaluation map from the space of $(p,v)$ has generic dimension 
\begin{equation}
\leq \mathrm{dim}(P) + 2(v \cdot D) - 2 \leq 
\mathrm{dim}(P) + 2m^* - 2. 
\end{equation}
Taking into account \eqref{eq:bubble-adjacency} and transversality of evaluations for $(r,u) \mapsto u(z^*)$, one arrives at a dimension $\leq \mathrm{deg}(x^-) - \mathrm{deg}(x^+) + (2m-2) - \mathrm{codim}(P)$ for the space of such $(r,u,v)$, which explains Condition \ref{th:a-chain}. (As in the parallel situation of Condition \ref{th:a-stratification}, there are really two cases, depending on whether \eqref{eq:bubble-adjacency} happens outside $D$, or on $D$; the second case uses the last sentence in Lemma \ref{th:bubble-dim}.) 

The final (and much simpler) requirement is the analogue of Condition \ref{th:bubble}.

\begin{condition}[Generic bubbling not at $z^*$] \label{th:a-bubble}
Consider the stratum where all points of $\Sigma$ are distinct, and none is equal to $z^*$. Look at solutions $(p,r,u)$ in this stratum, satisfying \eqref{eq:mult-leq} as well as \eqref{eq:adjacency-condition}. We then additionally require the presence of a single bubble, exactly as in Condition \ref{th:bubble}. This space should be empty if \eqref{eq:mainstratumdim0} is $\leq 1$. As in Condition \ref{th:a-stratification}, we need to add the corresponding condition for the auxiliary maps which make $c_P$ a pseudo-cycle.
\end{condition}

Let's assume from now on that these conditions hold. We then have the following transversality and compactness results.

\begin{lemma} \label{th:AMPcompact}
Consider spaces $\frakB_m(P,x^-,x^+)$ of expected dimension $\leq 1$. In any such space, the domains are cylinders where all the marked points (points of $\Sigma$, as well as $z^*$) are pairwise distinct. Moreover, each such space is regular, hence a manifold of the expected dimension.
\end{lemma}

\begin{proof}
If $\Sigma$ has multiple points, or if any of them lies on $z^*$, we have $|\Pi| < m$ in Condition \ref{th:a-stratification}, hence codimension $\geq 2$. Regularity follows from the same condition, with $m^* = 0$ and the partition $\Pi = (1,1,\dots,1)$.
\end{proof}

\begin{proposition} \label{th:AMPcompact-2}
(i) Each space $\frakB_m(P,x^-,x^+)$ of expected dimension $0$ is compact (a finite set).

(ii) Suppose that the expected dimension is $1$. Then, $\bar\frakB_m(P,x^-,x^+)$ is a one-manifold with boundary, obtained by adding the following as boundary points:
\begin{equation}
\coprod_{\substack{i+j = m \\ x}} \frakD_{i}(x^-,x) \times \frakB_{j}(P,x,x^+) \quad \text{and} \quad
\coprod_{\substack{i+j = m \\ x}} \frakB_{i}(P,x^-,x) \times \frakD_{j}(x,x^+)
\end{equation}
Here, it is understood that $\frakD_0$ stands for the space of Floer trajectories.
\end{proposition}

\begin{proof}
We will make a simplifying assumption, namely that $P$ is a closed manifold. There are additional limits for a general pseudo-cycle; those are straightforward to deal with (using the addenda in Conditions \ref{th:a-stratification} and \ref{th:a-bubble}), but would complicate the formulation of the argument.

Take a sequence in $\frakB_m(P,x^-,x^+)$. Think of the elements of the sequence as maps on the cylinder. Look at the stable map limit of a subsequence, which consists of the following.
\begin{itemize}
\itemsep.5em
\item A point $p \in P$.

\item A chain of maps on cylinders (with matching asymptotics, all of which are periodic orbits in $M \setminus D$ by \cite[Lemma 2.1.5]{pomerleano-seidel24})
\begin{equation}\label{eq:chainofbreakings}
u^1: C^1 \longrightarrow M, \dots, u^R: C^R \longrightarrow M.
\end{equation}
Each $C^k$ carries a divisor $\Sigma^k$, which can be empty. The degrees of these divisors satisfy $m^1 + \cdots + m^R = m$. One of the cylinders carries an extra marked point $z^* \in C^c$ (we call that one the principal cylinder). This is arranged so that forgetting the components which carry no marked points recovers the limit of the sequence of domains in $\bar{\frakB}_m$. As for the maps, if the component $C^k$ carries no marked points, then $u^k$ is a solution of Floer's equation. If it does carry marked points, we think of $C^k$ as a fibre of the universal curve over $\frakD_{m^k}$ or $\frakB_{m^k}$ (the latter applies to $k = c$), and $u^k$ must satisfy the corresponding continuation map equation. 

\item sphere bubble trees attached to points of the cylinders. We will consider this in more detail later, but there is an overall topological constraint \cite[Lemma 2.3.1]{pomerleano-seidel24}:
\begin{equation} \label{eq:topological-constraint}
\parbox{37em}{
at any $z \in C^k$, the local multiplicity of intersection of $u^k$ with $D$, together with the intersection numbers of all sphere bubbles attached at $z$ with $D$, must be equal to $\Sigma^k_z$ (the number of points in the divisor located at $z$).
}
\end{equation}
\end{itemize}
We will now explain how to rule out any limit not lying in the space $\frakB_m(P; x^-,x^+)$ itself. The argument is divided into several cases (other ways of organizing it would be possible).

{\bf (a)} {\em We have $u^c(z^*) \neq c_P(p)$.} In that case, there must be sphere bubbles attached to the cylinder $u^c$ at $z^*$. We keep those bubbles and throw away ones attached elsewhere. From \eqref{eq:topological-constraint} we see that the outcome satisfies \eqref{eq:mult-leq-2}. Hence this cannot happen, by Condition \ref{th:a-chain}.

{\bf (b)} {\em Suppose that the position of the marked points is non-generic (one of the divisors has two points in the same position, or one point located at $z^*$); and that $u^c(z^*) = c_P(p)$.} In that case, we forget all the bubbles. Because of \eqref{eq:topological-constraint}, the outcome still satisfies \eqref{eq:mult-leq}. We then get codimension $\geq 2$ from Condition \ref{th:a-stratification}.

{\bf (c)} {\em Suppose that the position of the marked points is generic (all points of the divisors distinct, none equal to $z^*$); that there are sphere bubbles; and that $u^c(z^*) = c_P(p)$.} We forget all bubbles except one, which is attached directly to one of the cylinder components. For that component, we are in the situation excluded by Conditions \ref{th:bubble} (non-principal cylinder) or \ref{th:a-bubble} (principal cylinder).

{\bf (d)} {\em None of the above.} This means that the position of the points is generic, there are no sphere bubbles, and $u^c(z^*) = c_P(p)$. The only remaining phenomenon is the familiar splitting of one cylinder into several, whose codimension is the number of pieces minus $1$. In situation (i), no such splitting can happen; and in situation (ii), one can have splitting into exactly two pieces.
\end{proof}

Proposition \ref{th:AMPcompact-2} implies that counting rigid points in the zero-dimensional spaces $\frakB_m(P,x^-,x^+)$ defines a map 
\begin{equation} \label{eq:P-m-map}
\begin{aligned}
& b_m(P): \mathit{CF}^*(w) \longrightarrow \mathit{CF}^{*+\mathrm{codim}(P)-2m}(w+m), \\
& \!\!\! \sum_{i+j=m} d_i b_j(P) - (-1)^{\mathrm{dim}(P)} b_i(P) d_j = 0.
\end{aligned}
\end{equation}
As in previous instances, there is a parallel construction over $\frakB_m^{\dag}$, giving rise to operations \begin{equation} \label{eq:P-m-map-2}
\begin{aligned}
&
b_m^{\dag}(P): \mathit{CF}^*(w) \longrightarrow \mathit{CF}^{*+\mathrm{codim}(P)-2m-1}(w+m+1),
\\
& \!\!\!
\sum_{i+j=m} d_i^\dag b_j(P) + (-1)^{\mathrm{dim}(P)} d_i b_j^\dag(P) -
b_i(P) d_j^\dag + b_i^\dag(P) d_j = 0.
\end{aligned}
\end{equation}
Combining these as in \eqref{eq:define-a} gives a chain map of degree $\mathrm{codim}(P)$,
\begin{equation} \label{eq:modulechain7}
\begin{aligned}
& b_{C_q}(P): C_q \longrightarrow C_q, \\
& b_{C_q}(P)(x) = \sum_m q^m b_m(P)(x), \\ & b_{C_q}(P)(\eta x) =
\sum_m q^m\big( (-1)^{\mathrm{dim}(P)} \eta\, b_m(P)(x) + b_m^\dag(P)(x)\big).
\end{aligned}
\end{equation}
These maps are clearly additive with respect to disjoint sum of pseudo-cycles. A familiar variation of the same construction, using pseudo-chains (with boundary), shows that two pseudo-cycles representing the same homology class induce chain homotopic maps \eqref{eq:modulechain7}. Hence, on the cohomology level we get a well-defined map, which is the operation $\bullet_q$ from \eqref{eq:circ-module}.

The remaining task at this point is to discuss Lemma \ref{th:kappa-and-quantum-product}, in the following more precise form:

\begin{lemma}
Take $c_P: P \rightarrow M$ to be a perturbation of the inclusion of $D$ into $M$, to an embedding which is transverse to $D$. Then the (degree $2$) maps
\begin{equation}
q \,a_{C_{q}}, \; b_{C_q}(P): C_q \longrightarrow C_q
\end{equation}
are chain homotopic. 
\end{lemma}

\begin{proof}
This is a direct analogue of \cite[Lemma 8.2.1]{pomerleano-seidel24} (more precisely, it is the counterpart of the special case $R = D$ of that Lemma). Naively, the idea is to set 
\begin{equation} \label{eq:p-d} 
(P,c_P) = (D, \mathit{inclusion}).
\end{equation}
To make this work, some extra explanation is required, since \eqref{eq:p-d} does not satisfy the transversality assumption in \eqref{eq:transverse-pseudo-cycle}. Let's look at what happens if we go through with our construction of $b_{C_q}$ for \eqref{eq:p-d}.

In order to satisfy \eqref{eq:adjacency-condition}, the point $z^*$ has to be part of the divisor $\Sigma$. Let's restrict to that part of $\frakB_m$, and consider \eqref{eq:adjacency-condition} as an intersection condition taking place in $D$. In the situation of Condition \ref{th:a-stratification}, one gets smooth moduli spaces of dimension
\begin{equation}
\mathrm{deg}(x^-) - \mathrm{deg}(x^+) + 2|\Pi| \leq \mathrm{deg}(x^-) - \mathrm{deg}(x^+) + 2m-2. 
\end{equation}
Look at the top-dimensional stratum, where $\Sigma$ consists of $m$ distinct points, one which is $z^*$. In that case, \eqref{eq:mult-leq} says that $u$ must intersect $D$ transversally at $z^*$. Based on that, a comparison of linearized operators as in \cite[Lemma 8.4.1]{pomerleano-seidel24} shows that these maps are also regular as elements of $\frakB_m(D,x^-,x^+)$ in the original analytic setup (when $r$ ranges over all $\frakB_m$, and the adjacency condition is considered as taking place in $M$). Hence, for \eqref{eq:p-d} we can still satisfy Condition \ref{th:a-stratification} at least in the cases where \eqref{eq:mainstratumdim0} is $\leq 1$, which is what matters for the application.

In the situation of Condition \ref{th:a-chain}, it is sufficient to retain one component of the stable map, which is non-constant and goes through the point $u(z^*)$. Write $w: S^* \rightarrow M$ for the underlying simple map of that component. By construction, $w \cdot D \leq e \leq m^*$. Generically, the space of $(r,u,w,\zeta)$ such that $u(z^*) = w(\zeta)$, is smooth of dimension
\begin{equation} \label{eq:d-dimension}
\mathrm{deg}(x^-) - \mathrm{deg}(x^+) + 2|\Pi| + 2(w \cdot D) - 4 \leq
\mathrm{deg}(x^-) - \mathrm{deg}(x^+) + 2m - 4.
\end{equation}
If \eqref{eq:mainstratumdim0} is $\leq 1$, this is negative, so the space is empty. Once more, it turns out that Condition \ref{th:a-chain} can be satisfied for the choice \eqref{eq:p-d}. The final aspect, Condition \ref{th:a-bubble}, goes through without modifications.

We have seen that $b_{C_q}(D)$ is well-defined, and therefore indeed (choosing the same data, which is compatible with all the genericity requirements) agrees with $q\,a_{C_q}$; the extra power of $q$ comes from how the point-counting invariants are packaged algebraically. The usual cobordism argument shows that if $P$ is a perturbation of \eqref{eq:p-d} then $b_{C_q}(P)$ is chain homotopic to $b_{C_q}(D)$, which completes the proof.
\end{proof}

\begin{remark} \label{th:mod-p-pseudocycles}
We have used pseudo-cycles as representatives of integer cohomology classes. For rational or complex coefficients, one can take appropriate linear combinations of pseudo-cycles. To define the module structure with $\bF_p$-coefficients, mod $p$ pseudo-cycles should be used instead, see e.g.\ \cite[Remark 4.11]{seidel-wilkins21}. For our purpose that isn't even strictly necessary, since we only use the $\bullet_q$-action of $c_1(M)$, which is an integral class. Indeed, in the only place where the $\bF_p$-version of the module structure occurs (the proof of Corollary \ref{th:torsion-1b}), we will indicate a workaround which avoids mod $p$ pseudo-cycles.
\end{remark}

\subsection{Compatibility with the quantum product\label{sec:compatibility}} This section is dedicated to proving the following property, previously stated in \eqref{eq:circ-module}:

\begin{proposition} \label{th:module-over}
The cohomology level maps induced by \eqref{eq:modulechain7} make $\mathit{SH}^*_q(M,D)$ into a module over $\mathit{H}^*(M)[q]$, equipped with its quantum product.
\end{proposition}

The idea is the same as for the classical quantum cap action on Floer cohomology (without quantum corrections, this is \cite[Proposition 4.1]{le-ono}; the connection with the quantum product was made in \cite{pss}). Consider $b_{C_q}(P_1) \circ b_{C_q}(P_2)$, schematically represented as follows:
\begin{equation}
\includegraphics{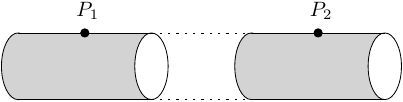}
\end{equation}
One deforms that by gluing the cylinders together, and then moving the two distinguished marked points towards each other, until they bubble off into a sphere:
\begin{equation}
\includegraphics{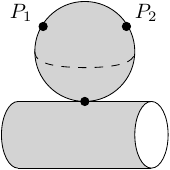}
\end{equation}
Algebraically, the moduli spaces arising from this deformation yield a chain homotopy between $b_{C_q}(P_1) \circ b_{C_q}(P_2)$ and $b_{C_q}(P)$, where $P$ is a suitable representative of $[P_1] \ast_q [P_2]$. There is no major conceptual hurdle in making this rigorous: the main task is simply to write down all the transversality and consistency conditions that have to be imposed in the process.

We first define the quantum product in a form which is convenient for our purpose. 
\begin{equation} \label{eq:tranverse-pseudo-cycle-pairs}
\parbox{37em}{
Let $c_{P_1}: P_1 \rightarrow M$, $c_{P_2}: P_2 \rightarrow M$ be pseudo-cycles which are transverse to each other and to $D$. This condition can be expressed by saying that the product pseudo-cycle $c_{P_1} \times c_{P_2}: P_1 \times P_2 \rightarrow M \times M$ must be transverse, in the same sense as in \eqref{eq:transverse-pseudo-cycle}, to the following submanifolds: $D \times M$, $M \times D$, the diagonal $\Delta_M$, and any intersection of those.
}
\end{equation}
This implies that $P = P_1 \times_M P_2$, with the obvious map $c_P: P \rightarrow M$, is again a pseudo-cycle transverse to $D$. We write that pseudo-cycle as $P_1 \cap P_2$.
The following is an analogue of Lemma \ref{th:bubble-dim}.

\begin{lemma} \label{th:bubble-2points}
Take an integer $e \geq 1$. For $J$ which makes $D$ into an almost complex submanifold, let $\bar\frakM_e^{(2)}(P_1,P_2)$ be the space of genus zero three-pointed $J$-holomorphic stable maps $v: (S,\zeta_0,\zeta_1,\zeta_2) \rightarrow M$, with $v \cdot D = e$, such that $v(\zeta_1) \in \overline{c_{P_1}(P_1)}$, $v(\zeta_2) \in \overline{c_{P_2}(P_2)}$. Let $\frakM_e^{(2)}(P_1,P_2) \subset \bar\frakM_e^{(2)}(P_1,P_2)$ be the subspace where: $S$ is smooth (has only one component); the map $v$ is simple, and transverse to $D$; and $v(\zeta_1) \in c_{P_1}(P_1)$, $v(\zeta_2) \in c_{P_2}(P_2)$. Consider the evaluation map at the remaining marked point $\zeta_0$.
Generically, the following holds.

(i) $\frakM_e^{(2)}(P_1,P_2)$ is smooth, of dimension $\mathrm{dim}(M)+2e-\mathrm{codim}(P_1) - \mathrm{codim}(P_2)$, and the evaluation map $\frakM_e^{(2)}(P_1,P_2) \rightarrow M$ is transverse to $D$.

(ii) The image of the evaluation map
$\bar\frakM_e^{(2)}(P_1,P_2) \setminus \frakM_e^{(2)}(P_1,P_2) \rightarrow M$
has dimension $\leq \mathrm{dim}(M) + 2e - 2 - \mathrm{codim}(P_1) - \mathrm{codim}(P_2)$; and the intersection of that image with $D$ is of dimension $\leq \mathrm{dim}(M) + 2e - 4 - \mathrm{codim}(P_1) - \mathrm{codim}(P_2)$. 
\end{lemma}

As a consequence, the evaluation map on $\frakM_e^{(2)}(P_1,P_2)$ is a pseudo-cycle transverse to $D$. For the rest of this section, we fix a $J = J^*$ which has the properties from Lemma \ref{th:bubble-2points} for all $e$. The quantum product of our pseudo-cycles is defined by
\begin{equation} \label{eq:pseudo-cycle-product}
\quad [P_1] \ast_q [P_2] = [P_1 \cap P_2] + \sum_{e \geq 1}\; q^e\, [\frakM_e^{(2)}(P_1,P_2)]. 
\end{equation}

With that in hand, let's return to our construction. We will assume that the data over $\frakB_m$ used to define \eqref{eq:modulechain7} satisfies the following additional constraints (which do not pose a problem for transversality):
\begin{itemize} 
\itemsep .5 em 
\item The almost complex structure at $z^*$ is always equal to our fixed $J^*$.

\item Consider the subset of $\frakB_m$ where at least $m^*$ points of $\Sigma$ lie at $z^*$. This can be identified with $\frakB_{m-m^*}$. We require that the choice of data (almost complex structures, inhomogeneous terms) should be compatible with this identification.
\end{itemize}

We introduce a parameter space $\frakB^{(2)}_m$ of pairs $(\Sigma,z_1^*,z_2^*)$, up to translation, where: \( \Sigma \) is a degree \( m \) divisor, and $z_1^*= (s_1^*, 0)$, $z_1^*= (s_2^*, 0)$ are two additional points lying on the line $\bR \times \{0\}$, such that $s_1^* < s_2^*$. This space again compactifies to a manifold with corners $\bar{\frakB}^{(2)}_m$. For simplicity, we just list the codimension one boundary strata of the compactification: 
\begin{itemize} \itemsep.5em 
\item 
There is a boundary stratum $\partial_{z_1^*=z_2^*}\bar{\frakB}^{(2)}_m$ where the two marked points $z_1^*$, $z_2^*$ have collided to a single point $z^*$. Clearly,
\begin{align} \label{eq:twopointcollision} 
\partial_{z_1^* = z_2^*}\bar\frakB^{(2)}_m \cong \frakB_m. 
\end{align}

\item The cylinder can split into two pieces, each of which carries one of the distinguished marked points. This gives rise to boundary strata of the form 
\begin{align} \label{eq:A2boundarytype1} 
\frakB_{i} \times \frakB_{j}, \quad i+j=m.  
\end{align} 

\item The cylinder splits into two pieces, one of which carries both of the distinguished marked points $z_1^*,z_2^*$. These strata are of the form
\begin{align} 
\label{eq:A2boundarytype2} \frakB_{i}^{(2)} \times \frakD_{j} \text{\quad  or \quad  } \frakD_{i} \times \frakB_{j}^{(2)} , \quad i+j=m. 
\end{align} 
\end{itemize} 
Choose continuation data over $\frakB^{(2)}_m$, whose extension to the compactification must satisfy the following consistency conditions.
\begin{itemize} 
\itemsep .5 em 
\item Over $\partial_{z_1^*=z_2^*}\bar\frakB^{(2)}_m$, it agrees with the previously chosen data on $\frakB_m$, using \eqref{eq:twopointcollision}.

\item Along boundary strata of type \eqref{eq:A2boundarytype1}, \eqref{eq:A2boundarytype2} we require it to be compatible with the product decompositions (and previously chosen data).  
\end{itemize}

Given periodic orbits $x^{\pm}$, we use this data to define moduli spaces $\frakB_m^{(2)}(P_1,P_2,x^{-},x^{+})$ consisting of quadruples $(p_1,p_2,r,u)$, where $p_i \in P_i$; $r \in \frakB_m^{(2)}$; and $u: C \rightarrow M$ satisfies \eqref{eq:intersectd} as well as
\begin{equation} \label{eq:adjacency-condition2}
u(z_1^*) = c_{P_{1}}(p_1),\quad u(z_2^*) = c_{P_{2}}(p_2).
\end{equation}
The expected dimension of these spaces is
\begin{align} \label{eq:mainstratumdim} \operatorname{deg}(x^{-})-\operatorname{deg}(x^{+})+2m+1-\operatorname{codim}(P_1)-\operatorname{codim}(P_2).
\end{align} 

Over $\frakB_m^{(2)}$, the regularity requirements are direct analogues of Conditions \ref{th:a-stratification}--\ref{th:a-bubble}, and we will not discuss them further here. There are additional conditions along $\partial_{z_1^* = z_2^*}\bar\frakB^{(2)}_m$, as follows.

\begin{condition}[Stratification]\label{th:strata2}
Consider, as in Condition \ref{th:a-stratification}, a stratum of $\partial_{z_1^* = z_2^*}\bar\frakB^{(2)}_m$ depending on $0 \leq m^* \leq m$ and a partition $\Pi$ of $m-m^*$. Take $(p_1,p_2,r,u,v,p_1,p_2)$ as follows. We have points $p_k \in P_k$. The parameter $r$ lies in our stratum, $u$ solves the corresponding continuation map equation, and $v: S \rightarrow M$ is a $J^*$-holomorphic map from $S \iso \bC P^1$, decorated with three distinct marked points $\zeta_0,\zeta_1,\zeta_2 \in S$, which is either constant or simple. The intersection condition (where $z^* = z_1^* = z_2^*)$ is
\begin{equation} \label{eq:incidencewithbubble}
u^{-1}(D) + (v \cdot D)\,z^* = \Sigma.
\end{equation}
The incidence conditions are 
\begin{equation} 
u(z^*) = v(\zeta_0), \;\; c_{P_1}(p_1) = v(\zeta_1),  \;\; c_{P_2}(p_2) = v(\zeta_2).  
\end{equation}  
We require that the space of such configurations is smooth of the expected dimension \begin{equation} 
\label{eq:dimensionofsomestratum} \operatorname{deg}(x^{-})-\operatorname{deg}(x^{+}) + 2|\Pi| + 2(v \cdot D) - \operatorname{codim}(P_1) - \operatorname{codim}(P_2). 
\end{equation}
As in Condition \ref{th:a-stratification}, we have an added requirement, which applies when $v(\zeta_1)$ or $v(\zeta_2)$ lies in the limit set of the pseudo-cycle, and the dimension is correspondingly lower.
\end{condition}

The next condition excludes more complicated bubbling at $z^*$.

\begin{condition}[Extra bubbling at $z^*$] \label{th:a2-extrabubble}
Take $0 < m^* \leq m$. Consider $(r,u,v)$, where: $r \in \partial_{z_1^* = z_2^*} \bar\frakB^{(2)}_m$ describes a divisor with $\Sigma_{z^*} = m^*$; we have $v \in \bar\frakM_e^{(2)}(P_1,P_2) \setminus \frakM_e^{(2)}(P_1,P_2)$ for some $e>0$; and \eqref{eq:bubble-adjacency}, \eqref{eq:mult-leq-2} hold. We require that this space is empty when \eqref{eq:mainstratumdim} is $\leq 1$.
\end{condition}

Finally, we have the analogue of Condition \ref{th:a-bubble}:

\begin{condition}[Generic bubbling not at $z^*$] \label{th:a2-genericbubble}
Take configurations $(u,r,v,p_1,p_2)$ as in Condition \ref{th:strata2}, such that the partition $\Pi$ of $m-m^*$ is generic. We require the presence of a single additional bubble at one of the remaining marked points, exactly as in Condition \ref{th:bubble}. 
These spaces should be empty if the dimension \eqref{eq:mainstratumdim} is $\leq 1$.
\end{condition}

Under these conditions, we have the following results. 

\begin{lemma} \label{th:AM2Pcompact-1} 
Consider spaces $\frakB_m^{(2)}(P_1,P_2,x^-,x^+)$ of expected dimension $\leq 1$. In any such space, the domains are cylinders where all the marked points (points of $\Sigma$, as well as $z_1^*$, $z_2^*$) are pairwise distinct. Moreover, each such space is regular, hence a manifold of the expected dimension.
\end{lemma}

This follows directly from dimension computations, exactly as in Lemma \ref{th:AMPcompact}.

\begin{proposition} \label{th:AM2Pcompact-2}
(i) If the expected dimension is $0$, then $\frakB_m^{(2)}(P_1,P_2,x^-,x^+)$ is compact (a finite set).

(ii) If the expected dimension \eqref{eq:mainstratumdim} is one, the compactification $\bar{\frakB}_m^{(2)}(P_1,P_2,x^-,x^+)$ is a one-manifold with boundary, obtained by adding the following as boundary points:
\begin{align} \label{eq:boundaryb21}
& \coprod_{\substack{i+j = m \\ x}} \frakB_i(P_1,x^-,x) \times \frakB_j(P_2,x,x^+); \\
&
\label{eq:boundaryb22}
\coprod_{\substack{i+j = m \\ x}} \frakD_{i}(x^-,x) \times \frakB_{j}^{(2)}(P_1,P_2,x,x^+) \quad \text{and} \quad
\coprod_{\substack{i+j = m \\ x}} \frakB_{i}^{(2)}(P_1,P_2,x^-,x) \times \frakD_{j}(x,x^+);
\\
\label{eq:boundaryb23}
&
\frakB_m(P_1 \cap P_2, x^-,x^+) \quad \text{and} \quad
\coprod_{\substack{i+j = m \\ j>0}} \frakB_{i}(\frakM_{j}^{(2)}(P_1,P_2), x^-,x^{+}).
\end{align}
Here, \eqref{eq:boundaryb21} happens when the cylinder splits into two, where the left piece carries $z_1^*$, and the right one $z_2^*$. If in such a splitting, one of the two pieces carries both distinguished points, we get \eqref{eq:boundaryb22} instead. Finally, collision of the two marked points yields \eqref{eq:boundaryb23}.
\end{proposition}

\begin{proof} 
The only case that deserves specific discussion is when the two distinguished marked points $z_1^*$ and $z_2^*$ collide. To keep things reasonably streamlined, we again assume for simplicity that $P_1,P_2$ are closed manifolds. 

Let's take a sequence in $\frakB_m^{(2)}(P_1,P_2, x^-,x^+)$ where that happens, and look at the limit of a subsequence. Note that simultaneously, some amount $m^* \geq 0$ of points of $\Sigma$ can collide with $z_1^*,z_2^*$. The limit has: points $p_1,p_2$; a chain of cylinders $u^k: C^k \longrightarrow M$ as in \eqref{eq:chainofbreakings}, with divisors $\Sigma^k$ whose degrees we denote by $m^k$; and possibly sphere bubbles attached at points of $\Sigma^k$. The total configuration must satisfy the topological constraint \eqref{eq:topological-constraint}.

There is a principal cylinder $u^c: C^c \longrightarrow M$, which carries the point $z^* \in C^c$ at which the two distinguished points have collided. At $z^*$, there is a (possibly constant) $J$-holomorphic stable map $v$ with three marked points $\zeta_0,\zeta_1,\zeta_2$, so that 
\begin{equation}
v(\zeta_1)=c_{P_1}(p_1), \;\; v(\zeta_2)=c_{P_2}(p_2), \;\; v(\zeta_0)=u^c(z^*). 
\end{equation}
We consider these cases:

{\bf (a)} {\em Suppose $v$ has non-smooth domain, or is not simple, or is not transverse to $D$.} In that case, we forget all of the bubbles on $u^c$ besides $v$. The outcome satisfies the intersection condition \eqref{eq:mult-leq-2}. Hence, by Condition \ref{th:a2-extrabubble}, this bubbling does not occur for dimension $\leq 1$ moduli spaces.   

{\bf (b)} {\em Suppose that $v$ is simple, but the partition of the remaining $m^c-m^*$ marked points is non-generic.} We again forget all the bubbles except for $v$. As before, the outcome still satisfies \eqref{eq:mult-leq-2}, and we get that these configurations are of codimension $\geq 2$ by Condition \ref{th:strata2}.

{\bf (c)} {\em Suppose that $v$ is simple, the position of the remaining $m^c-m^*$ marked points on $u^c$ is generic, and that there are sphere bubbles attached to some of these remaining points.} We forget all of these additional bubbles except one. For that component, we are in the situation of Condition \ref{th:a2-genericbubble}, which again rules it out.

{\bf (d)} {\em None of the above.} This means that the $m^c - m^*$ marked points on $u^c$ are in generic position, the sphere bubble $v$ is simple. The topological constraint implies that $v\cdot D=m^*.$ Furthermore, the arguments from Proposition \ref{th:AMPcompact-2} show that additional cylindrical components (beyond the principal one carrying $z^*$) do not arise on the boundary of moduli spaces of dimension $\leq 1.$ Thus, the codimension one boundary strata where $z_1^*,z_2^*$ collide are of the form \eqref{eq:boundaryb23}: the first case happens when $m^* = 0$ ($v$ is constant), and the second one for $m^* > 0$ ($v$ is a simple pseudo-holomorphic sphere). Note that by (a) this sphere is transverse to $D$. That makes it easy to apply a gluing process, which shows that \eqref{eq:boundaryb23} are indeed boundary points.
\end{proof}

As a consequence, counting points in zero-dimensional spaces $\frakB_m^{(2)}(P_1,P_2,x^{-},x^{+})$ defines a map 
\begin{equation} \label{eq:aPQhomotopy1}
\begin{aligned}
& b_{m}^{(2)}(P_1,P_2): \mathit{CF}^*(w) \longrightarrow \mathit{CF}^{*+|P_1|+|P_2|-2m-1}(w+m+1), \\
& b_m(P_1 \cap P_2) + \sum_{\substack{i+j=m \\ j>0}} b_i(\frakM_j^{(2)}(P_1,P_2))- \sum_{i+j=m} b_i(P_1) \circ b_j(P_2)
\\ 
= & \sum_{i+j=m} d_{i} \, b_{j}^{(2)}(P_1,P_2) + (-1)^{\mathrm{dim}(P_1) + \mathrm{dim}(P_2)}  b_{i}^{(2)}(P_1,P_2)\, d_{j}. 
\end{aligned}
\end{equation}
Yet again, there is a variant $\frakB^{\dag,(2)}_m$ of our parameter space, where one does not divide by translations. This leads to additional operations
\begin{equation} 
b_{m}^{(2),\dag}(P_1,P_2): \mathit{CF}^*(w) \longrightarrow \mathit{CF}^{*+|P_1|+|P_2|-2m-2}(w+m+2)
\end{equation}
satisfying
\begin{equation}
\begin{aligned}
\label{eq:aPQhomotopy2} 
& b_m^\dag(P_1 \cap P_2) + \sum_{\substack{i+j = m \\ j>0}} b_{i}^{\dag}(\frakM_{j}^{(2)}(P_1,P_2))- (-1)^{\mathrm{dim}(P_2)} b_{i}^{\dag}(P_1) \circ b_{j}(P_2)- b_{i}(P_1) \circ b_{j}^{\dag}(P_2) 
\\ 
= & \sum_{i+j = m} d_i^{\dag} b^{(2)}_j(P_1,P_2) 
+ (-1)^{\mathrm{dim}(P_1)+\mathrm{dim}(P_2)-1} d_i\, b^{(2),\dag}_j(P_1,P_2) 
\\[-1em] & \qquad \qquad - b^{(2)}_i(P_1,P_2) \, d_j^{\dag} + b^{(2),\dag}_i(P_1,P_2) \, d_j.
\end{aligned}
\end{equation}

\begin{proof}[Proof of Proposition \ref{th:module-over}]
Along familiar lines, we package our operations into a map of degree $\mathrm{codim}(P_1) + \mathrm{codim}(P_2)-1$:
\begin{equation} \label{eq:b-2-homotopy}
\begin{aligned}
& b^{(2)}_{C_q}(P_1,P_2): C_q \longrightarrow C_q, \\
& b^{(2)}_{C_q}(x) = \sum_m q^m b^{(2)}_m(P_1,P_2)(x), \\
& b^{(2)}_{C_q}(\eta x) = \sum_m q^m \big((-1)^{\mathrm{dim}(P_1)+\mathrm{dim}(P_2)-1}\, \eta b^{(2)}_m(P_1,P_2)(x) + b^{(2),\dag}_m(P_1,P_2)(x)\big).
\end{aligned}
\end{equation}
The equations \eqref{eq:aPQhomotopy1} and \eqref{eq:aPQhomotopy2} amount to \eqref{eq:b-2-homotopy} being a chain homotopy between $b_{C_q}(P_1) \circ b_{C_q}(P_2)$ and $b_{C_q}(P_1 \cap P_2) + \sum_{e>0} q^e b_{C_q}(\frakM_e^{(2)}(P_1,P_2))$. In view of \eqref{eq:pseudo-cycle-product}, this yields the desired result on cohomology.
\end{proof}

\subsection{Applications and discussion\label{sec:applications}}
We derive the remaining property of deformed symplectic cohomology stated in Section \ref{subsec:relative}; and conclude the section with informal remarks concerning possible further developments.

\begin{proof}[Proof of Corollary \ref{th:torsion-1}]
Take the homogeneous minimal polynomial \eqref{eq:f-q}. From \eqref{eq:f-q-c1} and Lemma \ref{th:kappa-and-quantum-product}, we know that $q^{\mathrm{deg}(f)} f(a_{q}) = f_q(q a_{q}) = 0$ as an endomorphism of $\mathit{SH}^*_q(M,D;\bC)$. 
The rest of the argument is a degree consideration: the endomorphism $q^n f(a_{q})$ is of degree $2n$, hence by \eqref{eq:deformed-sh} takes values in the subspace $H^*(M;\bC)[q] \subset \mathit{SH}^*_q(M,D;\bC)$. That subspace has no $q$-torsion, so from $q^N f(a_{q}) = 0$ for some $N$, it follows that $q^n f(a_{q}) = 0$.
\end{proof}

For future use, we note a variant:

\begin{corollary} \label{th:torsion-1b}
Suppose that $p$ satisfies \eqref{eq:no-torsion}. Then \eqref{eq:f-kappa} holds in the endomorphisms of $\mathit{SH}^*_q(M,D;\bF_p)$.
\end{corollary}

\begin{proof} 
We have $f_q(c_1(M)) = 0$ in quantum cohomology with $\bF_p$-coefficients, for the same reason as in the proof of Corollary \ref{th:congruence}. Then, the statement follows as in Corollary \ref{th:torsion-1}, working with $\bF_p$-coefficients throughout. 

Alternatively, if one wants to avoid mod $p$ pseudo-cycles (see Remark \ref{th:mod-p-pseudocycles}), the argument can be tweaked as follows. Using integral pseudo-cycles, and Floer cohomology with $\bF_p$-coefficients, one can define a version of \eqref{eq:circ-module}, 
\begin{equation} \label{eq:modified-circ}
\bullet_q: H^*(M;\bZ)[q] \otimes_{\bZ[q]} \mathit{SH}^*_q(M,D;\bF_p) \longrightarrow \mathit{SH}^*_q(M,D;\bF_p).
\end{equation}
This makes $\mathit{SH}^*_q(M,D;\bF_p)$ into a module over the integral quantum cohomology ring, and satisfies the $\bF_p$-counterpart of Lemma \ref{th:kappa-and-quantum-product}. We know that $f_q(c_1(M)) \in H^*(M;\bZ)[q]$ becomes zero after multiplication with some natural number, which is coprime to $p$. Using \eqref{eq:modified-circ}, it follows that the same holds for $q^{\mathrm{deg}(f)}f(a_{q})$ as an endomorphism of $\mathit{SH}^*_q(M,D;\bF_p)$, which is therefore zero. The grading argument for $q^n f(a_{q})$ works as before.
\end{proof}

\begin{remark} \label{th:bound}
(i) By \cite[Lemma 7.1.10]{pomerleano-seidel23} (stated there with complex coefficients, but the proof works over $\bF_p$ as well), $\mathit{SH}^*_{u,q}(M,D;\bF_p)$ is a finitely generated module over the Weyl algebra $\bF_p[u]\langle q,\nabla_{u\partial_q}\rangle$. This implies that the constant $A$ in Corollary \ref{th:torsion-2} can be chosen uniformly over all $x$. Namely, take $A \equiv 0$ mod $p$, such that $q^A f(\nabla_{u\partial_q}^p)$ acts trivially on some set $x_1,\dots,x_m$ of generators. Then, because $q^A f(\nabla_{u\partial_q}^p)$ commutes with $u$, $q$, and $\nabla_{u\partial_q}$, we have $q^A f(\nabla_{u\partial_q}^p)x = 0$ for all $x \in \mathit{SH}^*_{u,q}(M,D;\bF_p)$.

(ii) One expects that generators in the sense of (i) can be chosen so that
\begin{equation} \label{eq:minus-one-generators}
x_1,\dots,x_m \in H^*(M;\bF_p) \oplus H^*(D;\bF_p)z \subset \mathit{SH}^*_{u,q}(M,D;\bF_p). 
\end{equation}
This could be used to improve the statement of Corollary \ref{th:torsion-2}, as follows. Take the decreasing filtration $F^{\geq K} \subset \mathit{SH}^*_{u,q}(M,D;\bF_p)$, $K \in \bZ$, given by
\begin{equation}
F^{\geq K} = \begin{cases}
q^K H^*(M;\bF_p)[u,q] & K \geq 0, \\
H^*(M;\bF_p)[u,q] \oplus H^*(D;\bF_p)[u]z \oplus \cdots H^*(D;\bF_p)[u]z^{-K} & K < 0.
\end{cases}
\end{equation}
On the cochain level, this is represented by a version of the action filtration; see \cite[Section 4]{pomerleano-seidel24} for the definition of the filtration, and more specifically the discussion after \cite[Equation (4.3.9)]{pomerleano-seidel24} for its cohomology level implications. By definition, one has
\begin{equation} \label{eq:jump-filtration}
\begin{aligned}
& q(F^{\geq K}) \subset F^{\geq K+1}, \\
& \nabla_{u\partial_q}(F^{\geq K}) \subset F^{\geq K-1}.
\end{aligned}
\end{equation}
The $x_k$ from \eqref{eq:minus-one-generators} lie in $F^{\geq -1}$, hence by \eqref{eq:jump-filtration} we have
\begin{equation} \label{eq:f0}
q^{p\,\mathrm{deg}(f)+1} f(\nabla_{u\partial_q}^p) x_k \in F^{\geq 0}.
\end{equation}
We know from Corollary \ref{th:torsion-2} that $f(\nabla_{u\partial_q}^p) x_k$ is $q$-torsion. But $F^{\geq 0}$ does not contain any nontrivial $q$-torsion elements, hence \eqref{eq:f0} must be zero. Combining this with the argument from (i), where $A$ was supposed to be a multiple of $p$, we find that it would suffice to take 
\begin{equation} \label{eq:weak-bound}
A = p(\mathrm{deg}(f) + 1).
\end{equation}
We remind the reader that this argument is incomplete, as we have not justified \eqref{eq:minus-one-generators} (doing that would require a digression into low-energy solutions of the thimble maps underlying the isomorphism \eqref{eq:deformed-sh-2}, and that lies somewhat outside the concerns of the present paper).
\end{remark}

\begin{remark} \label{th:more} 
We now add some speculations concerning other algebraic structures on deformed symplectic cohomology. One expected structure, the pair-of-pants product \eqref{eq:relative-quantum-product}, has already been mentioned. One also expects to have an analogue of the quantum Steenrod operations \eqref{eq:quantum-steenrod-endomorphism}, along the lines of \cite{seidel14c, shelukhin-zhao19, wilkins23}. This should associate to any $b \in \mathit{SH}^*_q(M,D;\bF_p)$ a $(u,\theta,q)$-linear map
\begin{equation} \label{eq:relative-qsigma}
Q\Sigma_b: \mathit{SH}^*_{u,q}(M,D;\bF_p)[\theta] \longrightarrow
\mathit{SH}^*_{u,q}(M,D;\bF_p)[\theta],
\end{equation}
satisfing the analogue of Property \ref{th:quantum-properties}(iv) with respect to the pair-of-pants product. The equivariant version of $a_q^p$ mentioned in the context of the relative Etingof-Lee conjecture (Remark \ref{th:relative-etingof-lee}) is expected to be equal the special case of \eqref{eq:relative-qsigma} where $b = 1z$ is the $\bF_p$-coefficient version of the class from \eqref{eq:1z}; this equality is the counterpart of Lemma \ref{th:kappa-and-quantum-product}.

Assuming this all works out, let's see how it could be useful. We assume torsion-freeness as in \eqref{eq:no-torsion}. From \eqref{eq:q-times-z} one sees that, in the ring structure given by the pair-of-pants product, 
\begin{equation}
q^{\mathrm{deg}(f)} f(1 z) = f_q(c_1(M)) = 0 \in \mathit{SH}^*_q(M,D;\bF_p).
\end{equation}
Because of the grading and $q$-action in \eqref{eq:deformed-sh}, this implies that $q^n f(1 \, z) = 0$. As a consequence,
\begin{equation} \label{eq:f-of-qsigma-rel}
0 = Q\Sigma_{q^n f(1 z)} = q^{np} Q\Sigma_{f(1 z)} = q^{np} f(Q\Sigma_{1 z}),
\end{equation}
where $f(Q\Sigma_{1 z})$ is formed in the endomorphism ring of $\mathit{SH}^*_{u,q}(M,D;\bF_p)[\theta]$. If one assumes the relative Etingof-Lee conjecture, then \eqref{eq:f-of-qsigma-rel} leads to a statement about the relative quantum connection, namely Conjecture \ref{th:stronger}.
\end{remark}

\section{The relative quantum connection\label{sec:proof}}
In this section we prove the main results from Section \ref{subsec:fl}, namely Theorem \ref{th:q-fontaine-laffaille} and Corollary \ref{th:torsion-3}. Before we can get to that, there are some more preliminaries. The first part concerns deformed symplectic cohomology; there, we follow \cite[Section 7.1a]{pomerleano-seidel23} with two modifications. One change is working with mod $p$ coefficients. The other (far more important) change is that we use the minimal polynomial \eqref{eq:minimal-polynomial}, whereas \cite{pomerleano-seidel23} merely proved the existence of some polynomial with the desired properties, without assigning a geometric meaning to it; the key here is to appeal to Lemma \ref{th:kappa-and-quantum-product}. The second part of the preliminaries, which deals with the deformed wrapped Fukaya category, does the same for \cite[Section 7.1b]{pomerleano-seidel23}. With that at hand, we add the categorical results from Section \ref{sec:algebra} to the mix, and it finishes cooking quickly.

\subsection{Variants of deformed symplectic cohomology\label{subsec:variants-sh}}
Throughout, we work with a prime $p$ such that \eqref{eq:no-torsion} holds. $C_q$ is always the chain complex underlying deformed symplectic cohomology with $\bF_p$-coefficients. Write
\begin{equation}
\begin{aligned}
C_{q^{\pm 1}} & \stackrel{\mathrm{def}}{=} \bF_p[q^{\pm 1}] \otimes_{\bF_p[q]} C_q, \\
q^{-1}C_{q^{-1}} & \stackrel{\mathrm{def}}{=}  q^{-1}\bF_p[q^{-1}] \otimes_{\bF_p[q]} C_q = C_{q^{\pm 1}}/C_q.
\end{aligned}
\end{equation}
The endomorphism $a_{q}$ carries over to those complexes. Taking the minimal polynomial \eqref{eq:minimal-polynomial}, we invert $f(a_{q})$:
\begin{equation} \label{eq:1overf}
\begin{aligned}
C_{q,1/f} & \stackrel{\mathrm{def}}{=} \bF_p[a_{q},f(a_{q})^{-1}] \otimes_{\bF_p[a_{q}]} C_q, \\
C_{q^{\pm 1},1/f} & \stackrel{\mathrm{def}}{=} \bF_p[a_{q},f(a_{q})^{-1}] \otimes_{\bF_p[a_{q}]} C_{q^{\pm 1}}, \\
q^{-1}C_{q^{-1},1/f} & \stackrel{\mathrm{def}}{=} \bF_p[a_{q},f(a_{q})^{-1}] \otimes_{\bF_p[a_{q}]} q^{-1}C_{q^{-1}}.
\end{aligned}
\end{equation}
The resulting cohomology groups will be denoted by $\mathit{SH}^*(M,D;\bF_p)$ with the same decorations. We have seen one of them before, in \eqref{eq:invert-f}.

\begin{lemma} \label{th:les}
There is an isomorphism, compatible with the actions of $q$ and $a_{q}$,
\begin{equation}
q^{-1}\mathit{SH}^*_{q^{-1},1/f}(M,D;\bF_p) \iso \mathit{SH}^{*+1}_{q,1/f}(M,D;\bF_p).
\end{equation}
\end{lemma}

\begin{proof}
By definition, we have a short exact sequence
\begin{equation} \label{eq:plusminus-short}
0 \longrightarrow C_q \longrightarrow C_{q^{\pm 1}} \longrightarrow
q^{-1}C_{q^{-1}} \longrightarrow 0.
\end{equation}
Inverting $f(a_{q})$, as applied to chain complexes of $\bF_p[a_{q}]$-modules, is an exact functor. We apply that to \eqref{eq:plusminus-short} and get a long exact sequence
\begin{equation} \label{eq:plusminus-les}
\cdots \rightarrow \mathit{SH}^*_{q,1/f}(M,D;\bF_p) \longrightarrow \mathit{SH}^*_{q^{\pm 1},1/f}(M,D;\bF_p) \longrightarrow q^{-1}\mathit{SH}^*_{q^{-1},1/f}(M,D;\bF_p) \rightarrow \cdots 
\end{equation}
By Corollary \ref{th:torsion-1b}, $f(a_{q})$ is zero as an endomorphism of $\mathit{SH}^*_{q^{\pm 1}}(M,D;\bF_p)$. Therefore, the middle group in \eqref{eq:plusminus-les} is zero, and the connecting map is an isomorphism.
\end{proof}

There are analogues for the $S^1$-equivariant theory. If $C_{u,q}$ is the chain complex underlying that theory \cite[Section 7]{pomerleano-seidel24}, again with $\bF_p$-coefficients, we set 
\begin{equation} \label{eq:invq-2}
\begin{aligned}
C_{u,q^{\pm 1}} & \stackrel{\mathrm{def}}{=} \bF_p[q^{\pm 1}] \hat\otimes_{\bF_p[q]} C_{u,q}, \\
q^{-1}C_{u,q^{-1}} & \stackrel{\mathrm{def}}{=} q^{-1}\bF_p[q^{-1}] \hat\otimes_{\bF_p[q]} C_{u,q}
\end{aligned}
\end{equation}
and 
\begin{equation} \label{eq:1overf-2}
\begin{aligned}
C_{u,q,1/f} & \stackrel{\mathrm{def}}{=} \bF_p[\nabla_{u\partial_q},1/f(\nabla_{u\partial_q})] \hat\otimes_{\bF_p[\nabla_{u\partial_q}]} C_{u,q}, \\
C_{u,q^{\pm 1},1/f} & \stackrel{\mathrm{def}}{=} \bF_p[\nabla_{u\partial_q},1/f(\nabla_{u\partial_q})] \hat\otimes_{\bF_p[\nabla_{u\partial_q}]} C_{u,q^{\pm 1}}, \\
q^{-1} C_{u,q^{-1},1/f} & \stackrel{\mathrm{def}}{=} \bF_p[\nabla_{u\partial_q},1/f(\nabla_{u\partial_q})] \hat\otimes_{\bF_p[\nabla_{u\partial_q}]} q^{-1} C_{u,q^{-1}}.
\end{aligned}
\end{equation}
The actions of $q$ and $\nabla_{u\partial_q}$ carry over; see \eqref{eq:leibniz-rule} for how to define $q$ on \eqref{eq:1overf-2}. We use the same notational convention for cohomology as before. One more bit of explanation: $\hat\otimes$ denotes the (graded) $u$-adic completion of the tensor product. In general, completion is important in order to get a well-behaved theory. By a suitable choice of Hamiltonians, one can arrange that the Floer complex is bounded below \cite[Lemma 4.1.3(ii)]{pomerleano-seidel24}. Then, graded completion does nothing when applied in the first line of \eqref{eq:1overf-2} (but is still relevant for the other two flavors). As a consequence, 
\begin{equation} \label{eq:no-completion}
\mathit{SH}^*_{u,q,1/f}(M,D;\bF_p) \iso \bF_p[\nabla_{u\partial_q},1/f(\nabla_{u\partial_q})] \otimes_{\bF_p[\nabla_{u\partial_q}]} \mathit{SH}^*_{u,q}(M,D;\bF_p).
\end{equation}
(This is a cohomology level statement, and therefore valid independently of the Hamiltonians; the specific choice is used only to prove it.) We have previously used this as the definition of $\mathit{SH}^*_{u,q,1/f}(M,D;\bF_p)$, in \eqref{eq:invert-f-2}.

\begin{lemma} \label{th:les-2}
There is an isomorphism, compatible with the actions of $q$ and $\nabla_{u\partial_q}$,
\begin{equation}
q^{-1}\mathit{SH}^*_{u,q^{-1},1/f}(M,D;\bF_p) \iso \mathit{SH}^{*+1}_{u,q,1/f}(M,D;\bF_p).
\end{equation}
\end{lemma}

\begin{proof}
As an application of (the version over $\bF_p$ of) \cite[Lemma 2.3.6]{pomerleano-seidel23}, we see that each of the complexes \eqref{eq:1overf-2} is $u$-torsionfree. That Lemma also determines the $u = 0$ reduction of the complexes. By \eqref{eq:kappa-and-connection} (or more precisely, because of the underlying chain level statement), those reductions turn out to be the corresponding complexes \eqref{eq:1overf}.
 
From the obvious inclusion and projection, we get a map
\begin{equation} \label{eq:cone-1}
\mathit{Cone}\big(C_{u,q,1/f} \rightarrow C_{u,q^{\pm 1},1/f}\big) \longrightarrow
q^{-1}C_{u,q^{-1},1/f}.
\end{equation}
The $u = 0$ reduction of this map consists of the groups correspondingly obtained from \eqref{eq:plusminus-short}. Because that is a short exact sequence, the $u=0$ map is a quasi-isomorphism; from there, a $u$-filtration argument shows that \eqref{eq:cone-1} is a quasi-isomorphism as well. A similar argument, together with the acyclicity of $C_{q^{\pm 1},1/f}$ which we established as part of the proof of Lemma \ref{th:les}, shows that the projection
\begin{equation} \label{eq:cone-2}
\mathit{Cone}\big(C_{u,q,1/f} \rightarrow C_{u,q^{\pm 1},1/f}\big) \longrightarrow
C_{u,q,1/f}[1]
\end{equation}
is a quasi-isomorphism. Combining \eqref{eq:cone-1} and \eqref{eq:cone-2} yields the desired isomorphism on cohomology. (This quasi-isomorphism argument followed \cite[Lemma 2.3.8]{pomerleano-seidel23}; we have narrowed that down to our specific situation, avoiding the language of triangulated categories.)
\end{proof}

Finally, one can invert $u$ on both sides of Lemma \ref{th:les-2}. We denote the resulting isomorphism by
\begin{equation} \label{eq:les-3}
q^{-1}\mathit{SH}^*_{u^{\pm 1},q^{-1},1/f}(M,D;\bF_p) \iso \mathit{SH}^{*+1}_{u^{\pm 1},q,1/f}(M,D;\bF_p).
\end{equation}

\subsection{The deformed wrapped Fukaya category\label{subsec:deformed-fukaya}}
Take the wrapped Fukaya category of $M \setminus D$. Since that is a Weinstein manifold, the category (with coefficients in $\bZ$, or in any ring) is generated by finitely many co-cores. We consider the endomorphism $A_\infty$-algebra of the direct sum of those co-cores, in two different versions: $\scrA$ is with $\bF_p$-coefficients, and $\tilde{\scrA}$ is with $(\bZ/p^2)$-coefficients. We also consider the $q$-deformation given by the Borman-Sheridan element. In \cite[Section 6.2]{pomerleano-seidel23}, that deformation was interpreted as an instance of the Maurer-Cartan formalism, which in general involves denominators. Let's quickly recall what that meant. Let $\mathit{CF}^*(H)$ be the Floer complex of a suitable Hamiltonian with infinite slope \cite[Section 4.1]{pomerleano-seidel23}; by a suitable choice, one can arrange that this is concentrated in degrees $\geq 0$ \cite[Proposition 7.2.2]{pomerleano-seidel23}. Let $b \in \mathit{CF}^0(H)$ be the Borman-Sheridan cocycle \cite[Section 7.2c]{pomerleano-seidel23}. The definition of the deformed Fukaya category involves parameter spaces $\frakR_{d,m}$ of discs with $(d+1)$ boundary punctures and $m$ interior punctures, where one inserts $qb$ at each interior puncture. In a little more detail: the moduli spaces first give rise to operations \cite[Equation (6.2.6)]{pomerleano-seidel23}
\begin{equation}
\mu^{d,m}: \mathit{CF}^*(L_0,L_1) \otimes \cdots \mathit{CF}^*(L_{d-1},L_d) \otimes \mathit{CF}^*(H)^{\otimes m} \longrightarrow \mathit{CF}^{*+2-d}(L_0,L_d)
\end{equation}
One then defines the deformed $A_\infty$-structure by the formula \cite[Equation (6.2.7)]{pomerleano-seidel23}
\begin{equation}
\mu^d_q(a_1,\dots,a_d) = \sum_{m \geq 0} \frac{q^m}{m!} \,\mu^{d,m}(a_1,\dots,a_d,\overbrace{b,\dots,b}^m),
\end{equation}
which does have denominators. However, the spaces $\frakR_{d,m}$ carry a free action of the symmetric group permuting the $m$ points, and one can achieve transversality equivariantly. Since $b$ is defined over the integers, $\mu^{d,m}(a_1,\dots,a_d,b,\dots,b)$ is a multiple of $m!$, cancelling out the apparent denominator. The outcome is that the deformation is defined over $\bZ$, and of course one can then reduce that structure modulo any number. We use two versions $\scrA_q$ and $\tilde{\scrA}_q$, corresponding to the ones above.

\begin{remark}
Alternatively, one could define the deformed wrapped Fukaya category by using pseudo-holomorphic discs which intersect $D$. This would combine the techniques in the definition of the relative Fukaya category \cite{perutz-sheridan20}, which does work over $\bZ$ for basically the reason mentioned above, with the popsicle formalism (which allows for a direct limit over Hamiltonians with increasing slopes) from \cite{abouzaid-seidel07}. The outcome would be the open string counterpart of the definition of deformed symplectic cohomology in \cite{pomerleano-seidel24}. We did not choose that approach, since applying it for our purpose would involve re-doing much of the work in \cite[Section 6]{pomerleano-seidel23} in that setup.
\end{remark}

The next statements involve the relation between the (deformed) wrapped Fukaya category and symplectic cohomology. A word of caution: because we are relying on \cite{pomerleano-seidel23}, this uses a different approach to defining deformed symplectic cohomology, with a single Hamiltonian of ``infinite slope'' \cite[Section 4.1]{pomerleano-seidel23}, and punctured cylinders where $qb$ is inserted into the punctures \cite[Section 5.3e]{pomerleano-seidel23}. This can be defined with integer (or any) coefficients, for the same reason as the deformed Fukaya category. In this context, the connection on the $S^1$-equivariant version is defined in \cite[Section 5.3f]{pomerleano-seidel23}, and one can set $u = 0$ in that construction to obtain the corresponding version of $a_{q}$, even though this was not specifically introduced in that paper; in the notation of \cite[Equation (5.3.33)]{pomerleano-seidel23}, the underlying chain map would be $x \mapsto \sum_{m \geq 1} \frac{1}{(m-1)!} \mathit{KH}^{(A)}_{m,0}(\partial_q\alpha,\alpha^{\otimes m-1},x)$. Ultimately, as explained in \cite[Section 10]{pomerleano-seidel24}, the resulting deformed symplectic cohomology and its structures  are isomorphic to the ones we've used in the rest of the paper.

\begin{lemma} \label{th:oc}
The open-closed map and its (negative) cyclic version,
\begin{align}
\label{eq:oc} & \mathit{OC}_q: \mathit{HH}_*(\scrA_q) \longrightarrow \mathit{SH}^{*+n}_q(M,D;\bF_p), \\
\label{eq:cyclic-oc} & \mathit{OC}_{u,q}: \mathit{HC}_*(\scrA_q) \longrightarrow \mathit{SH}^{*+n}_{u,q}(M,D;\bF_p)
\end{align}
are isomorphisms.
\end{lemma}

Concerning \eqref{eq:oc}, the corresponding property for undeformed ($q = 0$) category was proved in \cite[Theorem 1.1]{ganatra13} under a nondegeneracy assumption, which is now known to hold for all Weinstein manifolds. Our statement follows by a $q$-filtration argument; and \eqref{eq:cyclic-oc} from that by a $u$-filtration argument (see e.g.\ \cite[Section 7.1b]{pomerleano-seidel23} for further discussion and references).

\begin{lemma} \label{th:module}
Take the preimage of the unit class $1 \in \mathit{SH}^0_q(M,D;\bF_p)$ under \eqref{eq:oc}. Then, the action of Hochschild cohomology on that class yields an isomorphism
\begin{equation}
x \longmapsto x \cdot \mathit{OC}_q^{-1}(1): \mathit{HH}^*(\scrA_q) \longrightarrow \mathit{HH}_{*-n}(\scrA_q).
\end{equation}
\end{lemma}

\begin{proof}
Let's first look at the situation for $q = 0$. The open-closed and closed-open map are isomorphisms, and compatible with the module structure of Hochschild homology over Hochschild cohomology (this is again \cite[Theorem 1.1]{ganatra13}). Compatibility means that the following diagram commutes:
\begin{equation}
\xymatrix{
\mathit{HH}^*(\scrA) \otimes \mathit{HH}_{*-n}(\scrA) 
\ar[rrr]^-{\text{module structure}}
&&& 
\mathit{HH}_{*-n}(\scrA) 
\ar[dd]^-{\mathit{OC}}_-{\iso}
\\
\mathit{SH}^*(M \setminus D;\bF_p) \otimes \mathit{HH}_{*-n}(\scrA) 
\ar[u]_-{\mathit{CO} \otimes \mathit{id}}^-{\iso}
\ar[d]^-{\mathit{id} \otimes \mathit{OC}}_-{\iso}
\\
\mathit{SH}^*(M \setminus D;\bF_p) \otimes \mathit{SH}^*(M \setminus D;\bF_p) \ar[rrr]_-{\text{product}}
&&& \mathit{SH}^*(M \setminus D;\bF_p)
}
\end{equation}
Restricting the bottom $\rightarrow$ to $\mathit{SH}^*(M \setminus D;\bF_p) \otimes 1$ yields the identity, by definition. Hence, restricting the top $\rightarrow$ to $\mathit{HH}^*(\scrA) \otimes \mathit{OC}^{-1}(1)$ is an isomorphism $\mathit{HH}^*(\scrA) \rightarrow \mathit{HH}_*(\scrA)$. One now considers the corresponding $q$-deformed map $\mathit{HH}^*(\scrA_q) \rightarrow \mathit{HH}_*(\scrA_q)$; a $q$-filtration argument shows that this is an isomorphism as well.
\end{proof}

\begin{lemma} \label{th:cy}
If a class in $\mathit{HH}^*(\scrA_q)$ has the property that its action on $\mathit{HH}_*(\scrA_q)$ is zero, then the class itself is zero.
\end{lemma}

\begin{proof}
This is an immediate consequence of Lemma \ref{th:module}.
\end{proof}

\begin{lemma} \label{th:compatibility-oc}
(i) Under \eqref{eq:oc}, the action of $[\partial_q\mu_{\scrA_q}]$ on Hochschild homology corresponds to the endomorphism $a_{q}$.

(ii) Under \eqref{eq:cyclic-oc}, the Getzler-Gauss-Manin connection on cyclic homology corresponds to the Floer-theoretic connection on equivariant symplectic cohomology.
\end{lemma}

\begin{proof}
(ii) is proved in \cite[Section 6.3]{pomerleano-seidel23}, through the construction of a suitable chain homotopy (a parallel result, in a slightly different geometric context, is proved in \cite{ganatra-sheridan25}). Now, the $u = 0$ part of the chain map underlying the connection on $\mathit{SH}^*_{u,q}(M,D)$ defines $a_{q}$; and similarly, the $u = 0$ part of the chain level Getzler-Gauss-Manin connection yields the action of $[\partial_q\mu_{\scrA_q}]$. Hence, the chain homotopy argument for (ii), when truncated to $u = 0$, yields (i).
\end{proof}

\begin{corollary} \label{th:satisfied}
Suppose that \eqref{eq:no-torsion} holds. Then $\scrA$ and its deformation $\scrA_q$ satisfy the assumptions (i)--(vi) of Corollary \ref{th:fontaine-laffaille-2}, where $f$ is taken to be \eqref{eq:minimal-polynomial}. 
\end{corollary}

\begin{proof}
(i) is true by construction. The rest is as in \cite[Lemma 7.1.24]{pomerleano-seidel23}, except for (iv) which we will now discuss. By Lemma \ref{th:cy}, it is sufficient to show that $q^B f([\partial_q\mu_{\scrA_q}])$ acts trivially on $\mathit{HH}_*(\scrA_q)$. In view of Lemma \ref{th:compatibility-oc}, this is equivalent to saying that $q^B f(a_{q})$ acts trivially on $\mathit{SH}^*_q(M,D;\bF_p)$; which is (a weaker version of) Corollary \ref{th:torsion-1b}.
\end{proof}

We carry out corresponding algebraic manipulations (passing to negative powers of $q$; inverting $f$; and in the equivariant case, inverting $u$) on the chain complexes on both sides of \eqref{eq:oc} and \eqref{eq:cyclic-oc}. The outcome, using the notation from Section \ref{sec:algebra} on the categorical side, are isomorphisms
\begin{equation}
\begin{aligned}
\label{eq:negative-oc}
& H^*(q^{-1}A_{q^{-1},1/f}) \iso q^{-1}\mathit{SH}^{*+n}_{q^{-1},1/f}(M,D;\bF_p), \\
& H^*(q^{-1}A_{u^{\pm 1},q^{-1},1/f}) \iso q^{-1}\mathit{SH}^{*+n}_{u^{\pm 1},q^{-1},1/f}(M,D;\bF_p).
\end{aligned}
\end{equation}

%
%
%

\subsection{Conclusion\label{sec:the-end}}
We now add the final ingredient, namely Theorem \ref{th:fontaine-laffaille-2}.

\begin{proof}[Proof of Theorem \ref{th:q-fontaine-laffaille}]
We combine \eqref{eq:negative-oc} with Lemmas \ref{th:les} and \ref{th:les-2} to get
\begin{equation}
\begin{aligned}
\label{eq:weird-oc}
& H^*(q^{-1}A_{q^{-1},1/f}) \iso \mathit{SH}^{*+n+1}_{q,1/f}(M,D;\bF_p), \\
& H^*(q^{-1}A_{u^{\pm 1},q^{-1},1/f}) \iso \mathit{SH}^{*+n+1}_{u^{\pm 1},q,1/f}(M,D;\bF_p).
\end{aligned}
\end{equation}
Applying Corollary \ref{th:fontaine-laffaille-2} to the deformed Fukaya category (which is possible by Lemma \ref{th:satisfied}) yields a map between the groups on the left hand side of \eqref{eq:weird-oc}. We use these isomorphism to transfer that to a map between the groups on the right hand side. The desired properties of this map follow from the compatibility properties of \eqref{eq:weird-oc} inherited from Lemma \ref{th:compatibility-oc}.
\end{proof}

\begin{corollary} \label{th:q-p-curvature}
In the situation of Theorem \ref{th:q-fontaine-laffaille}, the $u$-linear extension of $\Phi$ satisfies
\begin{equation} \label{eq:q-p-curvature}
q^p \Phi(x) = \Phi(q^p a_q^{p-1}(x) - u^{p-1} qx) = \Phi\big(q(c_1(M)^{\ast_q (p-1)} \bullet_q \cdot - u^{p-1})(x)\big).
\end{equation}
\end{corollary}

\begin{proof}
Iterating \eqref{eq:q-fontaine-laffaille-2}, one finds that for any $k \geq 1$, 
\begin{equation} \label{eq:qk-phi}
q^k \Phi(x) = \sum_{i=0}^{k-1} c_{k,i}\, u^i \nabla_{u\partial_q}^{k(p-1)-ip}
\Phi(q^{k-i} x),
\end{equation}
where the coefficients $c_{k,i} \in \bF_p$ satisfy $c_{1,0} = 1$ and the recursive relation
\begin{equation}
c_{k+1,i} = c_{k,i} + kc_{k,i-1} \quad \text{(with $c_{k,-1}$ and $c_{k,k}$ set to $0$).}
\end{equation}
We can package these as coefficients of polynomials, and solve the recursion:
\begin{equation}
g_k(z) = \sum_{i=0}^{k-1} c_{k,i} z^i = (1+z)(1+2z)\cdots(1+(k-1)z)
\end{equation}
For $k = p$, we have $g_p(z) \in \bF_p[z]$ of degree $(p-1)$ which vanishes at $z=1,1/2\dots,1/(p-1)$, or equivalently at $z = 1,\dots,p-1$, and with constant term $1$; which must be $g_p(z) = 1-z^{p-1}$. Plugging that back into \eqref{eq:qk-phi} and using \eqref{eq:q-fontaine-laffaille-1} leads to the first expression in \eqref{eq:q-p-curvature}:
\begin{equation}
q^p \Phi(x) = \nabla^{p(p-1)}_{u\partial_q}\Phi(q^p x) - u^{p-1} \Phi(qx) =
\Phi(q^p a_q^{p-1}(x)) - u^{p-1}\Phi(qx).
\end{equation}
The second one follows by applying Lemma \ref{th:kappa-and-quantum-product}
\end{proof}

\begin{proof}[Proof of Corollary \ref{th:torsion-3}]
We start by working with $\bF_p$-coefficients. Corollary \ref{th:q-p-curvature} has the general form $q^p \Phi(x) = \Phi(q N_q x)$, where the endomorphism $N_q$ is $q$-linear. We can iterate this and get $q^{np} \Phi(x) = \Phi(q^n N_q^nx) = \Phi(N_q^n q^nx)$. Corollary \ref{th:torsion-1b} shows that the right hand side is zero. In words, $q^{np}$ acts trivially on the image of $\Phi$. Because $q^{np}$ commutes with $\nabla_{u\partial_q}$, we then conclude from \eqref{eq:q-fontaine-laffaille-3} that 
\begin{equation} \label{eq:q-np}
q^{np}x = 0\;\; \text{for all } x \in \mathit{SH}^*_{u^{\pm 1},q,1/f}(M,D:\bF_p). 
\end{equation}
Since inverting $u$ and $f(\nabla_{u\partial_q})$ are mutually commuting operations, we can write
\begin{equation}
\mathit{SH}^*_{u^{\pm 1},q,1/f}(M,D;\bF_p) = \bF_p[\nabla_{u\partial_q},1/f(\nabla_{u\partial_q})] \otimes_{\bF_p[\nabla_{u\partial_q}]} \mathit{SH}^*_{u^{\pm 1},q}(M,D;\bF_p),
\end{equation}
where $\mathit{SH}^*_{u^{\pm 1},q}(M,D;\bF_p) = \bF_p[u^{\pm 1}] \otimes_{\bF_p[u]} \mathit{SH}^*_{u,q}(M,D;\bF_p)$. Therefore, \eqref{eq:q-np} says that 
\begin{equation} \label{eq:pointwise2}
\text{for each } x \in \mathit{SH}^*_{u^{\pm 1},q}(M,D:\bF_p), \text{ there is a $C>0$ such that } q^{np} f(\nabla_{u\partial_q})^C x = 0.  
\end{equation}
By \eqref{eq:deformed-sh-2} there is no $u$-torsion, so the same holds in $\mathit{SH}^*_{u,q}(M,D;\bF_p)$. 
One can take $C = pB$ and write $f(\nabla_{u\partial_q})^C = f(\nabla_{u\partial_q}^p)^B$.
The result as stated then follows by applying the long exact coefficient sequence \eqref{eq:sh-coefficient}.
\end{proof}

\section{The projective line\label{sec:p1}}

Here, we illustrate how some of the results in the body of the paper work out in the simplest example of $M = \bC P^1$. For the quantum connection this is done by elementary computations, while for the relative version we appeal to homological mirror symmetry.

\subsection{The absolute theory}
In the standard basis $(1,\mathit{PD}([\mathit{point}]))$ of $H^*(M;\bZ) = \bZ^2$, the map \eqref{eq:quantum-product-c1} is
\begin{equation}
q^{-1}c_1(M) \ast_q \cdot = \begin{pmatrix} 0 & 2q \\ 2q^{-1} & 0 \end{pmatrix}.
\end{equation}
When reduced mod $2$ this vanishes, so the quantum connection becomes trivial. We will therefore consider the connection with $\bF_p$-coefficients for odd $p$. The equation
\begin{equation} \label{eq:eigenvectors}
\nabla_{\partial_q} x^{\pm} = \pm 2x^{\pm} 
\end{equation}
has the (mod $p$) solution
\begin{equation} \label{eq:xpm}
\begin{aligned}
& x^{\pm} = (x^{\pm}_0,x^{\pm}_1), \\
& x_0^{\pm} = \sum_{k=2}^{(p+1)/2} \frac{1}{(k-1)!} \begin{pmatrix} 2k-2 \\ k-2 \end{pmatrix} (\pm q)^k = q^2 + \cdots + ((p-1)/2)!\, (\pm q)^{(p+1)/2}, \\
& x_1^{\pm} = -\!\!\!\!\sum^{(p-1)/2}_{k=0} \frac{1}{k!} \begin{pmatrix} 2k \\ k \end{pmatrix} (\pm q)^k
= -1 + \cdots + ((p-1)/2)!\, (\pm q)^{(p-1)/2}.
\end{aligned}
\end{equation}
Directly from \eqref{eq:eigenvectors} one sees that
\begin{equation}
\partial_q \mathrm{det}\begin{pmatrix} x_0^+ & x_0^- \\ x_1^+ & x_1^- \end{pmatrix} = 0.
\end{equation}
From this and looking at the powers of $q$ that can appear, it follows that the determinant must be a constant times $q^p$. That constant can be computed from the highest order terms in \eqref{eq:xpm} (and a bit of elementary number theory):
\begin{equation}
\mathrm{det}\begin{pmatrix} x_0^+ & x_0^- \\ x_1^+ & x_1^- \end{pmatrix} =
-2q^p.
\end{equation}
Since the determinant is invertible in $\bF_p[q^{\pm 1}]$, it follows that the vectors $x^{\pm}$ form a basis for $H^*(M;\bF_p)[q^{\pm 1}]$. In that basis, the quantum connection will split in the way predicted by Corollary \ref{th:trivial}. (Writing the connection in this basis also shows that for the minimal polynomial $f(z) = z^2-4$ we have $f(\nabla_{\partial_q}^p) = 0$, see Corollary \ref{th:congruence}.)

Concerning quantum Steenrod operations, the computation in \cite[Example 1.6]{seidel-wilkins21} adjusted to our notation, is that 
\begin{equation} \label{eq:qsigma-p1}
\begin{aligned}
&
(Q\Sigma_{q^{-1}c_1(M)})_{u=1} = \begin{pmatrix} \sigma_{11} & \sigma_{12} \\ 
\sigma_{21} & \sigma_{22} \end{pmatrix}, \\
& \sigma_{11} = -\sigma_{22} = 2\sum_{k=1}^{(p-1)/2} \frac{1}{k!\,(k-1)!}
\begin{pmatrix} 2k-1 \\ k-1 \end{pmatrix} q^{2k-p} = 2q^{2-p} + \cdots + q^{-1}/2, \\
& \sigma_{12} = 2\sum_{k=2}^{(p+1)/2} \frac{1}{k!\, (k-2)!} \begin{pmatrix} 2k-2 \\ k-1 \end{pmatrix} q^{2k-p} = 2q^{4-p} + \cdots + 2q, \\
& \sigma_{21} = -2\sum_{k=0}^{(p-1)/2} \frac{1}{(k!)^2} \begin{pmatrix} 2k \\ k \end{pmatrix} q^{2k-p} = -2q^{-p} + \cdots + 2 q^{-1}.
\end{aligned}
\end{equation}
In principle, one could try to use these formulae directly to check that, following Proposition \ref{th:jae-conjecture},
\begin{equation} \label{eq:jae-example}
(Q\Sigma_{q^{-1}c_1(M)})_{u=1} x^{\pm} = \nabla_{\partial_q}^p x^{\pm} = (\pm 2)^p x^{\pm} =
\pm 2 x^{\pm};
\end{equation}
but one can also bypass most of that computation, as follows. First, we use the highest order terms in \eqref{eq:xpm} and \eqref{eq:qsigma-p1} to see that
\begin{equation} \label{eq:highest}
\begin{aligned}
(Q\Sigma_{q^{-1}c_1(M)})_{u=1} x^{\pm} & = \begin{pmatrix} \half q^{-1} + \cdots 
& 2q + \cdots \\
2q^{-1} + \cdots & -\half q^{-1} + \cdots \end{pmatrix}
\begin{pmatrix}
((p-1)/2)!) (\pm q)^{(p+1)/2} + \cdots \\
((p-1)/2)!) (\pm q)^{(p-1)/2} + \cdots 
\end{pmatrix}
\\ & = \pm 2 \begin{pmatrix}
((p-1)/2)!) (\pm q)^{(p+1)/2} + \cdots \\
((p-1)/2)!) (\pm q)^{(p-1)/2} + \cdots 
\end{pmatrix}
\end{aligned}
\end{equation}
where $\cdots$ denotes lower powers of $q$. Secondly, $Q\Sigma_{q^{-1}c_1(M)}$ commutes with the connection by Property \ref{th:quantum-properties}(iv), and therefore 
\begin{equation} \label{eq:eigenvector-2}
\nabla_{\partial_q} (Q\Sigma_{q^{-1}c_1(M)})_{u=1} x^{\pm} = \pm 2 (Q\Sigma_{q^{-1}c_1(M)})_{u=1} x^{\pm}.
\end{equation}
The space of solutions of each equation \eqref{eq:eigenvectors} is a free rank $1$ module over $\bF_p[q^{\pm p}]$. From \eqref{eq:eigenvector-2} we know that $(Q\Sigma_{q^{-1}c_1(M)})_{u=1} x^{\pm}$ lies in that module. From this, \eqref{eq:highest}, and a consideration about the lowest power of $q$ that can appear in $(Q\Sigma_{q^{-1}c_1(M)})_{u=1} x^{\pm}$, one can derive that \eqref{eq:jae-example} holds.

\subsection{The relative theory}
Take $(M,D) = (\bC P^1, 2 \text{ points})$. We will use Lemma \ref{th:oc} to interpret $\mathit{SH}^*_q(M,D;\bF_p)$ as the Hochschild homology of the wrapped Fukaya category of $M \setminus D = \bC^*$, deformed by the Maurer-Cartan element which is determined by the compactification. Up to quasi-equivalence, the wrapped category is described by the algebra
\begin{equation}
\scrA \htp \bF_p[y,y^{-1}],
\end{equation}
the ring of functions on the punctured line $Y = \{y \neq 0\}$ over $\bF_p$; and the deformation to $\scrA_q$ comes from the superpotential (central element) $W = y+y^{-1}$. This is a Landau-Ginzburg model, and its Hochschild homology is
\begin{equation} \label{eq:hkr-1}
\mathit{HH}_{*-1}(\scrA_q) \iso H^{*-1}(\Omega^{-*}_Y[q], q\,dW \wedge \cdot);
\end{equation}
and on the right hand side, 
\begin{equation} \label{eq:example-kappaq}
a_{q}(\theta) = W\theta.
\end{equation}
Determining the cohomology of \eqref{eq:hkr-1} is elementary. Having done that, one then inverts 
\begin{equation}
f(a_{q}) = f(W) = (y^2+y^{-2}-2) = y^{-2}(y-1)^2(y+1)^2.
\end{equation}
We skip the computation, and only record the outcome. Let $Y^{\mathrm{reg}} = Y \setminus \{y = \pm 1\}$ be the space obtained by removing the critical points of $W$ from $Y$.

\begin{lemma} \label{th:p1}
We have
\begin{equation}
\mathit{SH}^*_{q,1/f}(M,D:\bF_p) \iso \Omega^1_{Y^{\mathrm{reg}}} = \bF_p[y^{\pm 1},(y^2-1)^{-1}] \mathit{dy},
\end{equation}
placed in degree $0$. This carries the trivial action of $q$ (obvious for degree reasons); and $a_{q}$ is still given by \eqref{eq:example-kappaq}.
\end{lemma}

\begin{remark} \label{th:landau-ginzburg}
The isomorphism \eqref{eq:hkr-1} is related to, but somewhat simpler than, the computation of the Hochschild homology of matrix factorizations. There, much of the difficulty is showing that the category of matrix factorizations has enough objects. In our case, $\mathit{HH}_*(\scrA_q)$ is defined via a (complete) $q$-deformation of the Hochschild complex of $\scrA$. Therefore, it is sufficient to look at the map between the underlying complexes, as in \cite{caldararu-tu, segal}. The fact that it's a quasi-isomorphism property can then be proved via $q$-filtrations, reducing it to the classical Hochschild-Kostant-Rosenberg theorem; in the literature, this is usually stated in characteristic $0$, but it also works in characteristic $p > \mathrm{dim}(Y)$, which is certainly the case here. Of course, the example under consideration here is so simple (only one-forms are relevant) that all computations could also be carried out in an ad hoc fashion.
\end{remark}

A similar argument applies to the equivariant theory, where one has 
\begin{equation} \label{eq:hkr-2}
\mathit{HC}_{*-1}(\scrA_q) \iso H^{*-1}(\Omega^{-*}_Y[u,q], ud + q\,dW \wedge \cdot)
\end{equation}
with the connection
\begin{equation} \label{eq:example-gauss-manin}
\nabla_{u\partial_q}\theta = (u\partial_q + W)\theta.
\end{equation}
Again, we only record the outcome of applying this isomorphism in our case:

\begin{lemma} \label{th:p1-2}
\begin{equation}
\mathit{SH}^*_{u,q,1/f}(M,D;\bF_p) \iso \Omega^1_{Y^{\mathrm{reg}}}[u],
\end{equation}
where the $q$-action is
\begin{equation} \label{eq:q-action}
q \, \theta = -u\, d(\theta/dW),
\end{equation}
and the connection is 
\begin{equation}
\nabla_{u\partial_q}\theta = W\theta.
\end{equation}
\end{lemma}
%

\begin{remark}
Continuing the discussion from Remark \ref{th:landau-ginzburg}, \eqref{eq:hkr-2} is related to, but simpler than, known results about the cyclic homology of matrix factorizations \cite{efimov,shklyarov}. 
Again, if one wanted to avoid any theory, our example is simple enough to be amenable to direct computation in the cyclic complex.
\end{remark}

One can write \eqref{eq:q-action} as $q y^k (\mathit{dy/y}) = (u(k+1) y^{k+1} + qy^{k+2}) (dy/y)$, and iterate that to get
\begin{equation} \label{eq:iterated-relation}
q^m y^k (dy/y) = \sum_{i=0}^m \begin{pmatrix} m \\ i \end{pmatrix} 
(k+i+1)(k+i+2) \cdots (k+m)
q^i u^{m-i} y^{k+m+i} (dy/y).
\end{equation}
Setting $m = p$, one finds that $q^p y^k(dy/y) = q^p y^{k+2p} (dy/y)$. Since both $y$ and $y^{-p} - y^p = (y-y^{-1})^p$ are invertible functions on $Y^{\mathrm{reg}}$, it follows that
\begin{equation} \label{eq:q-p-is-zero}
q^p = 0 \text{ as an endomorphism of $\mathit{SH}^*_{u,q,1/f}(M,D;\bF_p)$,}
\end{equation}
as predicted by Corollary \ref{th:torsion-3}. 

We analyze the situation a little further, in the simpler context where $u$ has been inverted. $K = \mathit{ker}(q) \subset \mathit{SH}^*_{u^{\pm 1},q,1/f}(M,D;\bF_p)$ is the $\bF_p[u,u^{-1}]$-module generated by elements $g(y^p)\, dW$. Multiplication by $q$ induces isomorphisms
\begin{equation}
q: \nabla_{u\partial_q}^k K \stackrel{\iso}{\longrightarrow} \nabla_{u\partial_q}^{k-1} K
\quad \text{for $k = 1,\dots,p-1$.}
\end{equation}
From that and \eqref{eq:q-p-is-zero}, it follows that $\bF_p[u^{\pm 1}] \otimes_{\bF[u]} \mathit{SH}^*_{S^1,q,1/f}(M,D;\bF_p)$ splits into the direct sum of the $\nabla_{u\partial_q}^k K$, and that $K$ itself must be the image of the Fontaine-Laffaille map \eqref{eq:q-fontaine-laffaille}. We will not try to determine the map (since that would require one to dig through all of \cite{pomerleano-seidel23, petrov-vaintrob-vologodsky17}), but there is a reasonable guess compatible with the properties \eqref{eq:q-fontaine-laffaille-1}, \eqref{eq:q-fontaine-laffaille-2}, namely
\begin{equation}
\Phi(g(y) dW) \stackrel{?}{=} g(y^p) dW.
\end{equation}

\begin{remark} \label{th:speculation}
With Remark \ref{th:p-m} in mind, it is worth taking a speculative look at the mod $p^m$ situation. The computations from Lemmas \ref{th:p1} and \ref{th:p1-2} carry over. Spelling out \eqref{eq:q-fontaine-laffaille-m}, we would like to see a map
\begin{equation}
\Phi: (\bZ/p^m)[y^{\pm 1}, (y^2-1)^{-1}] \mathit{dy} \longrightarrow
(\bZ/p^m)[y^{\pm 1}, (y^2-1)^{-1}] \mathit{dy}
\end{equation}
which satisfies
\begin{align}
\label{eq:fl-scary}
& W^p \Phi(\theta) = \Phi( W\theta), \\
& d(\Phi(\theta)/dW) = pW^{p-1}\, \Phi(d(\theta/dW)) \label{eq:fl-scary-2};
\end{align}
the $p$ in \eqref{eq:fl-scary-2} arises because in all constructions involving $\overline{SH}^*$ we multiply the $u$ variable by $p$, resulting in a modified version of \eqref{eq:q-action}. One can find a change of coordinates $y \mapsto \phi(y) = y^p + p\psi(y)$ over $\bZ/p^m$, such that $\phi^*W = W^p$: namely,
\begin{equation}
\begin{aligned}
& \phi(y) + \frac{1}{\phi(y)} = (y+y^{-1})^p \\ \Leftrightarrow \;\; &
(y^p-y^{-p})\psi(y) = 
\sum_{k=1}^{p-1} \frac1p \begin{pmatrix} p \\ k \end{pmatrix} y^{2k}
- y^{-p} \big(py^{-p} \psi(y)^2 - p^2 y^{-2p} \psi(y)^3 + \cdots\big);
\end{aligned}
\end{equation}
bearing in mind that $y^p-y^{-p}$ is invertible, this can be solved iteratively with respect to $m$, \`a la Hensel's Lemma. Then, the following possible formula satisfies \eqref{eq:fl-scary}, \eqref{eq:fl-scary-2}:
\begin{equation}
\Phi(g(y)\,dW) \stackrel{?}{=} g(\phi(y)) \, dW.
\end{equation}
\end{remark}

\section{Isolated hypersurface singularities\label{sec:appendix}}

Using the results from Section \ref{sec:p-curvature}, we will give a characteristic $p$ proof of Varchenko's theorem \cite{varchenko81} about the monodromy of isolated hypersurface singularities, in the context of twisted de Rham complexes (Theorem \ref{th:varchenko}). The argument builds on \cite[Appendix C]{chen24}; but the main step (Proposition \ref{th:explicit-cartier}) is new and of independent interest, since it sheds light on the role played by the Cartier isomorphism.
%
%

\subsection{Preliminaries}
Let $F$ be any field. Take some $W \in F[[x_1,\dots,x_n]]$, and its Jacobian ring
\begin{equation}
\mathit{Jac}(W) = F[[x_1,\dots,x_n]]/(\partial_{x_1}W,\dots,\partial_{x_n}W).
\end{equation}
From now on, we assume that $W(0) = 0$, and that $W$ has an isolated critical point at the origin, which means that $\mathit{Jac}(W)$ is finite-dimensional over $F$. As a consequence,
\begin{equation} \label{eq:high-power}
\text{there is some $N \geq 1$ such that } (x_1,\dots,x_n)^N \subset (\partial_{x_1}W,\dots,\partial_{x_n}W).
\end{equation}
In particular, multiplication with $W$ acts nilpotently on $\mathit{Jac}(W)$.

Take the spaces of $k$-forms,
\begin{equation}
\Omega^k \stackrel{\mathrm{def}}{=} \bigoplus_{i_1<\cdots<i_k} F[[x_1,\dots,x_n]] \mathit{dx}_{i_1} \wedge \cdots \wedge \mathit{dx}_{i_k}.
\end{equation}
Introduce a formal variable $u$, and consider the twisted de Rham complex 
\begin{equation} \label{eq:twisted-de-rham}
(\Omega^*((u)), ud-dW),
\end{equation}
where $(ud-dW)\theta = u\,d\theta - dW \wedge \theta$.

\begin{lemma} \label{th:u-filtration}
$H^k(\Omega^*((u)), ud-dW)$ is zero in degrees $k<n$. Moreover, the $F((u))$-dimension of $H^n(\Omega^*((u)), ud-dW)$ equals the $F$-dimension of $\mathit{Jac}(W)$.
\end{lemma}

\begin{proof}
(This is familiar.) Start with the simpler complex $(\Omega^*,-dW)$. By assumption, the $\partial_{x_i}W$ form a regular sequence in $F[[x_1,\dots,x_n]]$. Since $(\Omega^*,-dW)$ can be thought of as the associated Koszul complex, its cohomology is zero in degrees $<n$. We have
\begin{equation}
-dW \wedge (dx_1 \wedge \cdots \wedge dx_{i-1} \wedge dx_{i+1} \wedge \cdots \wedge dx_n) = 
(-1)^i (\partial_{x_i}W) \mathit{vol},
\end{equation}
where $\mathit{vol} = \mathit{dx}_1 \wedge \cdots \wedge \mathit{dx}_n$, and therefore an isomorphism
\begin{equation}
\mathit{Jac}(W) \stackrel{\iso}{\longrightarrow} H^n(\Omega^*,-dW), \quad
[f] \longmapsto [f\,\mathit{vol}].
\end{equation}
Next consider $(\Omega^*[[u]], ud-dW)$. An argument involving the previous computation and the $u$-filtration shows the following. First, $H^k(\Omega^*[[u]],ud-dW) = 0$ for $k<n$. Second, $H^n(\Omega^*[[u]],ud-dw)$ is a free $F[[u]]$-module, and the natural map $H^n(\Omega^*[[u]],ud-dW)/u \rightarrow H^n(\Omega^*,-dW)$ is an isomorphism. The desired result is obtained by inverting $u$.
\end{proof}

Assume that $F$ is of characteristic $\neq 2$. Take the connection
\begin{equation} \label{eq:u-connection}
\nabla_{\partial_u}: \Omega^*((u)) \longrightarrow \Omega^*((u)), \quad
\nabla_{\partial_u} = \partial_u + u^{-2}W - \textstyle\frac{u^{-1}}{2} \mathit{Gr}
\end{equation}
where $\mathit{Gr}$ multiplies $k$-forms by $k$. This satisfies
\begin{equation}
\nabla_{\partial_u} (ud-dW) - (ud-dW) \nabla_{\partial_u} = \textstyle \frac{u^{-1}}{2}(ud-dW),
\end{equation}
and therefore induces a connection on cohomology, for which we use the same notation,
\begin{equation} \label{eq:u-connection-2}
\nabla_{\partial_u}: H^n(\Omega^*((u)), ud-dW) \longrightarrow
H^n(\Omega^*((u)), ud-dW).
\end{equation}

\subsection{The characteristic $p$ argument}
We need the following elementary fact:

\begin{lemma} \label{th:scalar-p-curvature}
In an associative ring, take elements $\partial$ and $a$, and define $a^{(k)}$ inductively by $a^{(0)} = a$, $a^{(k)} = [\partial,a^{(k-1)}]$. 
Suppose that, for some prime $p$, we have that $a^{(0)},\dots,a^{(p-1)}$ mutually commute. Then
\begin{equation} \label{eq:scalar-p-curvature}
(\partial + a)^p = \partial^p + a^{(p-1)} + a^p \text{ mod p},
\end{equation}
where ``mod $p$'' means that the difference between the two sides is $p$ times some element.
\end{lemma}

\begin{proof} (See \cite[Theorem 1.3]{bavula09} or \cite[Theorem 3.12]{bostan-caruso-roques}.) By induction, one shows that
\begin{equation}
(\partial + a)^k =
\!\!\!\!
\sum_{\substack{0 \leq l \leq k \\ S_1,\dots,S_l}}
\frac{1}{l!}\, 
 a^{(|S_1|-1)} \cdots a^{(|S_l|-1)} \partial^{k-|S_1|-\cdots-|S_l|} \quad
 \text{for $k = 1,\dots,p$,}
\end{equation}
where the sum is over all choices of nonempty pairwise disjoint subsets $S_1,\dots,S_l \subset \{1,\dots,k\}$. Dividing by $l!$ corresponds to forgetting the ordering of those subsets. Let's therefore consider unordered collections of subsets, and the action of the cyclic group of order $k$ on those. If $k = p$ is prime, this action has exactly three fixed points, namely: not choosing any subset ($l = 0$); a single subset which is all of $\{1,\dots,p\}$ ($l = 1$); and the decomposition into singletons ($l = p$). These fixed points give rise to the three terms on the right in \eqref{eq:scalar-p-curvature}, and the rest yields a $p$-fold multiple.
\end{proof}

Let's return to our original situation, assuming that $F$ is a field of characteristic $p>2$. 

\begin{lemma} \label{th:spectral-sequence}
The twisted de Rham complex \eqref{eq:twisted-de-rham} is a complex of finitely generated free modules over $F[[x_1^p,\dots,x_n^p]]((u))$.
\end{lemma}

This is obvious (but important).

\begin{lemma} \label{th:u-p-curvature} 
The $p$-curvature of \eqref{eq:u-connection} is multiplication with $u^{-2p}W^p$.
\end{lemma}

\begin{proof} (See \cite[Equations (C.5) and (C.6)]{chen24c}.) Write the connection as $\partial + a$, where $\partial = \partial_u$ and $a = u^{-2}W - u^{-1}\mathit{gr/2}$, seen as elements of the algebra of $F$-linear endomorphisms of $\Omega^*((u))$. In the notation of Lemma \ref{th:scalar-p-curvature}, $a^{(k)} = \partial_u^k(a)$ are the derivatives, and these pairwise commute. One therefore gets 
\begin{equation}
\nabla_{\partial_u}^p = \textstyle \partial_u^p + (u^{-2}W - \frac{u^{-1}}{2} \mathit{Gr})^p
+ (p! u^{-p-1}W - \frac{(p-1)! u^{-p}}{2} \mathit{Gr}) = u^{-2p}W^p.
\end{equation}
\end{proof}

\begin{lemma}\label{lem:vanishing} 
The elements $(\partial_{x_i} W)^p \in F[[x_1^p,\dots,x_n^p]]$ act trivially on $H^*(\Omega^*((u)), ud-dW)$.  
\end{lemma}

\begin{proof} 
(See \cite[Equations (C.7) and (C.8)]{chen24c}.) Let $\iota_{\partial_{x_i}}$ denote contraction with the coordinate vector field $\partial_{x_i}$. The Cartan homotopy formula gives
\begin{equation}\label{eq:cartan-homotopy}
(ud-dW)\iota_{\partial_{x_i}}+\iota_{\partial_{x_i}}(ud-dW)
= u\partial_{x_i}-\partial_{x_{i}}W. 
\end{equation}
Write this as $\partial+a$, where $\partial = u\partial_{x_i}$ and $a = -\partial_{x_i}W$. Again in terms of Lemma \ref{th:scalar-p-curvature}, $a^{(k)} = u^k\partial_{x_i}^k(a) = -u^k\partial_{x_i}^{k+1}W$, and these mutually commute. Therefore
\begin{equation} \label{eq:p-th-power-of-l}
(u\partial_{x_i} - \partial_{x_i}W)^p = (u\partial_{x_i})^p 
- (\partial_{x_i}W)^p
- u^{p-1}\partial_{x_i}^pW = - (\partial_{x_i}W)^p.
\end{equation}
From \eqref{eq:cartan-homotopy} we know that $u\partial_{x_i} - \partial_{x_i}W$ is nullhomotopic; therefore, so is \eqref{eq:p-th-power-of-l}.
\end{proof}

Define
\begin{equation}
C^{-1}(\mathit{vol}) = d(x_1^p/p) \wedge \cdots \wedge d(x_n^p/p) = (x_1^{p-1} \cdots x_n^{p-1})
\, \mathit{dx}_1 \wedge \cdots \wedge \mathit{dx}_n,
\end{equation}
where the notation follows \cite{cartier57}. By Lemma \ref{lem:vanishing}, there is a well-defined map
\begin{equation} \label{eq:explicit-cartier}
\frac{F[[x_1^p,\dots,x_n^p]]((u))}{((\partial_{x_{1}}W)^p,\dots,(\partial_{x_{n}}W)^p)}
\longrightarrow H^n\bigl(\Omega^*((u)), ud-dW \bigr), \quad
[f] \longmapsto [f\, C^{-1}(\mathit{vol})].
\end{equation}

\begin{proposition} \label{th:explicit-cartier}
The map \eqref{eq:explicit-cartier} is an isomorphism.
\end{proposition}

We postpone the proof for now.

\begin{remark}
Unlike the computation from Lemma \ref{th:spectral-sequence}, there is no version of Proposition \ref{th:explicit-cartier} without inverting $u$, as the following (simplest) example shows. Take $n = 1$ and $W = x^2/2$. In $H^1(\Omega^*[[u]],ud-dW)$, one has
\begin{equation}
[x^{m+1} \mathit{dx}]=m u[x^{m-1}\mathit{dx}]
\end{equation}
and consequently,
\begin{equation}
[x^{p-1}dx]
=(p-2)(p-4)\cdots 1\,
 u^{(p-1)/2}[dx].
\end{equation}
The class $[dx]$ is not divisible by $u$, since its reduction modulo
$u$ is the nonzero generator of $\mathit{Jac}(W)\,dx$ (see the proof of Lemma \ref{th:u-filtration}).  Therefore, $[x^{p-1}dx]$ does not generate $H^1(\Omega^*[[u]],ud-dW)$ over $F[[u]]$; it does once we pass to $F((u))$.
\end{remark}

\begin{corollary} \label{th:w-conjugacy}
The $p$-curvature of \eqref{eq:u-connection-2} belongs to the same nilpotent conjugacy class as the action of $W$ on $\mathit{Jac}(W)$ (extended $u$-linearly).
\end{corollary}

\begin{proof} 
As a first step, take the Frobenius $\mathit{Frob}: F[[x_1,\dots,x_n]] \rightarrow F[[x_1^p,\dots,x_n^p]]$. The induced map of Jacobian rings fits into a diagram of $F$-vector spaces
\begin{equation} \label{eq:nilpotent-frobenius}
\xymatrix{
\mathit{Jac}(W)^{(1)} 
\ar[r]^-{\mathit{Frob}}_-{\iso} \ar[d]_-{W^{(1)}} & 
\frac{F[[x_1^p,\dots,x_n^p]]}{((\partial_{x_1}W)^p,\dots,(\partial_{x_n}W)^p)} \ar[d]^-{W^p}
\\
\mathit{Jac}(W)^{(1)} \ar[r]^-{\mathit{Frob}}_-{\iso} & 
\frac{F[[x_1^p,\dots,x_n^p]]}{((\partial_{x_1}W)^p,\dots,(\partial_{x_n}W)^p)}
}
\end{equation}
Here, $\mathit{Jac}(W)^{(1)} = \mathit{Jac}(W) \otimes_{\mathit{Frob}} F$ and $W^{(1)} = W \otimes_{\mathit{Frob}} \mathit{id}_F$. In the notation from Section \ref{subsec:nilpotent-2}, the nilpotent endomorphism $W$ of $\mathit{Jac}(W)$ belongs to the conjugacy class $\scrO_\pi(F)$ for some partition $\pi$. The left $\downarrow$ in \eqref{eq:nilpotent-frobenius} belongs to the same class (think in terms of Jordan normal forms, or in terms of the nullspaces of its iterates), and therefore, so does the right $\downarrow$.

The second step is to extend the last-mentioned map $u$-linearly, and appeal to Proposition \ref{th:explicit-cartier}, which gives a diagram
\begin{equation} \label{eq:explicit-cartier-diagram}
\xymatrix{
\frac{F[[x_1^p,\dots,x_n^p]]((u))}{((\partial_{x_1}W)^p,\dots,(\partial_{x_n}W)^p)}
\ar[r]^-{\eqref{eq:explicit-cartier}} \ar[d]_{W^p} &
H^n(\Omega^*((u)),ud-dW) \ar[d]^{W^p}
\\
\frac{F[[x_1^p,\dots,x_n^p]]((u))}{((\partial_{x_1}W)^p,\dots,(\partial_{x_n}W)^p)}
\ar[r]^-{\eqref{eq:explicit-cartier}} &
H^n(\Omega^*((u)),ud-dW).
}
\end{equation}
The $p$-curvature is $u^{-2p}W^p$, by Lemma \ref{th:u-p-curvature}. From \eqref{eq:explicit-cartier-diagram} we know that this is nilpotent, and multiplying it by the invertible element $u^{2p}$ does not change its conjugacy class. It follows that the $p$-curvature lies in $\scrO_\pi(F((u)))$.
\end{proof}

\subsection{Proof of Proposition \ref{th:explicit-cartier}}
Write $F[[x_1,\dots,x_n]]((u))^\wedge$ and $F[[x_1^p,\dots,x_n^p]]((u))^\wedge$ for the completions with respect to the ideal $(x_1,\dots,x_n)$ respectively $(x_1^p,\dots,x_n^p)$. In both cases, elements of the completion have the form
\begin{equation} \label{eq:completed-f}
\sum_{j=-\infty}^{\infty} f_j(x_1,\dots,x_n) u^j
\end{equation}
such that as $j \rightarrow -\infty$, the $f_j$ vanish to higher and higher order. (These spaces could equivalently be written as $F((u))[[x_1,\dots,x_n]]$, meaning power series in $x_1,\dots,x_n$ with coefficients in $F((u))$, respectively $F((u))[[x_1^p,\dots,x_n^p]]$; but that is maybe less intuitive from our viewpoint.) The same completion can be applied to differential forms, yielding $\Omega^*((u))^\wedge$, which inherits the differential $ud-dW$.

\begin{lemma} \label{th:completion}
(i) The inclusion into the completion induces isomorphisms
\begin{align} \label{eq:include-into-completion-1}
& \frac{F[[x_1,\dots,x_n]]((u))}{(\partial_{x_1}W,\dots,\partial_{x_n}W)}
\longrightarrow
\frac{F[[x_1,\dots,x_n]]((u))^\wedge}{(\partial_{x_1}W,\dots,\partial_{x_n}W)},
\\ \label{eq:include-into-completion-2}
& \frac{F[[x_1^p,\dots,x_n^p]]((u))}{((\partial_{x_1}W)^p,\dots,(\partial_{x_n}W)^p)}
\longrightarrow
\frac{F[[x_1^p,\dots,x_n^p]]((u))^\wedge}{((\partial_{x_1}W)^p,\dots,(\partial_{x_n}W)^p)}.
\end{align}

(ii) Similarly, the following map is an isomorphism:
\begin{equation} \label{eq:include-into-completion-3}
H^*(\Omega^*((u)), ud-dW) \longrightarrow H^*(\Omega^*((u))^\wedge, ud-dW).
\end{equation}
\end{lemma}

\begin{proof}
(i) Recall from \eqref{eq:high-power} that any power series which vanishes to sufficiently high order lies in the ideal $(\partial_{x_1}W,\dots,\partial_{x_n}W)$. More precisely: there is some $N$ such that if $f \in F[[x_1,\dots,x_n]]$ vanishes to order $k_1+\cdots+k_n \geq N$, then it can be written as $g_1 \partial_{x_1}W + \cdots + g_n \partial_{x_n}W$, where $g_1,\dots,g_n$ vanish to order $k_1+\cdots+k_n-N$. This means that, if in a series \eqref{eq:completed-f} each term vanishes to order $N$, then that series lies in $(\partial_{x_1}W,\dots,\partial_{x_n}W) \subset F[[x_1,\dots,x_n]]((u))^\wedge$. Modulo that ideal, we can therefore truncate any series to one with only finitely powers of $u^{-1}$, showing that \eqref{eq:include-into-completion-1} is onto.

Next, suppose that $g \in F[[x_1,\dots,x_n]]((u))$ maps to zero under \eqref{eq:include-into-completion-1}, hence can be written as $g_1 \partial_{x_1}W + \cdots + g_n\partial_{x_n}W$ for $g_1,\dots,g_n \in F[[x_1,\dots,x_n]]((u))^\wedge$. One can truncate the $g_i$ to some sufficiently high order in $u^{-1}$, such that the truncations $g_i^{\mathrm{trunc}} \in F[[x_1,\dots,x_n]]((u))$ have the property that $g - g_1^{\mathrm{trunc}}\partial_{x_1}W - \cdots - g_n^{\mathrm{trunc}}\partial_{x_n}W$ vanishes to order $N$. This shows that $g$ is zero in $F[[x_1,\dots,x_n]]((u))/(\partial_{x_1}W,\dots,\partial_{x_n}W)$, so the map \eqref{eq:include-into-completion-1} is injective. 

By taking $p$-th powers, one sees that any sufficiently high degree monomial in the $x_i^p$ belongs to $((\partial_{x_1}W)^p,\dots,(\partial_{x_n}W)^p) \subset F[[x_1^p,\dots,x_n^p]]$. Then, the same argument as before shows that \eqref{eq:include-into-completion-2} is an isomorphism.

(ii) Because $F[[x_1,\dots,x_n]]((u))$ is finite free over the Noetherian ring $F[[x_1^p,\dots,x_n^p]]((u))$, the completion is exact:
\begin{equation}\label{eq:completed-cohomology}
H^*\bigl(\Omega^*((u))^\wedge,ud-dW\bigr) \iso H^*\bigl(\Omega^*((u)),ud-dW)^\wedge.
\end{equation}
We have shown (Lemma \ref{lem:vanishing}) that $(\partial_{x_i}W)^p$ acts trivially on $H^*(\Omega^*((u)),ud-dW)$ cohomology. Hence, see (i), the same holds for any high degree monomial in the $x_i$, which means that completion does nothing on the right hand side of \eqref{eq:completed-cohomology}.
\end{proof}

\begin{remark}
The fact that \eqref{eq:include-into-completion-1} is an isomorphism holds in any characteristic. 
In contrast, \eqref{eq:include-into-completion-3} is not an isomorphism if the field $F$ has characteristic $0$: we would have
\begin{equation}
e^{W/u} \in F[[x_1,\dots,x_n]]((u))^\wedge, \;\; (ud-dW)e^{W/u} = 0 \;\; \Longrightarrow 
H^0(\Omega^*((u))^\wedge,ud-dW) \neq 0,
\end{equation}
in contrast with Lemma \ref{th:u-filtration}. This may remind one of the isomorphism between periodic and co-periodic cyclic homology \cite{kaledin}, which also only holds in positive characteristic.
\end{remark}

It will be useful for our computation to take a different point of view on $F[[x_1,\dots,x_n]]((u))^\wedge$. For that, consider
\begin{equation}
D_i = \partial_{x_i}-u^{-1}\partial_{x_{i}}W: F[[x_1,\dots,x_n]]((u))^\wedge
\longrightarrow F[[x_1,\dots,x_n]]((u))^\wedge.
\end{equation}
The $D_i$ are pairwise commuting $F[[x_1^p,\dots,x_n^p]]((u))^\wedge$-linear operators, and they satisfy (see the proof of Lemma \ref{lem:vanishing})
\begin{equation} \label{eq:d-power}
D_i^p = -u^{-p}(\partial_{x_{i}}W)^p.
\end{equation}
Introduce the auxiliary algebra 
\begin{equation}\label{eq:auxiliary-algebra}
A = \frac{(F[[x_1^p,\dots,x_n^p]]((u))^\wedge)[T_1,\dots,T_n]}
{(T_1^p+u^{-p}(\partial_{x_{1}}W)^p,\dots,T_n^p+u^{-p}(\partial_{x_{n}}W)^p)}.
\end{equation}
Because of \eqref{eq:d-power}, we have a well-defined $F[[x_1^p,\dots,x_n^p]]((u))^\wedge$-linear map
\begin{equation}\label{eq:cyclic-vector-map}
\begin{aligned}
& \Phi: A \longrightarrow F[[x_1,\dots,x_n]]((u))^\wedge, \\
& \Phi(T_1^{k_1}\cdots T_n^{k_n}) = D_1^{k_1}\cdots D_n^{k_n}(x_1^{p-1}\cdots x_n^{p-1}).
\end{aligned}
\end{equation}
The spaces on both sides of \eqref{eq:cyclic-vector-map} are free modules of rank $p^n$ over $F[[x_1^p,\dots,x_n^p]]((u))^\wedge$. 

\begin{lemma}
$\Phi$ is an isomorphism.
\end{lemma}

\begin{proof}
Take $k_1,\dots,k_n \in \{0,\dots,p-1\}$. We have
\begin{equation} \label{eq:triangular-term}
\begin{aligned}
& \Phi(T_1^{k_1}\cdots T_n^{k_n}) =
D_1^{k_1} \cdots D_n^{k_n} (x_1^{p-1}\cdots x_n^{p-1}) = \prod_{i=1}^n
(p-1)\cdots (p-k_i) x_i^{p-1-k_i} 
 \\ & \qquad \qquad + \text{\em (monomials in $(x_1,\dots,x_n)$ of higher degree, with $F[u^{-1}]$-coefficients).}
\end{aligned}
\end{equation}
The term we wrote down is obtained by choosing the derivative part of $D_i$ at every stage; if we choose the multiplication part instead, the degree becomes higher by at least two, since $\partial_{x_i}W$ vanishes at the origin.
Since the coefficient in \eqref{eq:triangular-term} is nonzero in $F$, the elements $\Phi(T_1^{k_1}\cdots T_n^{k_n})$ yield a basis of
$F[[x_1,\dots,x_n]]((u))^\wedge/(x_1^p,\dots,x_n^p)$. The reduction of $\Phi$ modulo that ideal is therefore an isomorphism. Since $F[[x_1^p,\dots,x_n^p]]((u))^\wedge$ is local and both sides of \eqref{eq:cyclic-vector-map} are finite free of the same rank,
Nakayama's Lemma implies that $\Phi$ is an isomorphism.
\end{proof}

We now return to our main topic. In view of Lemma \ref{th:completion}, Proposition \ref{th:explicit-cartier} is equivalent to the following:

\begin{proposition} \label{th:explicit-cartier-2}
The map
\begin{equation} \label{eq:explicit-cartier-2}
\frac{F[[x_1^p,\dots,x_n^p]]((u))^\wedge}{((\partial_{x_{1}}W)^p,\dots,(\partial_{x_{n}}W)^p)}
\longrightarrow H^n\bigl(\Omega^*((u))^\wedge, ud-dW \bigr), \quad
[f] \longmapsto [f\, C^{-1}(\mathit{vol})]
\end{equation}
is an isomorphism.
\end{proposition}

\begin{proof}
We have 
\begin{equation}
\begin{aligned}
& \frac{F[[x_1^p,\dots,x_n^p]]((u))^\wedge}{
((\partial_{x_{1}}W)^p,\dots,(\partial_{x_{n}}W)^p)} \iso \frac{A}{(T_1,\dots,T_n)} \\
& \qquad \iso \frac{F[[x_1,\dots,x_n]]((u))^\wedge}{\mathit{im}(\mathit{D_1}) + \cdots + \mathit{im}(\mathit{D_n})} 
\iso H^n\bigl(\Omega^*((u))^\wedge, ud-dW).
\end{aligned}
\end{equation}
The first isomorphism is by definition \eqref{eq:auxiliary-algebra}; the second one is induced by \eqref{eq:cyclic-vector-map}; and the third is by definition of $ud-dW$ and the $D_i$. Under the combined isomorphism, the class $[1]$ is sent to $[C^{-1}(\mathit{vol})]$. Hence, that isomorphism is precisely \eqref{eq:explicit-cartier-2}.
\end{proof}

\subsection{Spreading out}
Our aim is the following version of the result from \cite{varchenko81}:

\begin{theorem} \label{th:varchenko}
Take $W \in \bC[[x_1,\dots,x_n]]$ with an isolated critical point at the origin. 
Let $\pi$ be the partition such that multiplication by $W$, acting on $\mathit{Jac}(W)$, lies in the nilpotent conjugacy class $\scrO_\pi(\bC)$. Then the connection \eqref{eq:u-connection-2} is equivalent over $\bC((u))$ to $\partial_u + u^{-1}A$, where the nilpotent part of $A$ lies in $\bar\scrO_\pi(\bC)$.
\end{theorem}

There is a formal change of variables which transforms $W$ into a polynomial \cite[Section 6.3]{agv1}. Since the statement of Theorem \ref{th:varchenko} is preserved by such a coordinate change, we assume without loss of generality that $W$ is already polynomial.
We will use versions of differential forms with coefficients in $\bC$; in rings
\begin{equation} \label{eq:w-ring}
R \subset \bC \text{ is finitely generated, and contains $\half$ as well as the coefficients of $W$};
\end{equation}
and in finite fields $F = R/\frakm$ (see Section \ref{subsec:nilpotent}). To distinguish between those versions, the underlying ring or field will be used as a subscript.

\begin{lemma} \label{th:r}
 There is a ring \eqref{eq:w-ring} such that: 
 
 (i) $\mathit{Jac}_R(W) \iso H^n(\Omega^*_R, -dW)$ is a free $R$-module; and the map
\begin{equation} \label{eq:tensor-with-c-1}
H^n(\Omega^*_R,-dW) \otimes_R \bC \longrightarrow H^n(\Omega^*_{\bC},-dW)
\end{equation}
is an isomorphism. 

(ii) The lower degree cohomology groups vanish:  \begin{align}\label{eq:KoszulzeroR}
H^j(\Omega^*_R,-dW)=0,
\qquad j < n.
\end{align}

\end{lemma}
\begin{proof} (i) Start with any $R$ as in \eqref{eq:w-ring}, and let $K=\operatorname{Frac}(R)$ be its fraction field. Because the singularity is isolated, the comparison map
\begin{align} 
\frac{K[x_1,\dots,x_n]_{(x_1,\dots,x_n)}}{(\partial_{x_1}W,\dots,\partial_{x_n}W)} \longrightarrow \frac{K[[x_1,\dots,x_n]]}{(\partial_{x_1}W,\dots,\partial_{x_n}W)}
\end{align}
is an isomorphism. 
For some $N$,
\begin{equation}
(x_1,\dots,x_n)^N K[x_1,\dots,x_n]_{(x_1,\dots,x_n)}
\subset
(\partial_1W,\dots,\partial_nW)\cdot K[x_1,\dots,x_n]_{(x_1,\dots,x_n)}.
\end{equation} 
This means for any monomial $x_1^{k_{1}}\cdots x_n^{k_{n}}$ of degree $N$, we have 
\begin{equation}
x_1^{k_{1}}\cdots x_n^{k_{n}}=\sum h_{k,i} \cdot \partial_{x_i}W, \;\; h_{k,i} \in  K[x_1,\cdots,x_n]_{(x_1,\cdots,x_n)}.
\end{equation}
We can localize $R$ to a ring $R^{\operatorname{fin}}$ so that $h_{k,i} \in R^{\operatorname{fin}}[[x_1,\cdots,x_n]]$. Then 
\begin{align}\label{eq:spreadingoutsupport}
(x_1,\dots,x_n)^N \cdot R^{\operatorname{fin}}[[x_1,\dots,x_n]]
\subset
(\partial_1W,\dots,\partial_nW) \cdot R^{\operatorname{fin}}[[x_1,\dots,x_n]]. 
\end{align}
Having done this, it follows that $Jac_{R^{\operatorname{fin}}}(W)$ is generated as an $R^{\operatorname{fin}}$-module by the finitely many monomials of total degree less than $N$. Since a finite module is generically free \cite[Tag~051S]{stacks}, we may localize $R^{\operatorname{fin}}$ further to a ring $R^{\operatorname{free}}$ so that 
\begin{equation}
\frac{R^{\operatorname{fin}}[[x_1,\dots,x_n]]}{(\partial_{x_{1}}W,\dots,\partial_{x_{n}}W)} \otimes_{R^{\operatorname{fin}}} R^{\operatorname{free}} 
\end{equation}
is a finite free module. In view of \eqref{eq:spreadingoutsupport}, we have 
\begin{equation}
\begin{aligned}
& \frac{R^{\operatorname{fin}}[[x_1,\dots,x_n]]}
     {(\partial_{x_1}W,\dots,\partial_{x_n}W)}
\otimes_{R^{\operatorname{fin}}} R^{\operatorname{free}}
\cong
\frac{R^{\operatorname{fin}}[[x_1,\dots,x_n]]}
     {\bigl((x_1,\dots,x_n)^N,
     \partial_{x_1}W,\dots,\partial_{x_n}W\bigr)}
\otimes_{R^{\operatorname{fin}}} R^{\operatorname{free}}
\\
& \qquad \qquad \cong
\frac{R^{\operatorname{free}}[[x_1,\dots,x_n]]}
     {\bigl((x_1,\dots,x_n)^N,
     \partial_{x_1}W,\dots,\partial_{x_n}W\bigr)}
\cong
\mathit{Jac}_{R^{\operatorname{free}}}(W),
\end{aligned}
\end{equation}
which gives the desired freeness. 

In our construction, we've seen that $\mathit{Jac}_{R^{\operatorname{free}}}(W)$ is a quotient of $R^{\operatorname{free}}[[x_1,\dots,x_n]]/(x_1,\dots,x_n)^N$, and since that is a finite $R^{\operatorname{free}}$-module, tensoring with $\bC$ yields $\bC[[x_1,\dots,x_n]]/(x_1,\dots,x_n)^N$; one can then further quotient by the $\partial_{x_i}W$ to show that \eqref{eq:tensor-with-c-1} is an isomorphism. The proof of (ii) is similar.  \end{proof}

\begin{lemma} \label{th:rspectral} For $R$ as in Lemma \ref{th:r}: 

(i) $H^n(\Omega_R^*((u)), ud-dW)$ is a free $R((u))$-module. Moreover, the map 
\begin{equation} 
H^n(\Omega^*_R((u)), ud-dW) \otimes_{R((u))} \bC((u)) \longrightarrow H^n(\Omega^*_{\bC}((u)),ud-dW)
\end{equation}
is an isomorphism.

(ii) The lower degree cohomology groups vanish:  \begin{align}\label{eq:KoszulzeroR2}
H^j(\Omega^*_R((u)),ud-dW)=0
\qquad j < n.
\end{align}
\end{lemma}

\begin{proof}

The $u$-filtration argument from the proof of Lemma \ref{th:u-filtration} works with $R$-coefficients as well. In particular, the vanishing \eqref{eq:KoszulzeroR2} holds and  $H^n(\Omega^*_R[[u]],ud-dW)$ is a free $R[[u]]$-module, whose $u = 0$ reduction is $H^n(\Omega^*_R,-dW)$. The outcome is that
\begin{equation} \label{eq:u-map-c}
H^n(\Omega^*_R[[u]], ud-dW) \otimes_{R[[u]]} \bC[[u]] \longrightarrow H^n(\Omega^*_{\bC}[[u]],ud-dW)
\end{equation}
is a map of finite rank free $\bC[[u]]$-modules, whose $u=0$ reduction is an isomorphism. Therefore \eqref{eq:u-map-c} must be an isomorphism, and we can then invert $u$.
\end{proof}

\begin{lemma} \label{th:r-2}
(i) For $R$ as in Lemma \ref{th:r}, and any residue field $F = R/\frakm$, the map
\begin{equation} 
\label{eq:r-to-f-1}
H^n(\Omega^*_R,-dW) \otimes_R F \longrightarrow H^n(\Omega^*_F,-dW)
\end{equation}
is an isomorphism.

(ii) The same applies to
\begin{equation}
\label{eq:r-to-f-2}
H^n(\Omega^*_R((u)),ud-dW) \otimes_{R((u))} F((u)) \longrightarrow H^n(\Omega^*_F((u)),ud-dW).
\end{equation}
\end{lemma}

\begin{proof}
For (i), note that $R[[x_1,\dots,x_n]] \otimes_R F = F[[x_1,\dots,x_n]]$, and similarly $\Omega^*_R \otimes_R F = \Omega^*_F$. Passage to cohomology is straightforward using the hypotheses from Lemma \ref{th:r}. The argument for (ii) is parallel.
\end{proof}

Note that if some $R$ has the properties from Lemma \ref{th:r}, then those continue to hold for any larger ring in the class \eqref{eq:w-ring}.

\begin{proof}[Proof of Theorem \ref{th:varchenko}]
Let $R$ be as in Lemma \ref{th:r}. There is some matrix that transforms the nilpotent endomorphism $W$ of $\mathit{Jac}_{\bC}(W) = \mathit{Jac}_R(W) \otimes_R \bC$ into Jordan normal form. We may enlarge $R$ so that this transformation and its inverse are defined over $R$, and therefore can be reduced mod $\frakm$. From this and Lemma \ref{th:r-2}(i), it follows that for any $F = R/\frakm$,
\begin{equation}
W: \mathit{Jac}_F(W) \rightarrow \mathit{Jac}_F(W) \;\; \text{belongs to $\scrO_\pi(F)$.}
\end{equation}
Corollary \ref{th:w-conjugacy} and Lemma \ref{th:r-2}(ii) then imply the following:
\begin{equation}
\parbox{37em}{
The connection $\nabla_{\partial_u}$ on $H^n(\Omega^*_R((u)), ud-dW)$, when reduced mod any $\frakm$, has $p$-curvature in $\scrO_\pi(F((u)))$. 
}
\end{equation}
We now apply Proposition \ref{th:orbit-closure} and Lemma \ref{th:rspectral}, and reach the conclusion that the connection on $H^n(\Omega^*_{\bC}((u)),ud-dW)$ has the desired form.
\end{proof}

\end{document}